\documentclass[11pt]{amsart}
\makeatletter
\def\l@section{\@tocline{1}{0pt}{1pc}{}{}}

\def\@cite#1#2{{\m@th\upshape\bfseries[{#1\if@tempswa{\m@th\upshape\mdseries, #2}\fi}]}}
\makeatother

\usepackage{graphicx}
\usepackage{amsmath}
\usepackage{amscd}
\usepackage{amsfonts}
\usepackage{amssymb}
\usepackage{color}
\usepackage{mathtools}
\usepackage{MnSymbol}
\usepackage{enumitem}
\usepackage[
hypertexnames=false, colorlinks, citecolor=red,linkcolor=blue, urlcolor=red]{hyperref}

\usepackage{tikz}
\usetikzlibrary{calc,decorations.markings}

\newtheorem{theorem}{Theorem}[section]
\newtheorem{lemma}[theorem]{Lemma}
\newtheorem{proposition}[theorem]{Proposition}

\newtheorem{corollary}[theorem]{Corollary}

\theoremstyle{definition}
\newtheorem{definition}[theorem]{Definition}
\newtheorem{example}[theorem]{Example}
\newtheorem{remark}[theorem]{Remark}

\newtheorem{question}[theorem]{Question}
\newtheorem{assumption}{Assumption}

\mathtoolsset{centercolon}

\newcommand{\bA}{{\mathbb{A}}}

\newcommand{\bC}{{\mathbb{C}}}
\newcommand{\bD}{{\mathbb{D}}}

\newcommand{\bF}{{\mathbb{F}}}

\newcommand{\bN}{{\mathbb{N}}}

\newcommand{\bT}{{\mathbb{T}}}

\newcommand{\bZ}{{\mathbb{Z}}}

\newcommand{\cA}{\mathcal{A}}
\newcommand{\cB}{\mathcal{B}}

\newcommand{\cH}{\mathcal{H}}
\newcommand{\cI}{\mathcal{I}}

\newcommand{\cK}{\mathcal{K}}
\newcommand{\cL}{\mathcal{L}}
\newcommand{\cM}{\mathcal{M}}
\newcommand{\cN}{\mathcal{N}}
\newcommand{\cO}{\mathcal{O}}

\newcommand{\cT}{\mathcal{T}}

\newcommand{\cV}{\mathcal{V}}
\newcommand{\cW}{\mathcal{W}}
\newcommand{\cX}{\mathcal{X}}

\newcommand{\fA}{{\mathfrak{A}}}

\newcommand{\fC}{{\mathfrak{C}}}

\newcommand{\fF}{{\mathfrak{F}}}

\newcommand{\fJ}{{\mathfrak{J}}}

\newcommand{\fL}{{\mathfrak{L}}}
\newcommand{\fM}{{\mathfrak{M}}}
\newcommand{\fN}{{\mathfrak{N}}}

\newcommand{\fR}{{\mathfrak{R}}}
\newcommand{\fS}{{\mathfrak{S}}}
\newcommand{\fT}{{\mathfrak{T}}}

\newcommand{\rA}{\mathrm{A}}

\newcommand{\rL}{\mathrm{L}}

\newcommand{\rN}{\mathrm{N}}

\newcommand{\rS}{\mathrm{S}}
\newcommand{\rT}{\mathrm{T}}
\newcommand{\rU}{\mathrm{U}}

\newcommand{\ep}{\varepsilon}
\renewcommand{\phi}{\varphi}
\newcommand{\upchi}{{\raise.35ex\hbox{$\chi$}}}

\newcommand{\AND}{\text{ and }}
\newcommand{\IF}{\text{ if }}
\newcommand{\FOR}{\text{ for }}

\newcommand{\qand}{\quad\text{and}\quad}
\newcommand{\qif}{\quad\text{if}\quad}

\newcommand{\qfor}{\quad\text{for}\quad}
\newcommand{\qforal}{\quad\text{for all}\quad}

\renewcommand{\ae}{\operatorname{a.\!e.}}
\newcommand{\bFGv}{\bF_{G,v}^+}

\newcommand{\ca}{\mathrm{C}^*}
\newcommand{\dirlim}{\varinjlim}

\newcommand{\Fd}{\bF_d^+}

\newcommand{\la}{\langle}
\newcommand{\ra}{\rangle}
\newcommand{\ip}[1]{\la #1 \ra}
\newcommand{\bip}[1]{\big\la #1 \big\ra}

\newcommand{\linf}{{\ell^\infty}}
\newcommand{\mt}{\varnothing}
\newcommand{\ol}{\overline}

\newcommand{\sot}{\textsc{sot}}
\newcommand{\wot}{\textsc{wot}}
\newcommand{\wotclos}[1]{\ol{#1}^{\textsc{wot}}}

\DeclareMathOperator*{\wotsum}{\textsc{wot--}\!\!\sum}
\DeclareMathOperator*{\sotsum}{\textsc{sot--}\!\!\sum}
\DeclareMathOperator*{\wotlim}{\textsc{wot}--lim}

\newcommand{\Ad}{\operatorname{Ad}}
\newcommand{\Alg}{\operatorname{Alg}}
\newcommand{\Lat}{\operatorname{Lat}}
\newcommand{\Dim}{\operatorname{dim}}

\newcommand{\degin}{\operatorname{deg}_{\textup{in}}}
\newcommand{\degout}{\operatorname{deg}_{\textup{out}}}
\newcommand{\id}{\operatorname{id}}

\newcommand{\ran}{\operatorname{Ran}}
\newcommand{\rank}{\operatorname{rank}}

\newcommand{\spn}{\operatorname{span}}
\newcommand{\supp}{\operatorname{supp}}

\DeclareMathOperator{\rad}{Rad}

\begin{document}

\title[Structure of Free semigroupoid algebras]{Structure of free semigroupoid algebras}

\author[K.R. Davidson]{Kenneth R. Davidson}
\address{Pure Mathematics Department, University of Waterloo,
Waterloo, ON, Canada}
\email{krdavids@uwaterloo.ca  \vspace{-2ex}}

\author[A. Dor-On]{Adam Dor-On}
\address{Department of Mathematics, University Illinois at Urbana-Champaign, Urbana, IL, USA}
\email{adoron@illinois.edu \vspace{-2ex}}

\author[B. Li]{Boyu Li}
\address{Mathematics and Statistics Department, University of Victoria,
Victoria, BC, Canada}
\email{boyuli@uvic.ca}

\begin{abstract}
A free semigroupoid algebra is the \wot-closure of the algebra generated by a TCK family of a graph. We obtain a structure theory for these algebras analogous to that of free semigroup algebra. We clarify the role of absolute continuity and wandering vectors. These results are applied to obtain a Lebesgue-von Neumann-Wold decomposition of TCK families, along with reflexivity, a Kaplansky density theorem and classification for free semigroupoid algebras. Several classes of examples are discussed and developed, including self-adjoint examples and a classification of atomic free semigroupoid algebras up to unitary equivalence.\vspace*{-5mm}
\end{abstract}

\subjclass[2010]{Primary: 47L80, 47L55, 47L40. Secondary: 46L05.}
\keywords{graph algebra, free semigroupoid algebra, road coloring theorem, absolute continuity, wandering vectors, Lebesgue decomposition, operator algebras, Cuntz-Krieger algebras}

\thanks{The first author was partially supported by a grant from NSERC.\\
\strut\indent The second author was partially supported by an Ontario Trillium Scholarship}

\maketitle
\tableofcontents
\thispagestyle{empty}

\section{Introduction}

By the work of Glimm \cite{Gli61} (see also \cite{Dix77}) we know that for any non-type I C*-algebra, there is no countable family of Borel functions which distinguishes its irreducible representations up to unitary equivalence. 
In short, determining unitary equivalence of representations of non-type I C*-algebras is an intractable question in general.

Two of the simplest examples of a non-type I C*-algebra are the Toeplitz-Cuntz and Cuntz algebras $\cT_n$ and $\cO_n$ respectively. 
There are several good reasons to want to distinguish their representations up to unitary equivalence. 
One of these lies in the work of Bratteli and Jorgensen \cite{BJ97, BJ99, BJP96} where various families of representations of $\cO_n$ are considered, and are used to generate wavelet bases on self-similar sets, as well as to investigate endomorphisms of $B(\cH)$. 
The theme in these results is to restrict to subclasses of representations, so as to parametrize them up to unitary equivalence. 
This has been done successfully by Davidson and Pitts \cite{DP1999} for atomic representations, by Davidson, Kribs and Shpigel \cite{DKS} for finitely correlated representations, and by Dutkay and Jorgensen \cite{DJ14} for monic representations.

Extending these results to subclasses of representations of classical Cuntz-Krieger algebras has become important, as can be seen in the work of Marcolli and Paolucci \cite{MP11} and the work of Bezuglyi and Jorgensen \cite{BJ15}. 
In these works, representations of Cuntz-Krieger algebras associated to semibranching function systems are investigated, where monic representations are classified up to unitary equivalence, wavelets are constructed and Hausdorff dimension is computed for self-similar sets. 
In fact, in recent years this has been pushed further by Farsi, Gillaspy, Kang and Packer \cite{FGKP16} and together with Jorgensen \cite{FGJKP17a} to representations of higher-rank graphs, where previous results were extended and expanded. 
For instance, building on work in \cite{aHLRS15} and \cite{FGJKP17b}, KMS states were realized on higher-rank graph C*-algebras, and connections with Hausdorff dimensions of associated spaces were made.

Here we address the problem of unitary equivalence, without restricting to a subclass, is by weakening the invariant. 
For representations of $\cT_n$, this is done through the work on free semigroup algebras. 
Every representation of $\cT_n$ gives rise to a free semigroup algebra by taking the weak-$*$ closed algebra generated by the row-isometry that defines the representation. 
When two representations are unitarily equivalent, their free semigroup algebras are completely isometrically and weak-$*$ homeomorphically isomorphic. 
Hence, one tries to determine isomorphism classes of free semigroup algebras as dual operator algebras instead.

The theory of free semigroup algebras originates from the work of Popescu \cite{Pop1989a, Pop1989b, Pop1991, Pop1995}, and Popescu and Arias \cite{AP1995} where the non-commutative analytic Toeplitz algebra was introduced, and used to establish a robust dilation theory for row contractions. 
Popescu's non-commutative analytic Toeplitz algebra $\fL_n$ is then just the free semigroup algebra generated by creation operators on full Fock space $\fF_n$ on $n$ symbols.
In \cite{AP1995}, Popescu and Arias describe the invariant subspaces and prove reflexivity of Popescu's non-commutative analytic Toeplitz algebra $\fL_n$ in $B(\fF_n)$. 
In a sequence of papers, Davidson and Pitts \cite{DP1999, DP1998a, DP1998b} introduce and study properties of Popescu's non-commutative analytic Toeplitz algebra. 
They establish its hyperreflexivity, define free semigroup algebras and classify atomic free semigroup algebras. 

One turning point in the theory was the structure theorem of Davidson, Katsoulis and Pitts \cite{DKS}. 
This structure theorem allows for a tractable $2\times 2$ block structure where the $(1,1)$ corner is a von Neumann algebra, the $(1,2)$ corner is $0$ and the $(2,2)$ corner is analytic in the sense that it is completely isometrically and weak-$*$ homeomorphically isomorphic to $\fL_n$. 
This allowed Davidson, Katsoulis and Pitts to lift many results known for $\fL_n$ to general free semigroup algebras.

A serious roadblock in the theory, identified in \cite{DKS}, was the existence of wandering vectors for analytic free semigroup algebras. In \cite{DKS,DLP} it was shown that many spatial results that hold for $\fL_n$ will hold for general free semigroup algebras if this roadblock was removed. 
This wandering vectors roadblock was eventually removed in the work of Kennedy \cite{Kennedy2011}, leading to a very satisfactory structure and decomposition theory for free semigroup algebras \cite{Kennedy2013, FK2013}. 
These structure and decomposition results can then be applied to distinguish irreducible representations of $\cO_n$. 
For instance, we may classify them into analytic, von Neumann and dilation types, all of which are preserved under unitary equivalence. 
Representations obtained from dilating finite dimensional row contractions, known as finitely correlated representations, were classified by Davidson, Kribs and Shpigel \cite{DKS}. 
They show that unitary equivalence of such representations reduces to unitary equivalence for associated (reduced) row contractions on a finite dimensional space, a problem in matrix theory.

For a directed graph $G = (V,E)$ with countably many edges $E$ and vertices $V$, each edge $e\in E$ has a source $s(e)$ and range $r(e)$ in $V$. We denote by $\bF_G^+$ the collection of all directed paths of finite length in $G$, called the path space of $G$, or the \emph{free semigroupoid of $G$}. Vertices are considered as paths of length $0$.

Let $\cH_G : = \ell^2(\bF_G^+)$ be the Hilbert space with an orthonormal basis $\{ \xi_{\lambda} :\lambda \in \bF_G^+  \}$. For each vertex $v \in V$ and edge $e\in E$ we may define projections and partial isometries given by
\[
L_v(\xi_{\mu}) = \begin{cases} 
\xi_{\mu} & \text{if } r(\mu) = v \\ 
0 & \text{if } r(\mu) \neq v
\end{cases} \ \ \text{and} \ \
L_e(\xi_{\mu}) = \begin{cases} 
\xi_{e \mu} & \text{if } r(\mu) = s(e) \\ 
0 & \text{if } r(\mu) \neq s(e)
\end{cases}
\]
and
\[
R_v(\xi_{\mu}) = \begin{cases} 
\xi_{\mu} & \text{if } s(\mu) = v \\ 
0 & \text{if } s(\mu) \neq v
\end{cases} \ \ \text{and} \ \
R_e(\xi_{\mu}) = \begin{cases} 
\xi_{\mu e} & \text{if } s(\mu) = r(e) \\ 
0 & \text{if } s(\mu) \neq r(e)
\end{cases}
\]
The \emph{left} and \emph{right} free semigroupoid algebras given by $G$ are
\[
\fL_G := \wotclos{\Alg} \{ L_v, L_e : v\in V, e\in E \}
\]
and
\[
\fR_G := \wotclos{\Alg} \{ R_v, R_e : v\in V, e\in E \}.
\]

$\fL_G$ can be thought of as the analogue of the non-commutative analytic Toeplitz algebra for arbitrary directed graph, where $\fL_n$ corresponds to the graph on a single vertex with $n$ loops. 
In \cite{KP2004, JK2005, JP2006, KK2005} many of the results of Popescu and Arias, and of Davidson and Pitts on $\fL_n$ were extended and expanded to $\fL_G$ for arbitrary graphs. 
While in these papers $\fL_G$ is called a free semigroupoid algebra, we will expand that definition here so that it extends the original definition of free semigroup algebras. We recommend \cite{Raeburn} for more on Toeplitz-Cuntz-Krieger families of operators, their associated operator algebras and uniqueness theorems.

The family $\rL=(L_v,L_e)$ is an example of a family satisfying the Toeplitz-Cuntz-Krieger relations. More precisely, we say that a family $\rS= (S_v, S_e) :=\{S_v, S_e : v\in V,\ e \in E\}$ of operators on a Hilbert space $\cH$ is a Toeplitz-Cuntz-Krieger (TCK) family if 
\begin{enumerate}[leftmargin=15mm]
\item[(P)] $\{S_v : v \in V\}$ is a set of pairwise orthogonal projections;
 \item[(IS)] $S_e^*S_e=S_{s(e)}$, for every $e\in E$ ;
 \item[(TCK)] $\sum_{e\in F} S_e S_e^* \leq S_{v}$ for every finite subset $F \subseteq r^{-1}(v)$.
\end{enumerate}
The universal C*-algebra $\cT(G)$ generated by families $\rS$ satisfying the above relations is called the \emph{Toeplitz-Cuntz-Krieger} algebra.
If a Toeplitz-Cuntz-Krieger family $\rS=(S_v, S_e)$ satisfies the additional relation
\begin{enumerate}[leftmargin=15mm]
\item[(CK)]
$\sum_{r(e)=v} S_e S_e^* = S_v$, for every $v \in V$ with $0 < |r^{-1}(v)| < \infty$, 
\end{enumerate}
then we say that $\rS$ a Cuntz-Krieger (CK) family.
The universal such family generates the well-known graph C*-algebra $\cO(G)$ associated to $G$, which is a quotient of $\cT(G)$. Since we are dealing with \wot-closed algebras, following Muhly and Solel \cite{MS1999} we will say that a CK family $\rS$ is a \textit{fully coisometric} if it satisfies the additional property
\begin{enumerate}[leftmargin=15mm]
\item[(F)]
$\sotsum_{r(e)=v} S_e S_e^* = S_v$ for every $v \in V$.
\end{enumerate}
We will also say that a TCK family $\rS$ is \textit{non-degenerate} if
\begin{enumerate}[leftmargin=15mm]
\item[(ND)]
$\sotsum_{v\in V} S_v = I$. 
\end{enumerate}

It basically follows from the universal property of $\cT(G)$ that $*$-rep\-re\-sent\-ations of $\cT(G)$ are in bijective correspondence with TCK families, where one identifies a representation by the image of the canonical generators. 
As such, one can talk interchangeably about TCK families and $*$-representations of $\cT(G)$. The representations which factor through the quotient to $\cO(G)$ are precisely those associated to CK families. Throughout the paper, we will denote the representation of $\cT(G)$ associated to a TCK family $\rS$ by $\pi_{\rS}$.

By a combination of a Wold type decomposition theorem (such as Theorem \ref{thm:Wold-decomp} or \cite[Theorem 2.7]{JK2005}) and the gauge invariant uniqueness theorem for $\cO(G)$, it follows that $\cT(G)$ is $*$-isomorphic to $C^*(\rL)$ via a map that sends generators to generators. Hence, we will always identify these algebras.

\begin{definition}
Let $G$ be a countable directed graph, and let $\rS=(S_v, S_e)$ be a non-degenerate TCK family for $G$ acting on a Hilbert space $\cH$. The \wot~closed algebra $\fS$ generated by $\rS$ is called a \emph{free semigroupoid algebra} of $G$.
\end{definition}

An important class of examples for free semigroupoid algebras for a graph $G$ are finitely correlated algebras that arise from the work of \cite{Fuller2011}. Recall that when $\cA \subseteq B(\cH)$ is an operator algebra, a subspace $\cV \subseteq \cH$ is \emph{invariant} for $\cA$ if for each $T\in \cA$ we have $T(\cV) \subseteq \cV$, and that $\cV$ is coinvariant for $\cA$ if $\cV$ is invariant for $\cA^*$. 

A free semigroupoid algebra $\fS\subseteq\cB(\cH)$ is called \emph{finitely correlated} if there exists a \emph{finite dimensional} coinvariant subspace $\cV\subseteq\cH$ which is cyclic for $\fS$. 
Finitely correlated representations of free semigroup algebras were first studied in \cite{DKS} as an important class of examples, and results on them were extended to representations of product systems of C*-correspondences over $\bN^k$ by Fuller \cite{Fuller2011} (See also \cite{BDZ06} in this context). 

The following is obtained as a special case of \cite[Theorem 2.27]{Fuller2011}, and some methods from \cite[Theorem 4.9]{Fuller2011} when the C*-correspondence arises from a directed graph. The key observation here is that finiteness of $G$ and the strong double-cycle property assumption in \cite[Theorem 4.5]{Fuller2011} can be removed by using \cite[Theorem 3.1]{JK2005} instead, so that \cite[Theorem 4.9]{Fuller2011} holds without the assumption of the strong double-cycle property. We denote by $\rL_{G,v}$ the restriction of each operator in $\rL$ to the reducing subspace $\cH_{G,v}:= \{ \xi_{\lambda} :\lambda \in \bF_G^+, \ s(\lambda) =v  \}$ of basis vectors associated to paths emanating from $v$. 

We will say that a free semigroupoid algebra is \emph{left-regular type} if up to unitary equivalence it is generated by a TCK family of the form $\oplus_{v\in V}\rL_{G,v}^{(\alpha_v)}$ for some multiplicities $\alpha_v$ (some of which could be zero).

\begin{theorem}[Fuller]\label{T:fuller}
Let $\fS$ be the free semigroupoid algebra for a fully coisometric CK family $\rS$ for a graph $G$. Assume that $\fS$ is finitely correlated with finite dimensional cyclic coinvariant subspace $\cV$. Then
\begin{enumerate} [label=\normalfont{(\arabic*)}]

\item There is a unique \emph{minimal} cyclic coinvariant subspace $\hat\cV \subseteq \cV$ given by the span of all minimal coinvariant subspaces for $\rS$.

\item $\fS|_{\hat\cV^\perp}$ is a left-regular type free semigroupoid algebras.

\item $\hat\fA = P_{\hat\cV} \fS|_{\hat\cV}$ is a finite dimensional C*-algebra, whose Wedderburn decomposition is given by $\hat\fA \simeq \bigoplus_i \cM_{d_i}$ on $ \hat\cV \simeq \bigoplus_i \cV_i$ where each $\cV_i$ is a minimal coinvariant subspace of dimension $d_i$.

\item $\fS[\cV_i]$ is a minimal reducing subspace of $\rS$, and $\rS_i = \rS|_{\fS[\cV_i]}$ is irreducible.
The decomposition of $\hat\cV$ in item $(3)$ yields a decomposition of $\rS$ into a finite direct sum of irreducible TCK families $\rS \simeq \bigoplus_i \rS_i$.
\end{enumerate}
\end{theorem}

Hence, up to unitary equivalence we obtain the following $2 \times 2$ block and finite sum decomposition 
\[
 \fS =   
 \bigoplus_i \begin{bmatrix}
  \cM_{d_i} & 0 \\
  B(\bC^{d_i}, \fS[\cV_i] \ominus \cV_i)  & \mathfrak{L}
 \end{bmatrix},
\]
Where $\mathfrak{L}$ is the left-regular type free semigroupoid algebra generated by $\bigoplus_{v\in V}\rL_{G,v}^{(\alpha^{(i)}_{v})}$ for some multiplicities $\alpha_v^{(i)} \geq 0$.

Suppose $\rS$ and $\rT$ finitely-correlated TCK families  with spaces $\hat \cV_{\rS}$ and $\hat \cV_{\rT}$ as in Theorem \ref{T:fuller}. By \cite[Corollary 2.28]{Fuller2011} they are unitarily equivalent if and only if the operator families $P_{\hat \cV_{\rS}} \rS |_{\hat \cV_{\rS}}$ and $P_{\hat \cV_{\rT}} \rT|_{\hat \cV_{\rS}}$ are unitarily equivalent. When $\rS$ is irreducible, so that the above direct sum decomposition of $\fS$ has only a single summand, and $\rA = P_{\hat \cV_{\rS}} \rS |_{\hat \cV_{\rS}} = (A_v, A_e)$, a computation that we leave to the reader shows that the multiplicities $\alpha_v$ for $v\in V$ can be recovered from the ranks of the projections $\{A_v\}_{v\in V}$, via the formula 
$$
\alpha_v = -\rank(A_v) + \sum_{r(e)=v}\rank(A_{s(e)}).
$$

In this paper we will put Theorem \ref{T:fuller} in the wider context of the structure for general free semigroupoid algebras. 
We next describe how this is accomplished.


\begin{definition}
Let $G$ be a directed graph and $\fS$ be a free semigroupoid algebra generated by a TCK family $\rS$ on $\cH$. A vector $0 \neq x \in \cH$ is a \emph{wandering vector} for $\fS$ if $\{S_{\mu}x\}_{\mu \in \bF_G^+}$ are pairwise orthogonal.
\end{definition}

Wandering vectors are useful for detecting invariant subspaces which are unitarily equivalent to left-regular free semigroupoid algebras. At the end of \cite{MS2011}, Muhly and Solel ask if Kennedy's wandering vector results can be extended beyond free semigroup algebras. The following is a simplified version of our structure theorem, and extends Kennedy's wandering vector results in \cite{Kennedy2011} to free semigroupoid algebras.

\vspace{6pt}

\noindent {\bf Theorem \ref{thm:structure}} (Structure theorem){\bf .} Let $\fS$ be a free semigroupoid algebra on $\cH$ generated by a TCK family $\rS=(S_v,S_e)$ for a directed graph $G = (V,E)$. Let $\fM = W^*(\rS)$ be the von Neumann algebra generated by $\rS$. Then there is a projection $P \in \fS$ such that
\begin{enumerate}

\item $\fS P = \fM P$ so that $P\fS P$ is self-adjoint.

\item If $\fS$ is nonself-adjoint then $P\neq I$.

\item $P^{\perp}\cH$ is invariant for $\fS$.

\item With respect to $\cH = P\cH \oplus P^{\perp}\cH$, there is a $2 \times 2$ block decomposition 
\[
 \fS = 
 \begin{bmatrix}
  P\mathfrak{M} P & 0 \\
  P^{\perp}\mathfrak{M}P & \fS P^{\perp}
 \end{bmatrix}
\]

\item When $P \neq I$ then $\fS P^{\perp}$ is completely isometrically and weak-$*$ homeomorphically isomorphic to $\fL_{G'}$ for some induced subgraph $G'$ of $G$.

\item $P^{\perp}\cH$ is the closed linear span of wandering vectors for $\fS$.

\item If every vertex $v\in V$ lies on a cycle, then $P$ is the largest projection such that $P\fS P$ is self-adjoint.

\end{enumerate}

\vspace{6pt}

Hence, we see that the situation described after Theorem \ref{T:fuller} holds in greater generality, in line with the structure theorem from \cite{DKP}.

In an attempt to resolve the wandering vector problem, Davidson, J. Li and Pitts \cite{DLP} introduced a notion of absolute continuity for free semigroup algebras. 
This proved to be a critical new idea that identified analyticity in the algebra via a property of vector states.
Muhly and Solel \cite{MS2011} extended this definition and generalized many of the results in \cite{DLP} to isometric representations of W*-correspondences. 
In his second paper, Kennedy \cite{Kennedy2013} related absolute continuity to wandering vectors, and proved that for free semigroup algebras on at least two generators, absolute continuity is equivalent to analyticity. 
As a consequence, he established a Lebesgue-von Neumann-Wold decomposition of representations of $\cT_n$.

We say that a free semigroupoid algebra $\fS$ generated by a TCK family $\rS$ on $\cH$ is \emph{absolutely continuous} if the associated representation $\pi_{\rS} : \cT(G) \rightarrow B(\cH)$ has a weak-$*$ continuous extension to $\fL_G$. 
This definition was introduced in \cite{MS2011} for representations of W*-correspondences, and is motivated from operator theory where we ask for an extension of the analytic functional calculus of a contraction to an $H^{\infty}$ functional calculus. 

For every free semigroupoid algebra, there is a largest invariant subspace $\cV_{ac}$ on which it is absolutely continuous. We will say that $\fS$ is \emph{regular} if the absolutely continuous part $\cV_{ac}$ coincides with $P^{\perp}\cH$ where $P$ is the structure projection from Theorem \ref{thm:structure}. The following extends Kennedy's characterization of absolute continuity.

\vspace{6pt}

\noindent {\bf Theorem \ref{T:abs_cont}.} Let $\fS$ be a free semigroupoid algebra on $\cH$ generated by a TCK family $\rS=(S_v,S_e)$ of a graph $G$. Then $\fS$ is regular if and only if 
\begin{itemize} 
\item[\normalfont{(M)}] Whenever $\mu \in \bF_G^+$ is a cycle such that $S_\mu$ is unitary on $S_{s(\mu)}\cH$, the spectral measure of $S_\mu$ is either singular or dominates Lebesgue on $\bT$.
\end{itemize}

\vspace{6pt}

The above allows also us to extend Kennedy's decomposition theorem to representations of $\cT(G)$. We will say that a free semigroupoid algebra $\fS$ is of \emph{von Neumann type} if it is a von Neumann algebra, and that it is of \emph{dilation type} if $P\cH$ is cyclic for $\fS$, and $P^{\perp}\cH$ is cyclic for $\fS^*$ where $P$ is the structure projection of Theorem \ref{thm:structure}. We have already seen in Theorem \ref{T:fuller} that dilation type free semigroupoid algebras are certainly abundant, even when $P\cH$ is finite dimensional. The following extends Kennedy's Lebesgue decomposition theorem to regular TCK families.

\vspace{6pt}

\noindent {\bf Theorem \ref{T:LvNW-decomp}.} Let $G$ be a graph and $\rS$ be a TCK family on $\cH$ that generates a \emph{regular} free semigroupoid algebra $\fS$. Then up to unitary equivalence we may decompose
\[
\rS \cong \rS_l \oplus \rS_{a} \oplus \rS_{s} \oplus \rS_d
\]
where $\rS_l$ is left-regular type, $\rS_{a}$ is absolutely continuous, $\rS_{s}$ is of von Neumann type and $\rS_d$ is of dilation type.

\vspace{6pt}

Our results are applied to obtain a number of consequences including reflexivity of free semigroupoid algebras (see Theorem \ref{T:reflexivity}), a Kaplansky density theorem for regular free semigroupoid algebras (see Theorem \ref{T:Kaplansky}) and an isomorphism type theorem for nonself-adjoint free semigroupoid algebras of transitive row-finite graphs (see Corollary \ref{C:isomorphisms}). 
The first two results show that regular free semigroupoid algebras behave much like von Neumann algebras, while the third result shows that despite the above, when they are nonself-adjoint they still completely encode transitive row-finite graphs up to isomorphism.

We conclude our paper by developing classes of examples of free semigroupoid algebras that illustrate our structure and decomposition theorems, as we explain next. Examples are important for illuminating the results. Earlier in the paper, the left regular representations are examined carefully.
Also representations for a cycle are carefully studied. Here we provide two large classes of rather different characters.

In \cite{DKP} it was asked if a free semigroup algebra on at least two generators can in fact be self-adjoint. In \cite{Read} Read was able to find a surprising example of two isometries $Z_1,Z_2$ with pairwise orthogonal ranges, such that the weak operator topology algebra generated by them is $B(\cH)$. This points to the curious phenomenon that taking the \wot~closure of a nonself-adjoint algebra can suddenly make it into a von-Neumann algebra. Hence, it is natural to ask whether or not such self-adjoint examples occur for graphs that are on more than a single vertex.

Suppose $G=(V,E)$ is a finite, transitive and in-degree regular directed graph with in-degree $d$. 
A \emph{strong edge colouring} of $E$ is a function $c: E \rightarrow \{1,\dots,d\}$ such that any two distinct edges with the same range have distinct colours. 
Every colouring of $E$ then naturally induces a labeling of paths by words in the colours. 
We will say that a colouring as above is \emph{synchronizing} if for any (some) vertex $v\in V$, there is a word in colours $\gamma \in \bF_d^+$ such that every path $\mu$ with the coloured label $\gamma$ has source $v$.

A famous conjecture of Adler and Weiss from the 70's \cite{AW70} asks if aperiodicity of the graph implies the existence of a synchronized colouring. 
In \cite{Trahtman}, Avraham Trahtman proved this conjecture, now called the road colouring theorem. 
We use the road colouring theorem to provide examples of self-adjoint free semigroupoid algebras for larger collection of graphs.

\vspace{6pt}

\noindent {\bf Theorem \ref{T:Read ex}.} Let $G$ be an aperiodic, in degree regular, transitive and finite directed graph. 
Then there exists a CK family $\rS$ on $\cK$ such that its associated free semigroupoid algebra $\fS$ is $B(\cK)$.

\vspace{6pt}

Finally, we discuss atomic representations of $\cT(G)$, which are combinatorial analogues of TCK families. 
These are generalizations of those developed for representations of $\cT_n$ in \cite{DP1999} by Davidson and Pitts, which are more general than permutative representations introduced by Bratteli and Jorgensen in \cite{BJ99}. 
These examples have proven to be extremely useful for generating interesting representations of the Cuntz algebra, and it is expected that 
these examples in the graph algebra context will be equally important.
The associated TCK families have the property that there is an orthonormal basis on which the partial isometries associated to paths in the graph act by partial permutations, modulo scalars. 
We classify atomic representations up to unitary equivalence by decomposing them into left-regular, inductive and cycle types. 
For each type, we also characterize irreducibility of the representation.

This paper contains 11 sections, including this introduction section. 
In Section \ref{S:prelim} we provide some preliminaries and notation to be used in the paper. 
In Section \ref{S:left-reg-struc} we explain how the left-regular free semigroupoid algebra changes by restricting to reducing subspaces and by taking directed closures of the underlying graph. 
In Section \ref{S:wold} we provide our variant of the Wold decomposition for TCK families and use it to characterize free semigroupoid algebras of certain acyclic graphs. 
In Section \ref{S:special-alg} we reduce problems on wandering vectors to free semigroup algebras and characterize all free semigroupoid algebras of a single cycle. 
These results are used extensively in the rest of the paper. 
In Section \ref{S:stucture} we give a proof of the structure theorem. 
In Section \ref{S:abs cont} we develop absolute continuity of free semigroupoid algebras and relate them to wandering vectors. 
In Section \ref{S:LvNW} we generalize and extend Kennedy's Lebesgue decomposition to free semigroupoid algebras. 
In Section \ref{S:applications} we show free semigroupoid algebras are reflexive, provide a Kaplansky density theorem and prove isomorphism theorems for non-self-adjoint free semigroupoid algebras. 
In Section \ref{S:self-adjoint} we show that $B(\cK)$ is a free semigroupoid algebra for transitive, in-degree regular, finite and aperiodic graphs. 
Finally, in Section \ref{S:atomic} we classify atomic representations up to unitary equivalence by decomposing them into left-regular, inductive and cycle types.

\section{Preliminaries} \label{S:prelim}

Here we lay out some of the theory of algebras associated to directed graphs. We recommend \cite{Raeburn, JK2005, KK2005, KP2004} for further material and proofs for some standard results mentioned here.

A (countable) directed graph $G=(V,E,s,r)$ consists of a countable vertex set $V$ and a countable edge set $E$, and  maps $s,r:E\to V$.
To each edge $e\in E$, we call $s(e)$ its source and $r(e)$ its range in $V$. One should think of $e$ as an arrow or directed edge from $s(e)$ to $r(e)$. 
A finite path in $G$ is a finite sequence $\mu = e_n \dots e_1$ where $e_i \in E$ and $s(e_{i+1})=r(e_i)$ for $1 \le i < n$. 
We denote by $\bF_G^+$ the collection of all paths of finite length in $G$, called the path space of $G$, or the free semigroupoid of $G$. 
For each element $\mu \in \bF_G^+$, define the source $s(\mu) = s(e_1)$ and range $r(\mu) = r(e_n)$ in $V$. 
We can compose two paths by concatenation provided that the range and source match, i.e.\ if $\mu, \nu \in \bF_G^+$, 
say $\nu = f_m \dots f_1$, and $s(\mu) = r(\nu)$, then $\mu\nu = e_n\dots e_1f_m\dots f_1$. This makes $\bF_G^+$ into a semigroupoid. Also, we denote by $|\mu| :=n$ the length of the path $\mu = e_n...e_1 \in \bF_G^+$. A vertex $v$ is considered an element in $\bF^+_G$ as a path of length $0$ with $s(v)=r(v)=v$.

When $\rS = (S_v,S_e)$ is a TCK family, we define partial isometries for each path $\mu \in \bF_G^+$, where $\mu = e_n \dots e_1$ for $e_i\in E$, by 
\[
S_{\mu} = S_{e_n}\cdots S_{e_1}.
\]
We will make use of the fact that the C*-algebra $\ca(\rS)$ generated by a TCK family $\rS=(S_v, S_e)$ has the following description as a closed linear span,
\[
\ca(\rS) = \ol{\spn} \{ S_{\mu} S_{\nu}^* : \mu,\nu \in \bF_G^+  \} ,
\]
as can be verified by the use of (IS). In other words, the $*$-algebra generated by $\rS$, which is norm dense in $\ca(\rS)$, is the linear span of the terms $S_{\mu} S_{\nu}^*$ where $\mu,\nu \in \bF_G^+$.

Since we are interested in the \wot-closed algebra, it is worth noting that property (F) is not readily identified at the norm closed level. Also, if $V$ is finite, then $\cT(G)$ is unital, and non-degeneracy of the representation reduces to the statement that it is also unital. However, if $V$ is infinite, $\cT(G)$ is not unital, and non-degeneracy is also not readily detected in the norm closed algebra $\cT(G)$. For this reason, we will concentrate on non-degenerate TCK families rather than their corresponding non-degenerate representations of $\cT(G)$.

The (left) tensor algebra determined by $G$ is the norm closed algebra generated by $\rL$,
\[
 \cT_+(G) := \ol{\spn \{ L_{\mu} : \mu \in \bF_G^+ \}}.
\]
From their universality, we can deduce that the algebras $\cT_+(G)$, $\cT(G)$ and $\cO(G)$ carry an additional grading by $\bZ$ implemented by a point-norm continuous action of the unit circle. However, a spatial description for this grading is more useful to us. The space $\cH_G$ is $\bN$-graded by components $\cH_{G,n}$ with orthonormal basis $\{ \xi_{\mu} :|\mu| = n, \ \mu \in \bF_G^+  \}$ for $n \in \bN$. 
We denote by $E_n$ the projection onto $\cH_{G,n}$.

We may then define a unitary $W_{\lambda} : \cH_G \rightarrow \cH_G$ for every $\lambda \in \bT$ by specifying $W_{\lambda}(\xi_{\mu}) = \lambda^{|\mu|}\xi_{\mu}$. 
Then $\alpha : \bT \rightarrow B(\cH_G)$ given by $\alpha_{\lambda}(T) = W_{\lambda} T W_{\lambda}^*$ becomes a group action satisfying $\alpha_{\lambda}(L_v) = L_v$ and $\alpha_{\lambda}(L_e) = \lambda \cdot L_e$ for each $v\in V$ and $e\in E$. Since $\cT(G)$, $\cT_+(G)$ and $\fL_G$ are invariant under this action, we may look at $\alpha$ as an action on the algebras $\fL_G$, $\cT_+(G)$ and $\cT(G)$. 
In the first case it becomes a point-\wot\ continuous action, and in the latter two it becomes a point-norm continuous action.

For each of these algebras, this then enables the definition of Fourier coefficients, which are maps $\Phi_m$ from the algebra to itself.
The formula is given for every $m\in \bZ$ and $T\in \cT(G)$ by
\[
 \Phi_m(T) = \int_{\bT} \alpha_{\lambda}(T) \lambda^{-m-1} \,d\lambda 
 = \frac1{2\pi} \int_{-\pi}^\pi \alpha_{e^{i\theta}}(T) e^{-im\theta} \,d\theta.
\]
The operator $\Phi_m(T)$ maps each subspace $\cH_{G,n}$ into $\cH_{G,n+m}$ (where $\cH_{G,n+m}$ is $\{0\}$ when $n+m < 0$).
We define the Ces\`aro means by
\[
 \Sigma_k(T) = \sum_{|m| < k} \Big(1-\frac{|m|}{k} \Big)\Phi_m(T).
\]
These are completely contractive maps, and they are \wot-continuous on $\fL_G$. For $T \in \fL_G$, the means $\Sigma_k(T)$ converge in the \wot\ topology to $T$; whereas for $T$ in either $\cT(G)$ or $\cT_+(G)$, they converge in norm.
When $T \in \fL_G$, we have a Fourier series of the form $\sum_{\mu \in \bF_G^+}a_{\mu} L_{\mu}$, where $a_{\mu} = \langle T \xi_{s(\mu)}, \xi_{\mu} \rangle$ for $\mu \in \bF_G^+$, which uniquely determines the operator $T$. 
Thus, by uniqueness of Fourier coefficients we have that $\Phi_m(T) = \wotsum_{|\mu|=m} a_\mu L_\mu$, so that the Ces\`aro means can thus be rewritten in the form 
\[
\Sigma_k(T) = \wotsum_{|\mu|<k} \Big( 1 - \frac{|\mu|}{k} \Big) a_{\mu}L_{\mu}.
\]

\begin{remark} \label{R:Cesaro_in_T_+(G)}
Since there can be infinitely many edges, even between two vertices, $\Sigma_k(T)$ may not belong to $\cT_+(G)$ when $T \in \fL_G$.
This can be remedied using the net $\Sigma_k(A)L_F$ where $L_F = \sum_{v\in F} L_v$ for finite subsets $F \subseteq V$.
Observe that $\Phi_m(T) L_v  = \wotsum_{|\mu|=m,\ s(\mu) = v} a_\mu L_\mu$. 
Because the isometries $L_\mu$ with $|\mu|=m$ and $s(\mu) = v$ have pairwise orthogonal ranges and a common domain, we have
\[
 \|T\| \ge \| \Phi_m(T) \| 
 \ge \Big\| \wot\text{--}\!\!\!\!\!\sum_{\substack{|\mu|=m\\s(\mu) = v}} a_\mu L_\mu \Big\| 
 = \Big( \sum_{\substack{|\mu|=m\\s(\mu) = v}} |a_\mu|^2 \Big)^{1/2} .
\]
It follows that the sequence $(a_\mu)$ running over paths $\mu$ with $|\mu|=m$ and $s(\mu) = v$ belongs to $\ell_2$.
Hence, the sum $\sum_{\substack{|\mu|=m\\s(\mu) = v}} a_\mu L_\mu$ converges \textit{in norm}.
Therefore $\Phi_m(T) L_F$ belongs to $\cT_+(G)$; whence so does $\Sigma_k(T) L_F$.
Moreover  $\Sigma_k(T) L_F$ converges in the \wot\ topology to $T\in \fL_G$ as $k\to\infty$ and $F \to V$.
\end{remark}

For a non-degenerate TCK family $\rS=(S_v, S_e)$ of a directed graph $G$ acting on a Hilbert space $\cH$, we denote by
\[
\fS:= \wot\!\!-\!\!\Alg\{ S_{\mu} :\mu \in \bF_G^+  \}
\]
the free semigroupoid algebra generated by $S$ inside $B(\cH)$, which is unital. The underlying graph $G$ will be clear from the context.

An important way of constructing free semigroupoid algebras is through dilation of contractive $G$-families. Let $\rA=(A_v,A_e)$ be a family of operators on a Hilbert space $\cH$. We say that $\rA$ is a contractive $G$-family if
\begin{enumerate} [leftmargin=15mm]
\item[(P)] $\{A_v : v \in V\}$ is a set of pairwise orthogonal projections;
\item[(C)] $A_e^*A_e \le A_{s(e)}$, for every $e\in E$ ;
\item[(TCK)] $\sum_{e\in F} A_e A_e^* \le A_{v}$ for every finite subset $F \subseteq r^{-1}(v)$.
\end{enumerate}
The family above is \emph{CK-coisometric} if 
\begin{enumerate}[leftmargin=15mm]
\item[(CK)]
$\sum_{r(e)=v} A_e A_e^* = A_v$, for every $v \in V$ with $0 < |r^{-1}(v)| < \infty$, 
\end{enumerate}
and \emph{fully coisometric} if we have
\begin{enumerate}[leftmargin=15mm]
\item[(F)] $\sotsum_{v\in V} \sum_{r(e)=v} A_e A_e^* = A_v$ for every $v \in V$.
\end{enumerate}

From the work of Muhly and Solel \cite[Theorem 3.10]{MS1998}, we know that completely contractive representations of $\cT_+(G)$ are in bijection with contractive $G$-families $\rA$ as above. Hence, once again we can treat contractive $G$-families and completely contractive representations of $\cT_+(G)$ interchangeably. 

Contractive $G$-families are a lot easier to construct than TCK families. 
For instance, there are graphs $G$ with no TCK families on a finite dimensional Hilbert spaces, while on the other hand one can construct a contractive $G$-family for any graph on any finite dimensional space. 
In fact, free semigroupoid algebras generated by TCK dilations of contractive $G$-families on finite dimensional spaces are exactly the finitely correlated free semigroupoid algebras in Theorem \ref{T:fuller}.

Given a free semigroupoid algebra $\fS$ that is generated by a TCK family $\rS = (S_v,S_e)$ on $\cH$, for any co-invariant subspace $\cK \subseteq \cH$, define $A_\mu=P_\cK S_\mu \big|_\cK$ for $\mu\in \bF_G^+$. Then $\rA$ is a contractive $G$-family.

The graph-analogue of the Bunce, Frahzo, Popescu dilation theorem \cite{Fr1982, Bun84, Pop1989a} gives us the converse. 
Indeed, if we start with a contractive $G$-family $\rA = (A_v,A_e)$ on a Hilbert space $\cK$, by \cite[Theorem 3.3]{MS1998} we see that $\rA$ always has a unique minimal dilation to a TCK family. 
More precisely, there is a TCK family $\rS = (S_v,S_e)$ for $G$ on a Hilbert space $\cH$ containing $\cK$ as a co-invariant subspace such that $A_\mu=P_\cK S_\mu \big|_\cK$ for $\mu\in \bF_G^+$. 
In fact, $\rS$ can be chosen \emph{minimal} in the sense that $\cK$ is cyclic for $\fS$, and any two minimal TCK dilations for $\rA$ are unitarily equivalent. 
Moreover, by \cite[Corollary 5.21]{MS1998} we have that $\rS$ is CK when $\rA$ is CK-coisometric, and is fully coisometric when $\rA$ is fully coisometric. 
Hence, contractive $G$-families yield many interesting examples of TCK families via dilation.


Let $G=(V,E)$ be a directed graph. We call a subset $F\subseteq V$ {\em directed} if whenever $e\in E$ is such that $s(e)\in F$, then $r(e) \in F$. If $F \subseteq V$, let $F_G$ denote the smallest directed subset of $V$ containing $F$ in $G$. In general, we will denote by $G[F]:= G[F_G]$ the subgraph induced from $G$ on the smallest directed subset of $V$ containing $F$.

Given a TCK family $\rS$ for $G$, denote by $\supp(\rS)$ the collection of vertices $v\in V$ such that $S_v \neq 0$. 
Since for each edge $e\in E$, we have $S_{s(e)} =S_e^*S_e$ and $S_e S_e^* \leq S_{r(e)}$, we see that $\supp(\rS)$ is directed. Hence, if we take the subgraph $G[\rS] := G[\supp(\rS)]$, we obtain a reduced TCK family 
\[ \rS_r :=\{ S_v, S_e : v\in \supp(\rS),\ e\in s^{-1}(\supp(\rS)) \} \]
for $G[\rS] = (\supp(\rS), s^{-1}(\supp(\rS)))$. This means that we may assume from the outset that $S_v \neq 0$ for all $v\in V$ by identifying $\rS$ with a TCK family for an induced subgraph on a directed subset of vertices. 
We will say that a TCK family $\rS$ for a graph $G$ is \emph{fully supported} if $\supp(\rS) = V$.

We say that a directed graph $G$ is {\em transitive} if for any two vertices $v,w \in V$ we have a path $\mu \in \bF_G^+$ with $s(\mu) = v$ and $r(\mu) = w$. If $G$ is transitive, then since $\supp(\rS)$ is directed, either $\supp(\rS) = \emptyset$ or $\supp(\rS) = V$. Thus, when the TCK family is non-degenerate and $G$ is transitive, we must have $\supp(\rS)=V$.

Suppose that the graph $G$ is the union of countably many connected components as an undirected graph, say $G = \bigcupdot G_i$ is a union of undirected connected components $G_i = (V_i,E_i)$. 
We have a corresponding decomposition of $\fS$.
Let $\cH_i:= (\sotsum_{v\in V_i}S_v)\cH$. 
Then $\cH = \bigoplus \cH_{i}$ and $\fS_G = \prod \fS_{G_i}$ is the $\linf$-direct sum (direct product) of free semigroupoid algebras $\fS_{G_i}$ for $G_i$ acting on $\cH_{i}$. 
Hence, we may always assume that $G=(V,E)$ is connected as an undirected graph.

\section{Structure of the left regular algebra} \label{S:left-reg-struc}

In this section, we give some basic results about the structure of the left regular free semigroupoid algebra, relating them to their associated graph. we will assume that $G$ is connected as an undirected graph.

If $F,F' \subseteq V$ are two subsets, and $F'\subseteq F_G$, we write $F \succ F'$; i.e., $F \succ F'$ if $F'$ is contained in the smallest directed subset containing $F$. This is a transitive relation on subsets of $V$. One can have $F\succ F'$ and $F' \not \succ F$. 
We will also write $v \succ w$ to mean that $\{v\} \succ \{w\}$. Note that $G$ is transitive precisely when $v \succ w$ for any two vertices $v,w\in V$. 

Every path has a source, and hence $\cH_G \cong \bigoplus_{v\in V}R_v \cH_G$. For each subset $F\subseteq V$, we define 
\[ \cH_{G,F}:= \bigoplus_{v\in F}R_v \cH_G = \spn\{ \xi_{\lambda} : s(\lambda)\in F \} .\]
In particular, we use 
\[ \cH_{G,v} := \cH_{G,\{v\}}= R_v \cH_G = \spn\{ \xi_\lambda : s(\lambda) = v \} \qfor v\in V. \]
Each $\cH_{G,F}$ is easily seen to be a reducing subspace for $\fL_G$. For a subset $F\subseteq V$, we also write $\rL_{G,F}= (L_v|_{\cH_{G,F}}, L_e|_{\cH_{G,F}} : v \in V,\ e \in E )$ for the family of operators obtained by restricting to $\cH_{G,F}$ and similarly $\rL_{G,v}:= \rL_{G,\{v\}}$.
We define  $\fL_{G,F}:= \wotclos{\fL_G |_{\cH_{G,F}}}$ and set $\rL_{G,v}: = \rL_{G,\{v\}}$ and $\fL_{G,v} :=\fL_{G,\{v\}} = \wotclos{\fL_G |_{\cH_{G,v}}}$. We will see in Theorem~\ref{T:left_quotient} that the closure is unnecessary; i.e., $\fL_{G,F} = \fL_G|_{\cH_{G,F}}$. 

\begin{theorem} \label{T:left_quotient}
Suppose that $F\succ F'$ are subsets of vertices of $G$. Then there is a complete quotient map $\rho_{F,F'}:\fL_{G,F} \to \fL_{G,F'}$ given by the natural restriction map; i.e. $ \fL_{G,F'}$ is completely isometrically isomorphic and weak-$*$ homeomorphic to $\fL_{G,F}/\ker\rho_{F,F'}$. 
Moreover, there is a canonical weak-$*$ continuous completely isometric homomorphism $j_{F',F}: \fL_{G,F'} \to \fL_{G,F}$ which is a section of the quotient map.

In particular, there is a canonical weak-$*$ continuous completely isometric homomorphism $j_{F'}: \fL_{G,F'} \to \fL_G$ which is a section of the quotient map from $\fL_G$ to $\fL_{G,F'}$.
\end{theorem}

\begin{proof}
Since $F\succ F'$, for each $w\in F'$ there is a path $\mu_w$ with $s(\mu_w) = v_w \in F$ and $r(\mu_w) = w$.
There is a natural injection of $\cH_{G,w}$ into $\cH_{G,v_w}$ given by $R_{\mu_w}$, where $R_{\mu_w} \xi_\nu = \xi_{\nu\mu_w}$ for each $\nu \in \bF_G^+$ with $s(\nu)=w$. 
Thus, $R= \oplus_{w\in F'} R_{\mu_w}$ provides a natural injection of $\cH_{G,F'}$ into $\cH_{G,F}^{(\alpha)}$, where $\alpha = |F'|$. 
For each $\lambda\in \bF_G^+$ with $s(\lambda) = w$, we then have
\[   R_{\mu_w}^* L_\lambda|_{\cH_{G,s(\mu_w)}} R_{\mu_w} = L_\lambda|_{\cH_{G,w}} .\]
That is, the range of $R_{\mu_w}$ is an invariant subspace of $\cH_{G,s(\mu_w)}$ such that the compression of $L_\lambda|_{\cH_{G,s(\mu_w)}}$ to $R_{\mu_w}\cH_{G,w}$ is canonically unitarily equivalent to $L_\lambda|_{\cH_{G,w}}$. 
Hence, as this occurs on each block separately, we get
\[   R^* (L_\lambda|_{\cH_{G,F}})^{(\alpha)} R = L_\lambda|_{\cH_{G,F'}} .\]

Therefore, the map $\rho_{F,F'}:\fL_{G,F} \to \fL_{G,F'}$ given by $\rho_{F,F'}(A) = R^*A^{(\alpha)} R$ is a completely contractive homomorphism. Since this map is the composition of an ampliation and a spatial map, it is weak-$*$ continuous.

This isometry $R$ depends on the choice of the path $\mu_w$, but the map $\rho_{F,F'}$ is independent of this choice.
To see this, suppose that $\mu'_w$ is another set of paths as above; and let $R' = \oplus_{w\in F'}R_{\mu'_w}$ be the corresponding isometry.
Then the unitary map $W:\ran R \to \ran R'$ given by $W \xi_{\nu\mu_w} = \xi_{\nu \mu'_w}$ is readily seen 
to satisfy $WR=R'$ and $L_\lambda R' = W L_\lambda R$ for each $\lambda\in \bF_G^+$. So the map $\rho_{F,F'}$ is the unique map satisfying 
\[ \rho_{F,F'}(L_\lambda|_{\cH_{G,F}}) = L_\lambda|_{\cH_{G,F'}} .\]

Let $A\in\fL_{G,F'}$, and for a finite set $J \subseteq V$ denote $L_J := \oplus_{v\in J}L_v$. Then $\|\Sigma_k(A) L_J\|\le \|A\|$, and the bounded net $\Sigma_k(A)L_J$  for $k\ge1$ and finite subsets $J\subseteq V$ converges to $A$ in the \wot\ topology by Remark~\ref{R:Cesaro_in_T_+(G)}.
Observe that paths $\lambda\in \bF_G^+$ with non-zero coefficients in the Fourier expansion of $A$ all have source (and range) in the minimal directed subset $F'_G$ of $G$ containing $F'$.
However, being a polynomial,  $\Sigma_k(A)L_J$ makes sense as an element of $\fL_{G,F}$.
Observe that when acting on $\cH_{G,F}$, we have $\Sigma_k(A) \xi_\mu = 0$ if $r(\mu) \not\in F'_G$.
Indeed, the subspace 
\[ \cK = \spn\{ \xi_\mu  \in \cH_{G,F} : r(\mu)\in F'_G \} \]
is an invariant subspace of $\cH_{G,F}$. 

We claim that $\cK$ is contained in some multiple of $\cH_{G,F'}$ and conversely $\cH_{G,F'}$ is contained in some multiple of $\cK$.
Let $M$ be the set of all minimal length paths $\mu = e_n\dots e_1 \in \bF_G^+$ with $s(\mu)\in F$ and $r(\mu) \in F'_G$. Hence, we have that $r(e_i) \not\in F'_G$ for $i<n$.
We see that
\[
 \cK = \oplus_{\mu\in M} \fL_{G,r(\mu)}[\xi_\mu] 
 = \oplus_{\mu\in M} R_\mu \cH_{G,r(\mu)} 
 = R'[\oplus_{\mu\in M}  \cH_{G,r(\mu)}],
\]
where $R' = \oplus_{\mu \in M} R_{\mu}$. Thus, $\cK$ may be identified as a reducing subspace of multiplicity $|M|$ of $\cH_{G,F'}$. This identification is again canonical and intertwines the action of each $L_\lambda$.
On the other hand, since $F \succ F'$, we have for each $w\in F'$ that $\cH_{G,v_w}$ may be identified as a subspace of $\cK$ via $R_{\mu_w}$. Thus $\cH_{G,F'}$ may be identified as an invariant subspace of $K^{(\beta)}$ via $R= \oplus_{w\in F'} R_{\mu_w}$ with $\beta = |F'|$.

The subspace $\cH_{G,F} \ominus \cK$ is spanned by the vectors $\xi_\mu$ with $\mu = f_k\dots f_1$ such that $r(f_i) \not\in F'_G$ for $1 \le i \le k$.
It follows that $\Sigma_k(A)$ vanishes on $\cH_{G,v} \ominus \cK$.
Therefore, we see that
\[ 
 \Sigma_k(A) L_J|_{\cH_{G,F}} \cong  
 0|_{\cH_{G,v} \ominus \cK } \oplus \bigoplus_{\mu\in\cM} \Sigma_k(A) L_J |_{\cH_{G,r(\mu)}} .
\]
We have already observed above that $\| \Sigma_k(A)L_J |_{\cH_{G,r(\mu_w)}} \| \le \| \Sigma_k(A)L_J |_{\cH_{G,w}} \|$ for $w\in F'$.
Therefore, we deduce that 
\[ \| \Sigma_k(A)L_J |_{\cH_{G,F}} \| \le \| \Sigma_k(A)L_J |_{\cH_{G,F'}} \| \le \| \Sigma_k(A)L_J |_{\cH_{G,F}} \| .\]
Thus we have equality. Note that the identical argument is valid for matrices in $M_n(\fL_{G,F'})$. 

We use this to show that $\rho_{F,F'}$ is a complete quotient map. Take $A$ in $M_n(\fL_{G,F'})$. 
By the previous paragraph, for a finite subset $J\subset V$ we have 
\[ \| \Sigma_k^{(n)}(A)L_J  \| = \| \Sigma_k^{(n)}(A)L_J |_{\cH_{G,F'}^{(n)}} \| \le \|A\| .\]
By \wot-compactness of the unit ball of $M_n(\fL_{G,F})$, there is a cofinal convergent subnet with limit $B\in M_n(\fL_{G,F})$
with $\|B\|\le \|A\|$.
However, the Fourier series of $B$ is readily seen to coincide with the Fourier series of $A$.
Indeed, using the setup of the previous paragraph, we see that
\[
 B \cong  0|_{\cH_{G,F} \ominus \cK } \oplus \bigoplus_{\mu\in\cM} A|_{\cH_{G,r(\mu)}} .
\]

In particular, the map $j_{F',F}: \fL_{G,F'} \to \fL_{G,F}$ given by $j_{F',F}(A) = B$ is a completely isometric homomorphism of $\fL_{G,F'}$ into $\fL_{G,F}$.
The structure of $B$ makes it apparent that $\rho_{F,F'} \circ j_{F',F} = \id_{\fL_{G,F'}}$. So this is the desired section.
Since $j_{F',F}$ is completely isometric, it follows that $\rho_{F,F'}$ is a weak-$*$ continuous complete quotient map.
Since the composition of the section produced in the previous paragraph and the quotient is weak-$*$ continuous as well,
it follows that the section is weak-$*$ continuous.

The last statement of the theorem now follows from the fact that $\cH_G = \bigoplus_{v\in V} \cH_{G,v} = \cH_{G,V}$, and $V\succ F'$ for any subset $F' \subseteq V$.
\end{proof}

We obtain the following useful consequences, which will be useful to us throughout the paper.

\begin{corollary} \label{C:fL_{G,F}}
If $G$ is a graph and $F \subseteq V$, then $\fL_G|_{\cH_{G,F}} = \fL_{G,F}$.
\end{corollary}

\begin{corollary} \label{C:lr_max_element}
If $G$ is a graph and $F \subseteq V$ such that $F \succ w$ for all $w\in V$, then $\rho_F:\fL_G \to \fL_{G,F}$ is a completely isometric isomorphism and weak-$*$ homeomorphism. In particular when $G$ is transitive, this occurs for any subset $F=\{v\}$ for any vertex $v\in V$.
\end{corollary}

\begin{corollary} \label{C:lr_inductive}
Suppose that $G$ has a collection $\{F_i\}_{i\in I}$ of subsets of $V$ linearly ordered by $I$ such that $F_i \succ F_j$ when $i\geq j$, and for every $w\in V$, there is some $i$ so that $F_i \succ w$.
Then $\fL_G$ is a \wot-projective limit of the system $\rho_{F_i,F_j} :  \fL_{G,F_i} \to \fL_{G,F_j}$ for $j \leq i$,
and a \wot-inductive limit of the system $j_{F_j,F_i} : \fL_{G,F_j} \to \fL_{G,F_i}$ for $j \leq i$.
\end{corollary}

\begin{proof}
The maps $\rho_{F_i}:\fL_{G} \to \fL_{G,F_i}$ are the connecting maps which establish the projective limit.
Likewise, the maps $j_{F_i} : \fL_{G,F_i} \to \fL_G$ are the connecting maps which establish the inductive limit.
\end{proof}

\section{Wold decomposition} \label{S:wold}

We prove a Wold-decomposition type theorem for general Toeplitz-Cuntz-Krieger families which is similar to \cite[Section 2]{JK2005} and \cite[Theorem 3.2]{DS2017}. 
A more general Wold decomposition in the setting of C*-correspondences can be found in \cite[Corollary 2.15]{MS1999}. 
We show that once one removes all the reducing subspaces determined by the left regular part, what remains is a fully coisometric CK family for the largest sourceless subgraph $G^0$ of $G$. 

\begin{definition}
Given a free semigroupoid algebra $\fS$ generated by a TCK family $\rS= (S_v,S_e)$, 
we say that a non-trivial subspace $\cW$ is a \textit{wandering subspace} provided that $\{ S_\mu\cW : \mu \in \bF_G^+ \}$ are pairwise orthogonal. (Note that some of these subspaces may be $\{0\}$.)
A wandering subspace determines the invariant subspace
\[ \fS[\cW] = \ol{\fS \cW} = \bigoplus_{\lambda\in \bF_G^+} S_\lambda \cW .\]
We call a wandering subspace $\cW$ \emph{separated} if $\cW = \oplus_{v\in V} S_v\cW$.

The \textit{support} of $\cW$ is the set $\supp(\cW) = \{ v\in V : S_v \cW \ne \{0\} \}$.
A vector $0\neq \xi$ is called \textit{wandering} provided that $\bC \xi$ is a wandering subspace, and $\supp(\xi) := \supp(\bC \xi)$. We say that $\cW$ has \textit{full support} if $\supp(\cW) = V$.
\end{definition}

For a subspace $\cW$, denote by 
$$
\cW':= \fS[\cW] \ominus \bigoplus_{|\lambda| \geq 1} S_\lambda \cW
$$
where $\lambda \in \bF_G^+$ is a path. Then $\cW'$ is always a separated wandering space with $\fS[\cW']=\fS[\cW]$, and $\cW$ is a separated wandering subspace exactly when $\cW = \cW'$. 
When $\rS$ is non-degenerate, every wandering subspace supported on a single vertex is separated.

For a TCK family $\rS=(S_v,S_e)$ on a Hilbert space $\cH$ and a common invariant subspace $\cK \subseteq \cH$, we will write $\rS|_{\cK}:= (S_v|_{\cK}, S_e|_{\cK})$. 
Furthermore, for another TCK family $\rT=(T_v,T_e)$ on $\cH'$, we will let 
\[ \rT\oplus \rS := (T_v \oplus S_v, T_e \oplus S_e) \]
denote the TCK family direct sum on $\cH \oplus \cH'$, and $\rT^{(\alpha)}$ denotes the direct sum of $\rT$ with itself $\alpha$ times. The support of a subspace $\cW$ will generally not be directed, but it is easy to see that $\supp(\rS|_{\fS[\cW]})$ is the smallest directed set containing $\supp(\cW)$. In the following, any infinite sum of orthogonal projections is evaluated as an \sot-limit.

\begin{theorem}[Wold decomposition I] \label{thm:Wold-decomp}
Let $\rS=(S_v, S_e)$ be a non-degenerate TCK family for a graph $G$ acting on a Hilbert space $\cH$ generating the free semigroupoid algebra $\fS$. 
For any $v\in V$, define 
\[ \cW_v:= \big( S_v - \!\!\!\!\sum_{e\in r^{-1}(v)}\!\! S_e S_e^*\big) \cH .\]
If $\cW_v \ne \{0\}$, then $\cW_v$ is a wandering subspace for the reducing subspace $\cH_v = \fS[\cW_v]$ with $\supp(\cW_v) = \{v\}$.

In this case, let $\{\xi_{v,i} : 0 \le i <\alpha_v\}$ be an orthonormal basis for $\cW_v$, where $\alpha_v = \Dim \cW_v$.
Then for each $0 \leq i < \alpha_v$, the set $\{S_{\mu}\xi_{v,i} : s(\mu)=v\}$ is an orthonormal basis for a reducing subspace $\cH_{v,i}$ generated by the wandering vector $\xi_{v,i}$ supported on $v$ such that $\rS_{v,i}:= \rS|_{\cH_{v,i}}$ is unitarily equivalent to $\rL_{G,v}$; 
and $\cH_v = \bigoplus_{i<\alpha_v} \cH_{v,i}$. Furthermore, $\cW := \cH \ominus \sum_{e\in E} S_eS_e^*\cH = \bigoplus_{v\in V} \cW_v$ is a separated wandering space
and $\fS[\cW] = \bigoplus_{v\in V} \fS[\cW_v]$ is the maximal reducing subspace such that $\fS|_{\fS[\cW]}$ is of left regular type.

Thus, the TCK family $\rS$ is unitarily equivalent to
\[
\rT \oplus \bigoplus_{v\in V} \rL_{G,v}^{(\alpha_v)}
\]
where $\rT=(T_v, T_e)$ is a non-degenerate and fully coisometric CK family. This decomposition is unique up to unitary equivalence of $\rT$.
\end{theorem}

\begin{proof}
Note that $\cW_v = S_v \cW_v$, so clearly $\supp(\cW_v) = \{v\}$. 
Suppose that $f\in E$. Then
\[
 \big( S_v - \!\!\!\!\sum_{e\in r^{-1}(v)}\!\! S_e S_e^*\big) S_f =
 \begin{cases} S_f-S_f = 0 &\qif r(f) = v,\\ 0 &\qif r(f) \ne v .\end{cases} 
\]
We will simultaneously show that $\cW_v$ is wandering for all $v\in V$, and that $\cH_v$ and $\cH_w$ are orthogonal for all $v\neq w$ in $V$.
Consider $S_\mu \cW_v$ and $S_\nu \cW_w$ where either $v\ne w$ or $\mu \ne \nu$.
We may suppose that $|\mu|\ge |\nu|$. Write $\mu = \nu'\mu'$ where $|\nu'|=|\nu|$.
Also, if $|\mu'|\ge1$, write $\mu' = e'\mu''$. Then we calculate for $\xi, \eta\in\cH$,
{\allowdisplaybreaks 
\begin{align*}
 \Big\la S_{\mu} &\big( S_v - \!\!\!\!\sum_{e\in r^{-1}(v)}\!\! S_e S_e^*\big) \xi, S_{\nu}\big( S_w - \!\!\!\!\sum_{f\in r^{-1}(v)}\!\! S_f S_f^*\big)\eta \Big\ra \\&=
 \Big \la \big( S_w - \!\!\!\!\sum_{f\in r^{-1}(v)}\!\! S_f S_f^*\big) S_{\nu}^*S_{\nu'}S_{\mu'}\big( S_v - \!\!\!\!\sum_{e\in r^{-1}(v)}\!\! S_e S_e^*\big) \xi, \eta  \Big \ra \\&=
 \delta_{\nu,\nu'} \Big \la \big( S_w - \!\!\!\!\sum_{f\in r^{-1}(v)}\!\! S_f S_f^*\big) S_{\mu'} \big( S_v - \!\!\!\!\sum_{e\in r^{-1}(v)}\!\! S_e S_e^*\big) \xi, \eta  \Big \ra \\&=
 \begin{cases}
 0 &\qif \nu \ne \nu',\\
 0 &\qif \nu = \nu'=\mu \AND v \ne w \text{ since } S_w S_v = 0,\\
 0 &\qif \nu=\nu' \AND \mu' = e'\mu'' \text{ since } \Big( S_w - \!\sum_{f\in r^{-1}(v)}\! S_f S_f^*\Big) S_{e'} = 0.
 \end{cases}
\end{align*}
}
It follows that each $\cW_v$ is wandering, and that $\cH_v$ and $\cH_w$ are orthogonal for $v\ne w$.

Evidently $\cH_v$ is an invariant subspace. The first computation in the previous paragraph shows that if $f \in E$, then
$S_f^* \big( S_v -\sum_{e\in r^{-1}(v)} S_e S_e^*\big)  = 0$. \vspace{.2ex}
Therefore $S_f^* \cW_v = \{0\}$. Moreover, if $\lambda = e\lambda'$, then $S_f^* \ol{S_\lambda \cW} = \delta_{f,e} \ol{S_{\lambda'} \cW}$. It follows that $S_f^* \cH_v \subseteq \cH_v$. Therefore, $\cH_v$ is a reducing subspace.

When $\cW_v \ne \{0\}$, it is clear from the wandering property that the set $\{S_{\mu}\xi_{v,i} : s(\mu)=v\}$ is orthonormal and
spans an invariant subspace $\cH_{v,i}$. Again it is reducing because $S_f^* \xi_{v,i}=0$ for all $f\in E$, 
and when $\mu = e\mu'$, we have $S_f^* S_{\mu}\xi_{v,i} = \delta_{f,e} S_{\mu'}\xi_{v,i} $.
A similar calculation shows that $\cH_{v,i}$ is orthogonal to $\cH_{v,j}$ when $i \ne j$.
Finally, it is also clear that  $\cH_v = \bigoplus_{i<\alpha_v} \cH_{v,i}$. 

Define a unitary $U_{v,i} : \cH_{G,v} \to \cH_{v,i}$ by setting $U_{v,i}\xi_{\mu} = S_{\mu}\xi_{v,i}$ and extending linearly. 
This unitary then implements a unitary equivalence between $\rL_{G,v}$ and $\rS|_{\cH_{v,i}}$. 
Now define $\cK= (\bigoplus_{v\in V} \cH_v)^{\perp}$, and set $\rT = \rS|_\cK$. Then 
\[
 \rS \cong \rT \oplus \bigoplus_{v\in V} \rL_{G,v}^{(\alpha_v)} .
\]
From the definition of $\cW_v$, it follows that $T_v = \sum_{e\in r^{-1}(v)}T_e T_e^*$ for every $v\in V$.
Hence, $\rT$ is fully coisometric. Also, $\rT$ is non-degenerate because $\rS$ is.

It is clear from the definition that $\cW=\bigoplus_{v\in V}\cW_v$ is a separated wandering space.
Since the complement of $\fS[\cW]$ is fully coisometric, it contains no reducing subspace generated by a wandering space.
Hence, this space is maximal with this property.
Indeed, it is unique, because any reducing subspace determined by a separated wandering space $\cN$ has the property that $S_v \cN$ is not in the range of any edge, and thus is contained in $\cW_v$.
So $\cN \subseteq \cW$.

It is clear that the definition of $\cW_v$ is canonical, and that $\alpha_v = \Dim\cW_v$ is a unitary invariant. 
Thus, the construction of $\cH_v$ and $\cK$ is also canonical.  
Hence, this decomposition is essentially unique up to choice of bases, which is a unitary invariant, and thus two 
such TCK families $\rS$ and $\rS'$ are unitarily equivalent if and only if $\alpha_v=\alpha'_v$ for $v\in V$ and the remainders $\rT$ and $\rT'$ are unitarily equivalent.
\end{proof}

Now we wish to obtain some more precise information about the fully coisometric CK family $\rT$ in this decomposition.
Suppose we have a directed graph $G$. Let 
\[ V^s:= \{ v \in V : r^{-1}(v) = \emptyset  \} \]
denote the collection of sources in $G$. 
Let  $G_1$ be the induced graph on the vertices $V_1 = V \setminus  V^s$. 
Proceeding inductively on ordinals, if $G_{\lambda}$ is given with vertex set $V_{\lambda}$, 
we define the graph $G_{\lambda+1}$ to be the graph induced from $G$ on the vertices $V_{\lambda+1} = V_{\lambda} \setminus  V_{\lambda}^s$.
When $\lambda$ is a limit ordinal, we define $G_{\lambda}$ to be the graph induced from $G$ on the vertices $V_{\lambda} = \bigcap_{\kappa < \lambda} V_{\kappa}$.

Since there are countably many vertices, there is a countable ordinal $\kappa_G$ such that $G^0 := G_{\kappa_G}$ is sourceless. 
We call $G^0$ the \emph{source elimination} of $G$.

Note that since $G^0$ has no sources, a CK family for $G^0$ can be considered as a representation of $G$
which annihilates all vertices in $V\setminus V^0$. 
However, $G^0$ may have vertices which are infinite receivers.
If such a vertex $v$ does not satisfy $S_v = \sotsum_{r(e)=v} S_eS_e^*$, then there is a wandering space as defined in Theorem~\ref{thm:Wold-decomp}.
Once this is eliminated, we obtain a fully coisometric CK family.

\begin{theorem}[Wold decomposition II] \label{thm:precise-wold}
Suppose $G = (V,E)$ is a directed graph, and let $G^0 = (V^0,E^0)$ be the source elimination of $G$. 
Suppose $\rS = (S_v, S_e)$ is a non-degenerate TCK family for $G$ on a Hilbert space $\cH$; and let
\[
\rS \cong \rT \oplus \bigoplus_{v\in V} \rL_{G,v}^{(\alpha_v)}
\]
be the Wold-decomposition for $\rS$ as in Theorem~$\ref{thm:Wold-decomp}$.
Then $\rT$ is supported on $V^0$; and $\rT = (T_v,T_e)_{v\in V^0, e\in E^0}$ is a fully coisometric CK family for $G^0$.
\end{theorem}

\begin{proof}
By Theorem \ref{thm:Wold-decomp} we know that $\rT$ is a fully coisometric CK family, so it must satisfy $T_v = \sum_{e\in r^{-1}(v)}T_e T_e^*$ for every $v\in V$. 
So in particular, for $v\in V^s$, we have that $T_v = 0$. 
This means that for any $e\in s^{-1}(V^s)$ we have $T_e =0$ since $0 = T_{s(e)} = T_e^*T_e$. 

Now, suppose towards contradiction there is some minimal ordinal $\lambda$ for which $T_w \neq 0$ for some $w \in V_{\lambda}^s$. 
If $\lambda = \lambda' + 1$ is a successor, then $T_v = 0$ for all $v\in V_{\lambda'}^s$. 
Hence, for any $e\in s^{-1}(V_{\lambda'}^s)$ we have $T_e= 0$ since $0 = T_{s(e)} =T_e^*T_e$. 
Next, since $r^{-1}(w) \subseteq s^{-1}(V_{\lambda'}^s)$, we must have that $T_w = \sum_{e\in r^{-1}(w)} T_e T_e^* = 0$, a contradiction.
If $\lambda$ is a limit ordinal, we have that $r^{-1}(w) \subseteq s^{-1}(\bigcup_{\kappa < \lambda} V_{\kappa}^s)$.
Also, for any $e\in s^{-1}(\bigcup_{\kappa < \lambda} V_{\kappa}^s)$, we have $T_e= 0$. 
It follows that $T_w = \sum_{e\in r^{-1}(w)} T_e T_e^* = 0$, again a contradiction. Thus $\rT$ is supported on $V_0$, and is a fully coisometric CK family for $G_0$.
\end{proof}

We can explicitly describe the subspace on which the full CK family lives.

\begin{proposition} \label{P:fully-coisometric space}
Suppose $G = (V,E)$ is a directed graph, and let $G^0 = (V^0,E^0)$ be the source elimination of $G$. 
Suppose $\rS = (S_v, S_e)$ is a non-degenerate TCK family for $G$ on a Hilbert space $\cH$. 
The space $\cM$ on which the fully coisometric portion of the Wold decomposition $($Theorem~$\ref{thm:precise-wold})$ acts is given by:
\[ \cM = \bigcap_{k\ge1} \sum_{|\mu|=k} S_\mu S_\mu^* \cH .\]
\end{proposition}

\begin{proof}
Write the Wold decomposition as $\rS = \rT \oplus \rL$ where $\rT$ is a fully coisometric family and $\rL$ is a direct sum of left regular representations.
Observe that since $\rT$ is a fully coisometric family on $\cM$ and multiplication is \sot\ continuous on the unit ball, we have that for $\mu \in \bF_G^+$ that
\begin{align*}
T_v = \sotsum_{|\mu|=k, r(\mu)=v} T_\mu T_\mu^*
\end{align*}
so that
\begin{align*}
  I_\cM &= \sum_{v\in V} T_v = \sum_{v\in V} \sotsum_{|\mu|=k, r(\mu)=v} T_\mu T_\mu^*
  = \sotsum_{|\mu|=k} T_\mu T_\mu^*.
\end{align*}
As $T_\mu T_\mu^* \leq S_\mu S_\mu^*$ we get that $I_\cM 
 \le \sotsum_{|\mu|=k} S_\mu S_\mu^*$, so that $\cM$ is contained in the intersection in the theorem statement.

On the other hand, it is completely different in the left regular portion. Indeed, in each $\cH_{G,v}$ we have that
\begin{align*}
 \bigcap_{k\ge1} \bigoplus_{|\mu|=k} S_\mu S_\mu^* \cH_{G,v} &=
 \bigcap_{k\ge1} \bigoplus_{|\mu|=k} \spn\{ \xi_\nu:|\nu|\ge k\} = \{0\}.
\end{align*}
It follows that the intersection is exactly $\cM$.
\end{proof}

\begin{definition}
Let $G$ be a countable directed graph. 
We say that $G$ has a \emph{source elimination scheme} (SES) if $G^0$ is the empty graph;  
i.e., in the elimination procedure described before Theorem~\ref{thm:precise-wold}, every vertex appears as a source at some step.
\end{definition}

\begin{example}
It is clear by construction that a graph with SES must be acyclic. When $G$ is a graph with finitely many vertices, it is acyclic if and only if it has SES.
However, the graph $G$ with 
\[ V=\bZ \qand  E = \{e_n: s(e_n)=n,\ r(e_n)=n+1,\ n\in \bZ\} \]
is an acyclic graph with $G^0=G$. 
\end{example}

For directed graphs with SES, it is easy to classify all free semigroupoid algebras up to unitary equivalence.

\begin{corollary}
Let $G$ be a directed graph with SES. 
Then a non-degenerate TCK family $\rS = (S_v, S_e)$ is uniquely determined up to unitary equivalence by 
the set of multiplicities $\{ \alpha_v\}_{v\in V}$ appearing in Theorem~$\ref{thm:precise-wold}$.
\end{corollary}

\begin{proof}
It is clear that $\alpha_v$ is a unitary invariant by uniqueness of Wold-decomposition. 
Since for graphs with SES by Theorem \ref{thm:precise-wold}, if 
\[
\rT \oplus \bigoplus_{v\in V} \rL_{G,v}^{(\alpha_v)}
\]
is the Wold decomposition of the TCK family $\rS$, then $\rT= (T_v, T_e) = (0,0)$.
Therefore, the unitary equivalence class is determined by the multiplicities $\{ \alpha_v\}_{v\in V}$.
\end{proof}

\section{Single vertex and cycle algebras} \label{S:special-alg}

In this section, we develop two of the main tools that will arise in the proof of the structure theorem in the next section. First, we reduce the existence of wandering vectors to the same problem for free semigroup algebras associated to every vertex of the graph. Second, we carefully examine the free semigroupoid algebras of a cycle.

\begin{assumption}
From this point on we will assume throughout the paper that our TCK families are nondegenerate in the sense that they satisfy condition (ND).
\end{assumption}

Let $\mu = e_n \dots e_1$ be a cycle. We say that a $\mu$ is \emph{irreducible} if $s(e_i) \neq s(e_1)$ for all $1\leq i \leq n$. Note that an irreducible cycle may fail to be vertex-simple, as it may revisit other vertices, except for its initial vertex, along the way. The following observation will be useful in establishing many properties of free semigroupoid algebras by accessing deep results that have been established for free semigroup algebras.

\begin{proposition} \label{P:compress_to_vertex}
Let $\fS$ be a free semigroupoid algebra for $G=(V,E)$. 
Fix $v \in V$, and let $\bFGv =\{ \mu \in \bF_G^+ : s(\mu) = r(\mu) = v \}$. 
Then $\bFGv$ is the free semigroup generated by the set of irreducible cycles $\mu \in \bFGv$.
Moreover, $\fS_v := S_v \fS |_{S_v\cH}$ is a free semigroup algebra on $S_v\cH$ for $\bF_{G,v}^+$.
\end{proposition}

\begin{proof}
Clearly every $\mu \in \bFGv$ is a product of $\lambda \in \bFGv$ which are irreducible, and $\bFGv$ is isomorphic to the free semigroup generated by the set of irreducible cycles at $v$.

Next, if $\mu$ and $\nu$ are two distinct irreducible paths, then $S_\mu$ and $S_\nu$ have orthogonal ranges.
To see this, suppose that $|\nu|\le |\mu|$ and compute $S_\nu^* S_\mu$. This product is zero unless $\mu = \nu\mu'$. 
However, this would mean that $\mu$ returned to $v$ at an intermediate point, and this does not occur as $\mu$ is irreducible. Therefore, $S_\nu^* S_\mu=0$ as claimed.

It is evident that 
\[  \fS_v = \wotclos{\spn}\{ S_\mu |_{S_v\cH} : \mu \in \bFGv \} = S_v \fS|_{S_v\cH} \]
is the algebra generated by these irreducible paths. Therefore, this is a free semigroup algebra on $S_v\cH$.
\end{proof}

\begin{remark}
Let $I_v$ be the set of irreducible cycles at $v$. The above proposition is most useful when the number of generators $|I_v|$ is at least 2. In this case, results for free semigroup algebras are readily available to us. When $I_v = \mt$, we have $\fS_v = \bC S_v$; and usually this does not cause difficulties. However, when $|I_v|=1$, this means that the transitive component of $v$ is a simple cycle. The analysis here is more complicated. See Theorem~\ref{T:cycle_structure} below for the full story about cycles.
\end{remark}

The following proposition simplifies the issue of identifying wandering vectors for $\fS$, by reducing this problem to free semigroup algebras.

\begin{proposition} \label{P:extend wandering}
Let $\fS$ be a free semigroupoid algebra generated by a TCK family $\rS=(S_v,S_e)$ for $G$ on $\cH$, and let $x\in \cH$. Then $x$ is a wandering vector for $\fS_v := S_v \fS |_{S_v\cH}$ if and only if it is a wandering vector for $\fS$ supported on $v$. 
Hence, each $S_v\cH$ is spanned by wandering vectors for $\fS_v$ if and only if $\cH$ is spanned by wandering vectors for $\fS$.
\end{proposition}

\begin{proof}
It is clear that if $x$ is a wandering vector for $\fS$ supported on $v$, then it is a wandering vector for $\fS_v$. For the converse, we need to verify that if $\mu,\nu$ are distinct paths in $\bF_G^+$ with $|\mu|\ge |\nu|$ and $s(\mu)=s(\nu)=v$, 
then $0 = \ip{S_\mu x, S_\nu x} =  \ip{S_\nu^*S_\mu x,  x}$.
However, $S_\nu^*S_\mu = 0$ unless $\mu = \nu\mu'$ and $\mu'\in \bFGv$.
Since $x$ is wandering for $\fS_v$ and $|\mu'|\ge 1$, we have $\ip{S_{\mu'}x,x}=0$ as desired.

Next, if each $S_v\cH$ is spanned by wandering vectors for $\fS_v$, each wandering vector $x$ for some $\fS_v$ is a wandering vector for $\fS$.
Therefore, $\cH = \bigoplus_{v\in V}S_v \cH$ is spanned by wandering vectors for $\fS$ in $\oplus_{v\in V}S_v\cH = \cH$. 
Conversely, if $\cH$ is spanned by wandering vectors for $\fS$, each such $x=S_v x$ is wandering for $\fS_v$ for every $v\in \supp(x)$. Hence, each $S_v\cH$ is the span of wandering vectors for $\fS_v$.
\end{proof}

To obtain our structure theorem in the next section, we need to understand free semigroupoid algebras arising from a cycle graph. Let $G=C_n$ be the cycle on $n$ vertices $v_1,\dots,v_n$ with edges $e_i$ with $s(e_i)=v_i$ and $r(e_i) = v_{i+1\!\! \pmod{n} }$ for $1 \le i \le n$.

We let $H^\infty$ denote the algebra of all bounded analytic functions on $\bD$, and let $H^\infty_0$ denote the ideal of all functions in $H^\infty$ which vanish at $0$. We introduce the notation $M_n^+(H^\infty)$ for the set of $n\times n$ matrices with coefficients in $H^\infty$
such that the matrix entries $f_{ij}\in H^\infty_0$ when $i>j$ (i.e., the strictly lower triangular entries lie in $H^\infty_0$). This has an operator algebra structure obtained from considering this as a subalgebra of $M_n(L^\infty(m))$, where $m$ is Lebesgue measure on the circle $\bT$.

\begin{lemma}\label{L:cycle_wand}
Consider the left regular TCK family $\rL=(L_v,L_e)$ of the cycle $C_n$ on the summand $\cH_{C_n,v_1}$. There is an identification of $\cH_{C_n,v_1}$ with $\bC^n \otimes \ell^2$ such that 
$L_{e_i} \cong E_{i,i+1} \otimes I$ for $1 \le i < n$ and $L_{e_n} \cong E_{n,1} \otimes U_+$, where $U_+$ is the unilateral shift. With this identification, $\fL_{C_n,v_1}$ is unitarily equivalent to the subalgebra of $M_n(\cB(\ell^2))$ generated by $\big[f_{ij}(U_+) \big]$ where $\big[ f_{ij} \big] \in M_n^+(H^\infty)$.

The algebra $\fL_{C_n}$ is completely isometrically isomorphic and weak-$*$ homeomorphic to $M_n^+(H^\infty)$, 
and the canonical restriction maps to $\cH_{C_n,v_i}$ are completely isometric weak-$*$ homeomorphic isomorphisms.
\end{lemma}

\begin{proof}
The paths for $C_n$ with $s(\mu)=v_1$ are just indexed by their length, say $\mu_k$ for $k\ge0$. 
Moreover, $r(\mu_k) = v_{k \!\!\pmod n +1}$. 
So the spaces $\cH_{v_i} := L_{v_i}\cH_{C_n, v_1}$ have orthonormal basis $\{ \xi_{i,s} =\xi_{\mu_{i-1+ns}} : s \ge 0 \}$. 
Thus, there is a natural identification of each $\cH_{v_i}$ with $\ell^2$. 
With this basis, we readily identify $L_{v_i}$ with $E_{ii}\otimes I$, and $L_{e_i}$ with $E_{i,i+1} \otimes I$ for $1 \le i < n$ while $L_{e_n}$ is identified with $E_{n,1} \otimes U_+$. 

A calculation then shows that the simple cycle starting and ending at a vertex $v_i$ corresponds to $E_{ii}\otimes U_+$. This means that if $p \in \bC[z]$ is a polynomial such that $E_{ij} \otimes f(U_+)$ with $i > j$ is in the algebra generated by $E_{i,i+1} \otimes I$ for $1 \le i < n$ and $E_{n,1} \otimes U_+$, then it must be the case that $p$ vanishes at $0$ (i.e. no free coefficient). Thus, it follows that the \wot-closed algebra generated by $E_{i,i+1} \otimes I$ for $1 \le i < n$ and $E_{n,1} \otimes U_+$ is the \wot-closed generated by $[p_{ij}(U_+)]$ where $p_{ij} \in \bC[z]$ with $p_{ij}$ vanishing at $0$ when $i >j$. 

From this, it follows that for every polynomial matrix $\big[ p_{ij} \big]$ in $M_n^+(\bC[z]) \subseteq M_n^+(H^{\infty})$, there is a corresponding polynomial in $L_{v_i}$ and $L_{e_i}$ that is mapped to $\big[ p_{ij}(U_+) \big]$. 
Since any element $[f_{ij}] \in M_n^+(H^\infty)$ is the weak-$*$ Cesaro limit of polynomial matrices, $[f_{ij}]$ must be mapped to $\big[ f_{ij}(U_+) \big]$ under the unitary equivalence. 
As the algebra of matrices $\big[ f_{ij}(U_+) \big]$ for $[f_{ij}] \in M_n^+(H^\infty)$ is clearly a \wot-closed subspace of $M_n(\cB(\ell^2))$, we see that $\fL_{C_n,v_1}$ is unitarily equivalent to the algebra of matrices $\big[ f_{ij}(U_+) \big]$ for $[f_{ij}] \in M_n^+(H^\infty)$.

We next show that the algebra of matrices $\big[ f_{ij}(U_+) \big]$ for $[f_{ij}] \in M_n^+(H^\infty)$ is isomorphic to $M_n^+(H^\infty)$. Since the compression of $\bC^n \otimes \ell^2(\bZ)$ to $\bC^n \otimes \ell^2$ is completely contractive, 
it is evident that the norms of the matrix operators on $\bC^n \otimes \ell^2(\bZ)$ are at least as large. On the other hand, since $\bC^n \otimes \ell^2(\bZ)$ is the closed union of $\bC^n \otimes \ell^2(\bN_0 - k)$ as $k$ increases,
and the restriction to each of these invariant subspaces is unitarily equivalent to $\fL_{C_n,v_1}$, one sees that
the limit algebra is in fact completely isometrically isomorphic and weak-$*$ homeomorphic to $\fL_{C_n,v_1}$ via the canonical map that takes generators to generators.

Now identify $\ell^2(\bZ)$ with $L^2(\bT)$ by the Fourier transform which implements a unitary equivalence between $U$ to $M_z$. 
The von Neumann algebra $W^*(M_z)$ is the algebra $\{M_f:f\in L^\infty(\bT) \} \simeq L^\infty(\bT)$, where this is a normal $*$-isomorphism.
Then $\fL_{C_n,v_j}$ is unitarily equivalent to the subalgebra of $M_n(L^\infty(\bT))$ corresponding to $M_n^+(H^\infty)$.
So this map is completely isometric and a weak-$*$ homeomorphism.
Hence, this yields a completely isometric and weak-$*$ homeomorphic isomorphism between $\fL_{C_n,v_1}$ and $\fL_{C_n,v_1}$.

Finally, since for each $1\leq i \leq n$ we similarly have a completely isometric and weak-$*$ homeomorphic isomorphism $\phi_i : M_n^+(H^{\infty}) \rightarrow \fL_{C_n,v_i}$, the rest follows from Corollaries \ref{C:fL_{G,F}} and \ref{C:lr_max_element}.
\end{proof}

A similar analysis shows that a general CK family depend only on a single unitary.

\begin{lemma}\label{L:cycle_CK}
Let $\rS=(S_v,S_e)$ be a CK family for $C_n$ on a Hilbert space $\cH$. 
Then there is an identification of $\cH$ with $\bC^n \otimes \cK$ and a unitary operator $V$ on $\cK$ such that 
$S_{e_i} \cong E_{i,i+1} \otimes I$ for $1 \le i < n$ and $S_{e_n} \cong E_{n,1} \otimes V$. 
\end{lemma}

\begin{proof}
This follows in a similar manner to the previous lemma. We denote $\cH_{v_i} = S_{v_i}\cH$, and use $S_{e_i}$ to identify $\cH_{v_i}$ with $\cH_{v_{i+1}}$ for $1\leq i < n$. The only difference now is that the unitary edge $S_{e_n}$ mapping $\cH_{v_n}$ onto $\cH_{v_1}$ will be identified with an arbitrary unitary operator $V$ instead of the unilateral shift.
\end{proof}

We can split the spectral measure $\mu$ of $V$ into a portion which is absolutely continuous $\mu_a$ with respect to Lebesgue measure $m$ and one that is singular $\mu_s$ so that we get a decomposition $V \cong V_a \oplus V_s$ on $\cK = \cK_a \oplus \cK_s$. This yields a corresponding decomposition of the CK family $\rS$ on $(\bC^n \otimes \cK_a) \oplus (\bC^n \otimes \cK_s)$ into $\rS \cong \rS_a \oplus \rS_s$ where $\rS_a$ is its absolutely continuous part and $\rS_s$ is its singular part. This allows us to describe general free semigroupoid algebras for a cycle.

\begin{theorem}\label{T:cycle_structure}
Let $\rS=(S_v,S_e)$ be a TCK family for $C_n$. 
Write its Wold decomposition as $\rS = \rT \oplus \sum_{i=1}^n \oplus \rL_{C_n, v_i}^{(\alpha_i)}$ where $\rT$ is CK.
Let $V$ be the unitary from Lemma~$\ref{L:cycle_CK}$ in the structure of $\rT$. Decompose $V=V_a \oplus V_s$ according to the spectral measure $\mu$ of $V$, where $V_a$ is a unitary with spectral measure $\mu_a$ absolutely continuous with respect to Lebesgue measure $m$, and $V_s$ is a unitary with spectral measure $\mu_s$ singular with respect to Lebesgue measure $m$. Then there are two situations,
\begin{enumerate} [label=\normalfont{(\arabic*)}]
\item If $\rS$ is a CK family $($i.e.\ $\alpha_i = 0$ for $1\le i \le n )$ and $m \not \ll \mu$, then $\fS$ is isomorphic to the von Neumann algebra $M_n(L^\infty(\mu))$.
\item If $\sum_{i=1}^n \alpha_i \ge 1$ or $m \ll \mu$, then $\fS$ is isomorphic to $\fL_{C_n} \oplus M_n(L^\infty(\mu_s))$.
\end{enumerate}
All of these isomorphisms are completely isometric weak-$*$ homeomorphic isomorphisms.

Furthermore, in case $(1)$ there are no wandering vectors, and in case $(2)$ the wandering vectors span the left regular part and the absolutely continuous portion of the CK part.
\end{theorem}

\begin{proof}
Suppose $\rS$ is a CK family with unitary $V$ as above with spectral measure $\mu = \mu_a + \mu_s$. A theorem of Wermer \cite{Wermer} shows that the \wot-closed algebra $W(V)$ generated by $V$ is completely isometrically and weak-* homeomorphically isomorphic to $L^\infty(\mu)$ if $m \not\ll \mu$,
and $W(V)$ is completely isometrically and weak-* homeomorphically isomorphic to $H^\infty \oplus L^\infty(\mu_s)$ if $m\ll\mu$.
Thus, when $m \not\ll \mu$, up to the identification of Lemma~$\ref{L:cycle_CK}$ it follows that for each function $f \in L^\infty(\mu)$ we have that the operators $E_{i,j}\otimes f(V)$ all belong to $\fS$ so that $\fS \simeq  M_n(L^\infty(\mu))$.

When $m\ll\mu$, we have that $W(V) \cong H^\infty(V_a) \oplus W^*(V_s)$, where $V \cong V_a\oplus V_s$ is the decomposition of $V$ into its absolutely continuous and singular parts, and $W^*(V_s)$ is the von-Neumann algebra generated by $V_s$. Since $\mu_a$ is equivalent to $m$, it follows that for the bilateral shift $U$ we have that $V_a^{(\infty)}$ and $U^{(\infty)}$ are unitarily equivalent. 

The form of generators of $\fS$ as in Lemma~$\ref{L:cycle_CK}$ yields that for any polynomial $p \in \bC[z]$ such that $E_{ij} \otimes p(V)$ is an element of $\fS$ with $i >j$, we must have that $p$ vanishes at $0$. Thus, we see that $[p_{ij}(V)]$ generates $\fS$ when $p_{ij} \in \bC[z]$ are polynomials that vanish at $0$ when $i >j$. This identification preserves the direct sum decomposition in the previous paragraph, and we get that $\fS \cong M_n^+(H^{\infty}(V_a)) \oplus M_n(L(\mu_s))$. By Lemma~\ref{L:cycle_wand} and as $V_a^{(\infty)}$ and $U^{(\infty)}$ are unitarily equivalent, it follows that $\fS \simeq \fL_{C_n} \oplus M_n(L^\infty(\mu_s))$.

Next, consider the case in which there is a left regular part and a CK part. Wermer's analysis \cite{Wermer} of the \wot-closed algebra of an isometry shows that the \wot-closed algebra $W(U_+ \oplus V)$ is isomorphic to $H^\infty \oplus L^\infty(\mu_s)$. By using similar arguments as before as well as Lemma~\ref{L:cycle_wand} again, we see that still $\fS \simeq \fL_{C_n} \oplus M_n(L^\infty(\mu_s))$.

In case (1), the algebra is self-adjoint and thus cannot have wandering vectors. 
Indeed, if there was a wandering vector, it would yield an invariant subspace on which the algebra $\fS$ is unitarily equivalent to a nonself-adjoint algebra. 
This will contradict the previous analysis.

In case (2), the argument from the previous paragraph shows that the wandering vectors are contained in the left regular and absolutely continuous parts of the space, which we identified with $\bC^n \otimes \cK_l$ and $\bC^n \otimes \cK_a$ respectively. 
In order to show that $\bC^n \otimes (\cK_l \oplus \cK_a)$ is the closed span of wandering vectors, we turn to Proposition~\ref{P:extend wandering}, which reduces the problem to showing for each $1\leq i \leq n$ that the free semigroup algebra $\fS_{v_i}$ has $\cK_l \oplus \cK_a$ as the closed span of wandering vectors.
Clearly $\cK_l$ is spanned by wandering vectors.
If $\mu_a \approx m$, then by the proof of \cite[Theorem II.1.3]{Dav96}, and perhaps after splitting and unifying subsets, we can arrange for
\[ V_a \cong \bigoplus_{i\ge0} M_{z,A_i} \quad\text{where}\quad \bT = A_0 \supseteq A_i \ \ \forall i>0 , \]
where $M_{z,A}$ is multiplication by $z$ on $L^2(A)$.  Then for $i>0$ 
\[ L^2(A_0) \oplus L^2(A_i) \supseteq L^2(A_i^c) \oplus L^2(A_i) \cong L^2(\bT) .\]
Thus, $L^2(A_i^c) \oplus L^2(A_i)$ is a reducing subspace on which the restriction of $V$ is unitarily equivalent to $M_z$, and is thus spanned by wandering vectors. Since the union of subspaces $\{L^2(A_i^c) \oplus L^2(A_i)\}_{i\geq 0}$ has dense span in $\cK_a$ (up to the above unitary identification), we see that $\cK_a$ is also spanned by wandering vectors.

In the remaining case, some $\alpha_i\ge1$ and $\mu_a \approx m|_A$ for some measurable subset $A \subseteq \bT$.
Thus $\fS_{v_i}$ is generated by an isometry $V$ which, modulo multiplicity, has the form $V = U_+ \oplus M_{z,A} \oplus V_s$ on $\cH = H^2 \oplus L^2(A) \oplus L^2(\mu_s)$. 
In this case, the wandering vectors span a dense subset of $H^2 \oplus L^2(A)$. 
This is a well-known fact, but we review the argument. Denote $S := U_+ \oplus M_{z,A}$, and let $\ep>0$ and $k\in\bZ$. 
By strong logmodularity of $H^\infty$ (see \cite[Theorem II.4.5]{Garnett}), 
there is a function $h\in H^\infty\subseteq H^2$ such that $|h|^2 = \upchi_{A^c} + \ep \upchi_A$.
Let $x = h \oplus (1-\ep)^{1/2}z^k \upchi_A \oplus 0$.
Then if $m=l+p$, and $p>0$,
\begin{align*}
 \ip{S^m x, S^l x} &= \ip{S^px,x} \\
 &= \bip{z^ph \oplus (1-\ep)^{1/2}z^{k+p} \upchi_A \oplus 0, h \oplus (1-\ep)^{1/2}z^k \upchi_A \oplus 0}\\
 &= \int z^p |h|^2 + (1-\ep) \int z^p \upchi_A = \int z^p = 0.
\end{align*}
It follows that the wandering vectors span $H^2 \oplus L^2(A)$. 
\end{proof}

\begin{example}
Consider two CK families on $C_n$ as in Lemma~\ref{L:cycle_CK} with unitaries $M_{z,A}$ and $M_{z,A^c}$,
where $A$ is a subset of $\bT$ with $0 < m(A) < 1$ and $A^c=\bT\setminus A$. 
Then by Theorem~\ref{T:cycle_structure}(1), they both generate free semigroupoid algebras which are von Neumann algebras. However, their direct sum forms the CK family with unitary $M_z$, which has spectral measure $m$. This yields a representation with wandering vectors and is hence \emph{nonself-adjoint}. 

This shows that when the graphs are cycles, there are nonself-adjoint free semigroupoid algebra with complementary reducing subspaces so that the cutdown by each is a von Neumann algebra with absolutely continuous associated isometries $M_{z,A}$ and $M_{z,A^c}$ (see Section~\ref{S:abs cont}).
\end{example}

\section{The structure theorem} \label{S:stucture}

In this section, we establish a structure theorem for free semigroupoid algebras of a graph $G$ akin to the one in \cite{DKP} for free semigroup algebras, along with a complete identification of wandering vectors. For a TCK family $\rS=(S_v, S_e)$ of a directed graph $G$ acting on a Hilbert space $\cH$, in this section we denote by
\[
\fA_S : = \overline{\Alg}^{\| \cdot \|} \{ S_{\mu} :\mu \in \bF_G^+  \} 
\]
and 
\[
\fT_S : = \ca(\{ S_{\mu} :\mu \in \bF_G^+  \})
\]
the norm closed algebra and C*-algebra generated by $S$ inside $B(\cH)$, respectively.

\begin{proposition} \label{P:complete_quotient}
Let $G$ be a countable directed graph, and let $\fS$ be the free semigroupoid algebra generated by a TCK family $\rS=(S_v,S_e)$. 
Suppose that $\phi: \fS \rightarrow \fL_G$ is a \wot-continuous and completely contractive homomorphism such that $\phi(S_{\lambda}) = L_{\lambda}$ for all paths $\lambda \in \bF_G^+$.
Then $\phi$ is a surjective complete quotient map. Hence, $\fS / \ker \phi$ is completely isometrically isomorphic and weak-$*$-homeomorphic to $\fL_G$.
\end{proposition}

\begin{proof}
By the universality of $\cT(G)$, we have a surjective $*$-homomorphism $\pi_\rS : \cT(G) \rightarrow \fT_S$ given by $\pi_\rS(L_{\lambda}) = S_{\lambda}$ for all $\lambda \in \bF_G^+$. 
Hence, if $X:= \sum_{\ell} A_{\ell}\otimes S_{\lambda_{\ell}}$ is a finite sum in $M_p(\fA_S)$, then $\|X\| \leq \|A\|$ where $A = \sum_{\ell} A_{\ell}\otimes L_{\lambda_{\ell}} \in M_p(\cT_+(G))$.

For $A\in M_p(\fL_G)$ with $\|A\|\le1$, consider its Fourier series expansion $\sum A_{\lambda} \otimes L_{\lambda}$, 
and the Ces\`aro means $\Sigma_k(A) = \sum_{|\lambda|<k}\big(1 - \frac{|\lambda|}{k}\big) A_{\lambda} \otimes L_{\lambda}$. 
If $F$ is a finite subset of $V$, recall that $L_F := \sum_{v\in F} L_v$. By Remark~\ref{R:Cesaro_in_T_+(G)}, the net of operators $\Sigma_k(A) L_F$ belongs to the unit ball of $\cT_+(G)$,
and converges to $A$ in the \wot\ topology.
Define a net 
\[
 X_{k,F} = \sum_{\substack{|\lambda|<k\\s(\lambda) \in F}} \big(1 - \frac{|\lambda|}{k}\big) A_{\lambda} \otimes S_{\lambda} .
\]
Then, by the first paragraph of the proof, we have 
\[ \|X_{k,F}\| \leq \| \Sigma_k(A) L_F \| \le \| A \| .\]
Moreover, $\phi^{(p)}(X_{k,F}) = \Sigma_k(A)L_F$.

Since the ball of radius $\|A\|$ in $M_p(\fS)$ is \wot-compact, there is a cofinal subnet $X_{\alpha}$ of $\{X_{k,F}\}$ that converges \wot\ to some element $X\in M_p(\fS)$. 
By the \wot-continuity of $\phi^{(p)}$, we see that $\phi^{(p)}(X) = A$.
By complete contractivity of $\phi$, we have  
\[
\|X \| \leq \|A\| = \|\phi^{(p)}(X) \| \leq \|X\| .
\]
Thus the map $\phi$ is surjective, and the map 
\[ \tilde\phi : \fS/\ker\phi \to \fS_G \]
is a completely isometric isomorphism.

Since $\phi$ is \wot-continuous and both $\fS$ and $\ker \phi$ are weak-$*$ closed, we see that $\widetilde{\phi}$ is weak-$*$ to \wot\ continuous. 
However, the weak-$*$ and \wot\ topologies on $\fL_G$ coincide by \cite[Corollary 3.2]{JK2005}, and so $\widetilde{\phi}$ is weak-$*$ to weak-$*$ continuous. 
Using the fact that $\widetilde{\phi}$ is an isometry, the usual Banach space pre-duality arguments show that the inverse of $\widetilde{\phi}$ is also weak-$*$ to weak-$*$ continuous.
Therefore, $\widetilde{\phi}$ is a weak-$*$ homeomorphism.
\end{proof}

We recall that for a TCK family $\rS=(S_v,S_e)$, the support of $\rS$ is given by $\supp(\rS) = \{ v\in V \ | \ S_v\neq 0 \}$,
which is a directed subset of $V$; 
and  we denote by $G[\rS] := G[\supp(\rS)]$ the subgraph induced from $G$ on the support of $\rS$.

\begin{definition}
We say that a free semigroupoid algebra $\fS$ generated by a TCK family $\rS=(S_v,S_e)$ is {\em analytic} (or \emph{type L}) if there is a completely isometric isomorphism and weak-$*$-homeomorphism $\phi : \fS \rightarrow \fL_{G[\rS]}$ such that $\phi(S_{\lambda}) = L_{\lambda}$ for any $\lambda \in \bF^+_{G[\rS]}$.
\end{definition}

By Proposition \ref{P:complete_quotient} it suffices to require that $\phi$ above only be a \wot-continuous completely contractive \emph{injective} homomorphism. Moreover, if $G$ is a transitive graph, then the only non-empty directed subgraph of $G$ is $G$ itself. So if $\fS$ is of analytic, then it must be isomorphic to $\fL_G$. 

For a wandering subspace $\cW \subseteq \cH$ of a free semigroupoid algebra $\fS$, we will denote by $G[\cW]$ the graph $G[\supp(\cW)]$. For a wandering vector $x\in \cH$, we will also denote $G[x]$ for the graph $G[\{x\}]$. To streamline the presentation of our results, we make the following definition.

\begin{definition}
Let $G$ be a directed graph and let $\fS$ be the free semigroupoid algebra of a TCK family $\rS=(S_v,S_e)$ on a Hilbert space $\cH$. We say that a vertex is \emph{wandering} for $\fS$ if there is a non-zero wandering vector for $\fS$ supported on $v$. We will denote by $V_w$ the collection of all wandering vertices for $\fS$.
\end{definition}

When $x\in \cH$ is a wandering vector supported on $v$ for a free semigroupoid algebra $\fS$, and $S_{\lambda}x \neq 0$ for a path $\lambda \in \bF_G^+$ with $s(\lambda)=v$, then $S_{\lambda}x$ is a wandering vector supported on $r(\lambda)$. Hence, it is easy to see that $V_w$ is directed.

\begin{corollary} \label{cor:type-L-wandering}
Let $\fS$ be a free semigroupoid algebra of a graph $G$ generated by a TCK family $\rS=(S_v,S_e)$ on $\cH$.
If $\fS$ has a maximal wandering subspace $\cW \subseteq \cH$, then there is a \wot-continuous completely contractive homomorphism $\phi: \fS \rightarrow \fL_{G[\cW]}$ obtained via the restriction map to the subspace $\fS[\cW]$ which is surjective, and is a complete quotient map. Hence, the induced map $\widetilde{\phi} : \fS / \ker \phi \rightarrow \fL_{G[\cW]}$ is a completely isometric weak-$*$ homeomorphic isomorphism.

Moreover, when $\cM$ is the closed span of all wandering vectors for $\fS$, then $\fS|_{\cM}$ is completely isometrically isomorphic and weak-$*$-homeomorphic to $\fL_{G[V_w]}$. In particular, if $\cH$ is the closed span of wandering vectors, then $\fS$ is analytic type.
\end{corollary}

\begin{proof}
Since $\cW$ is wandering, $\fS[\cW] = \bigoplus_{v\in \supp(\cW)} \fS[S_v \cW]$. 
Let $\cM = \fS[\cW]$, and for $v\in \supp(\cW)$, fix a unit vector $x_v$ in $S_v\cW$. 
Let $U_v: \cH_{G,v} \rightarrow \cM$ be given by $U_v \xi_{\lambda} = S_{\lambda}x_v$ for $\lambda\in \bF_G^+$ with $s(\lambda)=v$, which is an isometry. 
Then $U := \bigoplus_{v\in \supp(\cW)}U_v$ is an isometry of $\cH_{G,\supp(\cW)} = \bigoplus_{v\in\supp(\cW)} H_{G,v}$ into $\cH$. 
The map $\phi(A) = U^*AU$ is a \wot-continuous completely contractive homomorphism that takes $S_{\lambda}$ to $L_{\lambda}|_{\cH_{G,\supp(\cW)}}$. 
By Proposition \ref{P:complete_quotient}, we get that $\phi : \fS \rightarrow \fL_{G[\cW]}$ and the induced map $\widetilde{\phi}$ have the properties required in the statement.

Next, we construct a map $\psi: \fS|_{\cM} \rightarrow \fL_{G[V_w]}$ given by $\psi(S_{\lambda}|_{\cM}) = L_{\lambda}$ for $\lambda \in \bF_G^+$ that will be completely contractive and \wot~continuous. 
Indeed, if $\{x_i\}_{i \in I}$ is the collection of all wandering vectors, we have a completely contractive and \wot~continuous map $\phi_i : \fS|_{\cM} \rightarrow \fL_{G[\supp(x_i)]}$. 
After restricting, Corollary \ref{C:fL_{G,F}} yields that for each $v\in V_w= \bigcup_{i\in I} \supp(x_i)$ we have a well-defined completely contractive and \wot~continuous map $\phi_v : \fS|_{\cM} \rightarrow \fL_{G,v}$ that maps $S_{\lambda}|_{\cM}$ to $L_{\lambda}$ for each $\lambda \in \bF_G^+$. 
Since each $\cH_{G,v}$ is reducing for $\fL_{G[V_w]}$, we see that $\psi = \oplus_{v\in V_w} \phi_v : \fS|_{\cM} \rightarrow \fL_{G[V_w]}$ is a well-defined completely contractive \wot~continuous map.

Thus, by Proposition \ref{P:complete_quotient} it will suffice to show that $\psi$ is injective. For any $T\in \fS|_{\cM}$ and a wandering vector $x\in \cH$, let $T_x:= T|_{\fS[x]}$ be the restriction. 
If $\psi(T) = 0$, it is clear that $\psi(T)|_{\cH_{G,\supp(x)}} = 0$. So by the first paragraph we know that $T_x$ is unitarily equivalent to $\psi(T)|_{\cH_{G,\supp(x)}}$; whence $T_x = 0$. As this was done for an arbitrary wandering vector $x\in \cM$, we conclude that $T|_{\cM} = 0$.
Therefore, the map $\psi$ is an \emph{injective} \wot-continuous completely contractive homomorphism. 
Thus, by Proposition \ref{P:complete_quotient} it is a completely isometric isomorphism that is also a weak-$*$-homeomorphism.
\end{proof}

\begin{definition}
For a free semigroupoid algebra $\fS$ of a graph $G$ generated by a TCK family $\rS= (S_v,S_e)$, we define
\[
\fS_0:= \overline{\langle S_e \rangle}^{\wot} = \wotclos{\spn} \{ S_{\lambda} :|\lambda|\geq 1  \}
\]
to be the \wot-closed ideal generated by the edge partial isometries. Similarly, we define
\[
 \fS_0^k := \overline{\langle S_\mu  : |\mu|=k \rangle}^{\wot} = \wotclos{\spn} \{ S_{\lambda} :|\lambda|\geq k  \} .
\]
In particular, we have the ideals $\fL_{G,0}^k$ in $\fL_G$.
\end{definition}

\begin{lemma} \label{L:SGA_0^k}
For $k\ge1$,  
\[
 \fS_0^k = \big\{ \sotsum_{|\lambda|=k} S_{\lambda} A_{\lambda} : A_{\lambda}\in \fS  \big\}.
\]
Each element of $\fS_0^k$ has a unique such decomposition with $A_{\lambda} = S_{s(\lambda)} A_{\lambda}$.
\end{lemma}

\begin{proof}  
The algebraic ideal $\cI_k$ generated by $\{ S_{\lambda}: |\lambda| =k\}$ consists of polynomials $\sum_{|\mu|\ge k} a_\mu S_\mu$ in the sense that $a_\mu=0$ for all but finitely many $\mu$.
By combining terms for $\mu$ of the form $\mu=\lambda\mu'$ for $|\lambda| = k$, we get an expression $B= \sum_{|\lambda|=k} S_\lambda A_\lambda$
where $A_\lambda$ is a polynomial in the (not necessarily norm-closed) algebra generated by $\rS$. 
Moreover, one computes that $A_\lambda = S_\lambda^* B$.

Given $A \in \fS_0^k$, there is a net $A_{\alpha} \in \cI_k$ \wot-converging to $A$.
We write $A_\alpha = \sum_{|\lambda|=k} S_\lambda A_{\alpha,\lambda}$.
Then define
\[
 A_{\lambda} := S_{\lambda}^*A = \wotlim S_{\lambda}^*A_{\alpha} = \wotlim A_{\alpha,\lambda}.
\]
Thus $A_\lambda$ lie in $\fS$.
The projection $P = \sotsum_{|\lambda|=k} S_{\lambda}S_{\lambda}^*$ satisfies $PS_{\mu} = S_{\mu}$ for all $\mu \in \bF_G^+$ with $|\mu|\geq k$.
It follows that $PA_\alpha = A_\alpha$ for all $\alpha$, and therefore $PA = A$.
Hence,
\[
A = PA = \sotsum_{|\lambda|=k} S_{\lambda}S_{\lambda}^*A = \sotsum_{|\lambda|=k} S_{\lambda}A_{\lambda} .
\]
The $A_{\lambda}$ are uniquely determined by the equation $A_{\lambda} = P_{s(\lambda)} A_{\lambda} = S_{\lambda}^*A$.
\end{proof}

\begin{corollary} \label{C:shift SGA_0^k}
If $|\mu|= n \le k$, then $S_\mu^* \fS_0^k \subseteq \fS_0^{k-n}$, including $k=n$ where we interpret $\fS_0^0$ to be $\fS$.
\end{corollary}

\begin{proof}
Note that if $|\lambda|=k$, then
\[
S_\mu^* S_\lambda = \begin{cases}S_{\lambda'} &\IF \lambda = \mu\lambda'\\0&\text{ otherwise}.\end{cases}
\] 
Therefore if $A$ in $\fS_0^k$ is written $A = \sotsum_{|\lambda|=k} S_{\lambda} A_{\lambda}$, we obtain that
\[ S_\mu^*A = \sotsum_{|\lambda'|=k-n, \ r(\lambda')=s(\mu)} S_{\lambda'} A_{\mu\lambda'} \]
which lies in $\fS_0^{k-n}$.
\end{proof}

\begin{theorem} \label{T:dichotomy}
Let $\fS$ be the free semigroupoid algebra of a TCK family $\rS=(S_v,S_e)$ on a Hilbert space $\cH$ for a graph $G$. 
Let $V_0 = \{S_v : S_v\in \fS_0\}$. Then $V_w = V \setminus V_0$, and we have a completely isometric isomorphism and weak-$*$ homeomorphism
\[ \fS/ \fS_0 \cong \ell^\infty(V_w) .\]
Moreover, if we denote by $P_0 = \sotsum _{v\in V_0}S_v$, then we have that
\begin{enumerate} [label=\normalfont{(\arabic*)}]
\item
$P_0^{\perp} \cH$ is invariant for $\fS$.
\item
$P_0\fS P_0$ is self-adjoint.
\end{enumerate}
\end{theorem}

\begin{proof}
First, we verify that $V_w = V \setminus V_0$. If there's a wandering vector supported at $v$, we will show that $v \in V\setminus V_0$. Indeed, by Corollary \ref{cor:type-L-wandering} there is a completely contractive \wot~continuous map $\phi$ from $\fS$ to $\fL_{G,v}$ such that $\phi(S_v) = L_v$. If on the other hand $v\in V_0$, then we have that $S_v \in \cap_{k\geq 1} \fS_0^k$. By Lemma \ref{L:SGA_0^k} we may write $S_v = \sum_{|\lambda| = k} S_{\lambda}A_{\lambda}$ with $A_{\lambda} \in \fS$. Then 
$$
L_v = \phi(S_v) = \wotlim \sum_{|\lambda| = k}L_{\lambda} \phi(A_{\lambda}) \in \fL_{G,0}^k.
$$
However, as $\bigcap_{k\geq 1} \fL_{G,0}^k = \{0\}$ we get in contradiction that $L_v = 0$.  

Conversely, for each $v\in V$, by Proposition~\ref{P:compress_to_vertex} the algebra $\fS_v:= S_v \fS|_{S_v\cH}$ is a free semigroup algebra with generators indexed by the irreducible cycles around $v$. 
By Proposition~\ref{P:extend wandering} there is a wandering vector for $\fS$ supported on $v$ if and only if the free semigroup algebra $\fS_v:= S_v \fS|_{S_v\cH}$ has a wandering vector. 

If $v\in V\setminus V_0$ and there are no cycles at $v$, then any vector in $S_v \cH$ is wandering for $\fS_v$, and hence is wandering for $\fS$ and supported at $v$ by Proposition~\ref{P:extend wandering}.
If $v\in V\setminus V_0$ and there are at least 2 irreducible cycles at $v$, then $\fS_v$ is a free semigroup algebra on 2 or more generators. Since $\fS_v \neq \fS_{v,0}$, $\fS_v$ has a wandering vector by Kennedy's result \cite[Theorem 4.12]{Kennedy2011}; and again this provides a wandering vector for $\fS$ supported at $v$.

Finally, if $v\in V\setminus V_0$ and there is exactly one irreducible cycle at $v$, then the compression to this cycle is a free semigroupoid algebra $\fC$ classified by Theorem~\ref{T:cycle_structure}. Since $\fC \ne \fC_0$, we see that $\fC$ is not a von Neumann algebra, and hence lies in cases (2) of Theorem~\ref{T:cycle_structure}, which yields the desired wandering vector.
Therefore, we can find a wandering vector for $\fS$ for every $v\in V\setminus V_0$, and hence $V_w = V \setminus V_0$.

Now consider $(1)$. Clearly, $S_v P_0\cH \subseteq P_0 \cH$ for all $v\in V$. Next, if $e\in E$ is such that $r(e) \notin V_0$, then $S_e^*P_0 \cH = \{0\} \subseteq P_0 \cH$. 
So assume $r(e) \in V_0$. As $V_w$ is directed, we must have that $s(e)\in V_0$, so that 
\[ S_e^*P_0 \cH = P_{s(e)}S_e^* P_0 \cH \subseteq P_{s(e)} \cH \subseteq P_0 \cH .\]
Thus $P_0\cH$ is invariant for $\fS^*$; whence $P_0^\perp\cH$ is invariant for $\fS$.

For item $(2)$, let $e\in E$ be such that $P_0S_e P_0 = S_e$. 
In particular, this means that $r(e) \in V_0$, so that $S_{r(e)} \in \fS_0$. 
Hence,
\[ S_e^* = S_{s(e)} S_e^* S_{r(e)} \in P_0 (S_e^* \fS_0) P_0 \subseteq P_0 \fS P_0 \]
by Corollary~\ref{C:shift SGA_0^k}.
Therefore, $P_0 \fS P_0$ is self-adjoint.

Finally, observe that 
\[ \fS = \wotclos{\spn} \{ S_{\lambda} :|\lambda|\geq 0  \} = \wotclos{\ell^{\infty}(V\setminus V_0) + \fS_0} .\]
Define a map
\[ \Psi(A) = P_0^\perp \Phi_0(A) = \wotsum_{v\in V\setminus V_0} a_v S_v \]
where $\Phi_0$ is the 0-th order Fourier coefficient map. 
This is clearly a completely contractive weak-$*$ continuous homomorphism onto $\ell^{\infty}(V \setminus V_0)$.
The kernel contains $\spn\{S_{\lambda} :|\lambda|\ge 1\}$ and thus contains $\fS_0$.
Also, for any polynomial $A$, it is clear that $A - \Psi(A)$ belongs to $\fS_0$.
Taking limits shows that $A - \Psi(A)$ belongs to $\fS_0$ in general.
Therefore, $\fS = \ell^{\infty}(V \setminus V_0) + \fS_0$ and $\ker\Psi = \fS_0$.
The natural injection of $\ell^\infty(V\setminus V_0)$ into $\fS$ is a completely isometric section of this map.
It follows that $\Psi$ is a complete quotient map, and induces a completely isometric weak-$*$ homeomorphism of $\fS/ \fS_0$ and $\ell^\infty(V_w)$.
\end{proof}

\begin{remark}
To get that $V_w = V \setminus V_0$ in the above proof we apply a deep result of Kennedy \cite{Kennedy2011}. We note that results in \cite{Kennedy2011,Kennedy2013} are true for free semigroup algebras with $\aleph_0$ generators. 
It may appear on a casual reading that $d \geq 2$ must be finite, but a careful look at the proofs and the cited references shows that they work for $d=\aleph_0$ as well.
\end{remark}

The above result can be considered as a generalization of the Dichotomy in \cite[Theorem 1.5]{DKP}. 
A bona fide dichotomy is obtained when the graph $G$ is transitive.

\begin{corollary}[Dichotomy] \label{C:dichotomy}
Let $\fS$ be a free semigroupoid algebra of a \emph{transitive} graph $G = (V,E)$ generated by a TCK family $\rS=(S_v,S_e)$ on $\cH$.
Then either $\fS$ is a von Neumann algebra or there is a \wot-continuous completely contractive homomorphism $\phi: \fS \to \fL_G$ such that $\phi(S_{\lambda}) = L_{\lambda}$ for every $\lambda \in \bF_G^+$.
\end{corollary}

\begin{proof}
Since $G$ is transitive, the only directed subgraphs are $G$ and the empty graph, so that either $V_w = \emptyset$ or $V_w = V$. 
In the first case $P_0 = I$ so $\fS = P_0 \fS P_0$ is a von Neumann algebra, and in the second case there is a \wot-continuous completely contractive homomorphism $\phi: \fS \to \fL_G$ such that $\phi(S_{\lambda}) = L_{\lambda}$ for every $\lambda \in \bF_G^+$.
\end{proof}

\begin{lemma} \label{L:unique decomp}
Suppose that $G$ is a graph and $\fS$ is nonself-adjoint. Denote $G':= G[V_w] = (V_w,E_w)$, $\fS':=P_0^{\perp}\fS P_0^{\perp}$, 
and let $\phi : \fS \rightarrow \fL_{G[V_w]}$ be the canonical surjection given in Corollary $\ref{cor:type-L-wandering}$. 
Then every element $A\in \fS'$ can be uniquely represented as
\[
A = \sum_{|\lambda| < k} a_\lambda S_{\lambda} + \sum_{|\mu| = k}S_{\mu} A_{\mu}
\]
with $a_{\lambda} \in \bC$ and $A_{\mu} \in \fS'$ satisfying $A_{\mu} = S_{s(\mu)}A_{\mu}$.
Moreover, $\phi^{-1}(\fL_{G',0}^k) = \fS_0^k$ for all $k\geq 1$.
\end{lemma}

\begin{proof}
By Lemma~\ref{L:SGA_0^k}, if $A \in \fS_0^k$, then $A = \sum_{|\mu|=k} S_\mu A_\mu$ for $A_\mu \in \fS$.
Hence, $\phi(A) = \sum_{|\mu|=k} L_\mu \phi(A_\mu)$ belongs to $\fL_{G',0}^k$. 

On the other hand, since $\fS$ is nonself-adjoint, Theorem~\ref{T:dichotomy} applied to $\fS$ and $\fL_{G'}$ provides canonical isomorphisms
\[ \fS / \fS_0 \cong \ell^{\infty}(V_w) \cong \fL_{G'}/\fL_{G',0} .\] 
Therefore, the induced quotient map from $\fS / \fS_0$ to $\fS / \phi^{-1}(\fL_{G',0})$ is injective, 
and we see that $\fS_0 = \phi^{-1}(\fL_{G',0})$. 
Furthermore, since $\varphi(P_0) = 0$, it is clear that both $P_0 \fS$ and $\fS P_0$ are subsets of $\varphi^{-1}(\fL_{G',0})$. 
Thus, $\phi$ factors through the compression $\fS' = P_0^{\perp} \fS P_0^{\perp}$.

For $A \in \fS'$, let the Fourier series of $\phi(A)$ be  $\sum_{\lambda \in \bF_{G'}^+}a_{\lambda}L_{\lambda}$. 
In particular, $a_v$ for $v\in V_w$ are the unique scalars such that $A - \sum_{v\in V_w} a_v S_v$ lies in $\phi^{-1}(\fL_{G',0}) = \fS_0$. 
Thus by Lemma~\ref{L:SGA_0^k}, $A$ has the form 
\[ A = \sum_{v\in V_w} a_v S_v + \sum_{e\in E_w} S_e A_e ,\]
where $A_e$ are uniquely determined by $S_{s(e)}A_e = A_e \in \fS'$. 
Now apply this same decomposition to each $A_e$. 
By induction, we obtain a decomposition
\[
 A = \sum_{|\lambda|<k} a_{\lambda}S_{\lambda} + \sum_{|\mu|=k} S_{\mu}A_{\mu}
\]
where $A_{\mu} \in \fS'$ satisfy $S_{s(\mu)}A_{\mu} = A_{\mu}$. 
Moreover, this decomposition is unique again by Lemma~\ref{L:SGA_0^k}.
Finally, for $A\in \fS'$ we have $\phi(A) \in \fL_{G',0}^k$ if and only if $a_\lambda = 0$ for all $\lambda \in \bF_{G'}^+$ with $|\lambda|<k$. Therefore $(\fS'_0)^k = \phi^{-1}(\fL_{G',0}^k)$, and as $\varphi$ factors through $\fS'$ we have that $\fS_0^k = \varphi^{-1}(\fL^k_{G',0})$ as required.
\end{proof}

\begin{lemma} \label{L:left ideal}
Let $\fM = W^*(\fS)$ be the von Neumann algebra generated by $\fS$.
Then the ideal $\fJ:= \bigcap_{k\geq 1} \fS_0^k$ is a \wot-closed left ideal of $\fM$. 
There is a projection $P\in \fS$ such that $\fJ = \fM P$.
Moreover, $P$ commutes with each $S_v$ for $v\in V$.
Thus, $\fS$ contains the von Neumann algebra $P\fM P$.
\end{lemma}

\begin{proof}
By Corollary \ref{C:shift SGA_0^k}, for $e\in E$, we have
\[
 S_e^* \fJ = \bigcap_{k\geq 1}S_e^* \fS_0^k \subseteq \bigcap_{k\geq 1}\fS_0^{k-1} = \fJ.
\]
So $\fJ$ is an ideal of $\fS$ which is invariant under left multiplication by any $S_\mu^*$.
Thus, $\fJ$ is a left ideal of $\fM$.
By basic von Neumann algebra theory, there is a projection $P\in \fM$ such that $\fJ = \fM P$.
However, this lies in $\fS$, so $P \in \fS$. Hence, $\fS$ contains the self-adjoint algebra $P\fM P$.

For $v\in V$, since $S_v \in \fS$, we have that $\fJ = \fM P$ contains $\fJ S_v = \fM PS_v$. 
It follows that $PS_v = PS_vP$, and since the right term is self-adjoint, we get $S_vP=PS_v$. 
\end{proof}

This corollary characterizes when $\fS$ is self-adjoint.
There is one exceptional case, the \emph{trivial graph}  consisting of a vertex with no edges.

\begin{corollary} \label{C:fS=fS_0}
Let $\fS$ be the free semigroupoid algebra on $\cH$ generated by a TCK family $\rS= (S_v,S_e)$ for $G$.
Assume that the graph $G_\rS= G[\supp\rS]$ has no trivial components.
Then the following are equivalent:
\begin{enumerate}[label=\normalfont{(\arabic*)}]
\item $\fS_0=\fS$.
\item $\fS$ is a von Neumann algebra.
\item The set $V_w$ is empty.
\end{enumerate}
In this case, $G_\rS$ must be a disjoint union of transitive components.
\end{corollary}

\begin{proof}
We assume without loss of generality that $G=G_\rS = (V,E)$.

The hypothesis $\fS_0=\fS$ implies that $\fS_0^k = \fS$ for every $k\ge 1$.
Therefore $\fJ = \fS$ contains $I$, so that $P=I$ and $\fS = \fM$ is self-adjoint.

Assume that $\fS$ is self-adjoint. 
If the connected components of $G$ are not transitive, then there is an edge $e\in E$ such that $v:=s(e)$ and $w:=r(e)$ such that $w \not\succ v$.
But then $0 \ne S_e = P_{\cW_w} S_e P_{\cW_v} \in\fS$ and $P_{\cW_v} \fS P_{\cW_w} = \{0\}$.
This implies that $\fS$ is not self-adjoint, a contradiction.
Hence, $G$ must be the disjoint union of its transitive components.

Since $G$ has no trivial components and each component is transitive, for each vertex $v\in V$, there is some edge $e$ with $s(e)=v$. 
As $\fS$ is self-adjoint, there is a net $A_\lambda$ of polynomials converging in the \wot\ topology to $S_e^*$. 
Consequently, $S_v = S_e^*S_e = \wotlim A_\lambda S_e$ belongs to $\fS_0$. 
Hence, $V_w = \emptyset$.

Finally, if $V_w$ is empty, Theorem~\ref{T:dichotomy} shows that $\fS/\fS_0 \cong \{0\}$; so that $\fS_0=\fS$. 
\end{proof}

\begin{remark}
The case of a trivial graph $G$ is pathological because every free semigroupoid algebra on $G$ is just $\bC$. It is self-adjoint, but $\fS_0=\{0\} \ne \fS$.

If $\rS$ is a TCK family such that $G_\rS=G[\supp\rS]$ is not all of $G$, then the hypothesis of no trivial components applies to $G_\rS$.
For example, if $V=\{v,w\}$ and $E=\{e\}$ with $s(e)=v$ and $r(e)=w$, consider the TCK family $S_v=S_e=0$ and $S_w=1$ on $\bC$.
Then $\fS = \bC$ is self-adjoint and $\fS_0=\{0\}$ because $G_\rS$ is the trivial graph.
\end{remark}

We can now provide a structure theorem for free semigroupoid algebras which is analogous to the structure theorem for free semigroup algebras \cite{DKP}.

\begin{theorem}[Structure theorem] \label{thm:structure}
Let $\fS$ be the free semigroupoid algebra on $\cH$ generated by a TCK family $\rS= (S_v,S_e)$ for $G$.
Let $\fM  = W^*(\rS)$ be the von Neumann algebra generated by $\rS$. When $V_w\neq \emptyset$, let $\phi:\fS \to \fL_{G[V_w]}$ be the homomorphism of Corollary~$\ref{cor:type-L-wandering}$. Then there is a projection $P$  belonging to $\fS \cap \{S_v : v\in V\}'$ that dominates the projection $P_0 = \sum_{v\in V_0} S_v$ such that
\begin{enumerate}[label=\normalfont{(\arabic*)}]
\item $\fM P = \bigcap_{k\geq 1}\fS_0^k = \fJ$.

\item If $V_w\neq \emptyset$ then $\ker\phi = \fJ$.

\item $P^{\perp}\cH$ is invariant for $\fS$.

\item
With respect to $\cH = P\cH \oplus P^\perp\cH$, there is a  $2\times 2$ block decomposition
\[
 \fS = 
 \begin{bmatrix}
  P\mathfrak{M} P & 0 \\
  P^{\perp}\mathfrak{M}P & \fS P^{\perp}
 \end{bmatrix}
\]

\item If $V_w\neq \emptyset$, then $\fS P^{\perp}$ is completely isometrically and weak-$*$-hom\-eo\-mor\-phically isomorphic to $\fL_{G[V_w]}$.

\item
$P^{\perp}\cH$ is the closed linear span of wandering vectors.

\item
If there is a cycle in $G[\supp(\rS]$ through every vertex $v\in \supp(\rS)$, then $P$ is the largest projection such that $P\fS P$ is self-adjoint.
\end{enumerate}
\end{theorem}

\begin{proof}
We may assume at the outset that $G$ is equal to the induced subgraph $G_{\rS}$ on $\supp(\rS)$, so that $S_v\neq 0$ for every $v\in V$. Next, let us dispose of the case of trivial components. If $G$ is the trivial graph, with $V=\{v\}$ and $E=\emptyset$, then a TCK family for $G$ consists only of the projection $S_v$.
Indeed, as $\rS$ is non-degenerate, then $S_v=I$. Thus, $\fS = \bC S_v$, and the ideal $\fS_0=\{0\}$.
The projection $P=0$ and items (1--6) are trivial, and (7) does not apply. For the rest of the proof, up to decomposing $G$ as a disjoint union of undirected connected components, we assume there are no trivial components.

If $\fS$ is self-adjoint (and $G$ has no trivial components), by Corollary \ref{C:fS=fS_0} this is equivalent to $\fS_0^k = \fS$ for all $k\ge1$. Hence $P=I$, and items (1) through (4), (6) and (7) are clear. Hence, we may assume for the rest of the proof that $\fS$ is nonself-adjoint, and that $V_w$ is non-empty.
 
By Lemma~\ref{L:left ideal}, we get that $\fM P = \cap_{k\geq 1}\fS_0^k$ and $P$ commutes with each $S_v$. 
Hence, for every $S_v\in \fS_0$ we have $S_v = S_v^k \in \bigcap_{k\geq 1}\fS_0^k = \fJ$, 
so that $P\geq S_v$ for $v\in V_0$ and hence $P \geq P_0$.

Lemma~\ref{L:unique decomp} shows that 
\[  \fJ =  \bigcap_{k\geq 1} \phi^{-1}( \fL_{G[V_w],0}^k) = \phi^{-1}\big( \bigcap_{k\geq 1} \fL_{G[V_w],0}^k \big) = \ker\phi .\]
Since $\phi$ is non-zero, $\fJ \ne \fS$ and so $P \ne I$.

As for (3) and (4), observe that $P\cH = P \fM \cH = \fJ^*\cH$. Thus 
\[
 S_{\lambda}^* P\cH = S_{\lambda}^*\fJ \cH = (\fJ S_{\lambda})^* \cH \subseteq \mathfrak{J}^* \cH = P\cH
\]
Therefore $P\cH$ is invariant for $\fS^*$, whence $P^{\perp}\cH$ is invariant for $\fS$.
This shows that the $1,2$ entry of the matrix in (4) is $0$.
Now 
\[ \fS P \subseteq \fM P = \fJ = \fJ P \subseteq \fS P .\]
Hence $\fS P = \fM P = P \fM P + P^\perp \fM P$ form the $1,1$ and $2,1$ entries of the matrix. 
That leaves the $2,2$ entry $P^\perp \fS P^\perp = \fS P^\perp$ by the invariance of $P^\perp\cH$.

For item (5), Proposition \ref{P:complete_quotient} shows that $\phi$ induces a completely isometric weak-$*$-homeomorphic isomorphism $\widetilde{\phi} : \fS / \fJ \to \fL_{G[V_w]}$. 
However, as $ \fJ = \fS P$, we see that $\fS /  \fJ \cong \fS P^{\perp}$.
Therefore, $\fS P^\perp$ is completely isometrically isomorphic and weak-$*$ homeomorphic to $\fL_{G[V_w]}$.

For item (6), by Proposition \ref{cor:type-L-wandering} the closed linear span of the wandering vectors is contained in $P^{\perp}\cH$. 
For the converse, as $P$ commutes with $S_v$ for every $v\in V$, by Proposition \ref{P:compress_to_vertex} and Proposition \ref{P:extend wandering} it will suffice to show that the closed linear span of non-zero wandering vectors for $\fS_v$ is $P^{\perp} S_v \cH$ for each $v\in V$.

There are three cases to be handled. 
The first case is when there are no cycles at the vertex $v$, so that $\fS_v \cong \bC$. 
In this case, every vector in $P^{\perp}S_v\cH$ is wandering. 
The second case is when $\fS_v$ is an analytic type free semigroup algebra generated by at least two irreducible cycles. 
In this case \cite[Corollary 2.5]{DLP} and \cite[Lemma 3.2]{DLP} reduce the problem to establishing the \emph{existence} of a wandering vector in $P^{\perp}S_v \cH$ (where again the proofs work for $\aleph_0$ generators). 
By Theorem \ref{T:dichotomy} we have $V_w=V \setminus V_0$, which provides such a wandering vector at each vertex $v$ with $P^{\perp}S_v \neq 0$. 
The third case occurs when there is a single irreducible cycle at $v$. 
Here we apply Theorem~\ref{T:cycle_structure} to conclude that the wandering vectors span $P^\perp \cH_v$ which is the complement of the singular part.

For item (7), assume every vertex has a cycle through it and suppose that $Q\in \fS$ is a projection such that $Q\fS Q$ is self-adjoint. 
Decompose $\cH = P \cH \oplus P^{\perp} \cH$ to obtain a $2\times 2$ decomposition for $Q$ as in $(4)$. 
Since $Q$ is self-adjoint, its $2,1$ corner must vanish, and hence $P$ commutes with $Q$. 
So by (5), $P^{\perp}Q$ can be identified as a self-adjoint element of $\fL_{G[V_w]}$ via $\phi$.
By \cite[Corollary 4.6]{KP2004}, $\phi(P^{\perp}Q)$ must be the sum of projections $\sum_{v\in F}L_v$ for some subset $F\subseteq V_w$.
Hence, $P^{\perp}Q = \sum_{v\in F}S_vP^{\perp}$. 
If $F$ were non-empty, there would be a cycle $\mu$ with $s(\mu)=r(\mu)=v\in F$.
Now $S_\mu^*$ does not belong to $\fS$, for if it did, then $S_v = S_\mu^*S_\mu$ would lie in $\fS_0$ and so $v \in V_0 = V \setminus V_w$, contrary to hypothesis. As $S_\mu^* P$ does belong to $\fS$, we see that $S_\mu^* P^\perp$ does not belong to $Q\fS Q$. 
Thus, $Q\fS Q$ is not self-adjoint.
It follows that $F$ must be empty and $Q\leq P$. 
In other words, $P$ is the largest projection such that $P \fS P$ is self-adjoint.
\end{proof}

\begin{remark}
In \cite[Theorem 7.8]{MS2011} Muhly and Solel provide a structure theorem that works in the general context of $W^*$-correspondences. Their theorem applies to our context when the graph $G$ has finitely many vertices and edges to obtain items (1)--(5) in our theorem above. Our proof avoids $W^*$-correspondence machinery, works for all countable directed graphs, and also yields items (6) and (7).
\end{remark}

\begin{example}
Let $\rS=(S_v,S_e)$ be an arbitrary TCK family on $\cH$. 
Let $U$ be the bilateral shift on $\ell^2(\bZ)$. 
Then the TCK family $\rT = (S_v \otimes I, S_e\otimes U)$ always generates an analytic free semigroupoid algebra $\fT$ on $\cH \otimes \ell^2(\bZ)$.

To see this, let $\{\xi_n:n\in\bZ\}$ be the standard basis for $\ell^2(\bZ)$. 
Fix any non-zero vector in $x \in \cH$. We claim that $x \otimes \xi_n$ is a wandering vector.
Note that if $|\mu|=k$, then 
\[ T_\mu (x \otimes \xi_n) = S_\mu \otimes U^k (x \otimes \xi_n) = S_\mu x \otimes \xi_{n+k} .\]
Therefore the subspace $\fT[x\otimes\xi_n]$ is graded, and hence if $|\mu|=k \ne l = |\nu|$, then 
\[ \ip{T_\mu (x \otimes \xi_n), T_\nu (x \otimes \xi_n)} =  \ip{S_\mu x \otimes \xi_{n+k}, S_\nu x \otimes \xi_{n+l}} = 0 .\]
Also if $|\mu|=|\nu|$, then $S_\mu$ and $S_\nu$ have pairwise orthogonal ranges; so that 
\[ \ip{T_\mu (x \otimes \xi_n), T_\nu (x \otimes \xi_n)} =  \ip{S_\mu x , S_\nu x } = 0 .\]
Thus $x \otimes \xi_n$ is a wandering vector.
This shows that the wandering vectors span $\cH\otimes \ell^2(\bZ)$, and therefore $\fT$ is analytic type.
\end{example}

\section{Absolute continuity and wandering vectors} \label{S:abs cont}

In \cite{DLP}, a notion of absolute continuity for representations of the Toeplitz-Cuntz algebra was introduced and used to tackle the question of the existence of wandering vectors for analytic type representations. 
This enabled the recognition of the analytic part of a representation via properties of the vector states.
This was an important step towards a solution of the problem of existence of wandering vectors for free semigroup algebras.
Absolute continuity was then further investigated by Muhly and Solel \cite{MS2011} in the context of $W^*$-correspondences, where a representation theoretic characterization was found. 
This characterization was then used by Kennedy in \cite{Kennedy2013} where he showed that every absolutely continuous representation of a free semigroup algebra on at least two generators is analytic type. 

In this section we extend these results to free semigroupoid algebras.
In particular, we show that the vectors that determine an absolutely continuous state determine the absolutely continuous part of the representation.
In particular, the span of all wandering vectors consist of absolutely continuous vectors, so that the analytic part is contained in the absolutely continuous part. On the other hand, there can be absolutely continuous part which is not in the analytic part because of the pathology that occurs in the case of a simple cycle (Theorem~\ref{T:cycle_structure}). Indeed this is the heart of the difference 
between algebras which are regular and those which are not.

Recall that TCK families $\rS$ are in bijection with $*$-representations $\pi_S$ of $\cT(G)$, and that $\pi_L$ is the left regular representation of $\cT(G)$.

\begin{definition}
A linear functional $\phi$ on $\cT_+(G)$ is  \textit{absolutely continuous} if there are vectors $\xi,\eta\in\cH_G$ so that $\phi(A) = \ip{\pi_L(A) \xi,\eta}$.
A TCK family $\rS$ (or the free semigroupoid algebra $\fS$ that it generates) associated to a representation $\pi$ of $\cT(G)$ on $\cH$ is \textit{absolutely continuous} if the functionals $\phi(A) = \ip{\pi_{\rS}(A) x,y}$ on $\cT_+(G)$ are absolutely continuous for every $x,y\in \cH$.
Also, a vector $x \in \cH$ is \textit{absolutely continuous} if the vector state $\phi(A) = \ip{\pi(A)x,x}$ is absolutely continuous.
\end{definition}

We note that the restriction to $\fS[\cW]$ for a wandering subspace $\cW$ is always absolutely continuous, as on that space it is unitarily equivalent to a left-regular type algebra. 
Another important observation is that $\varphi : \cT_+(G) \rightarrow \bC$ is absolutely continuous if and only if it extends to a weak-$*$ continuous functional on $\fL_G$. 
Indeed, if it is absolutely continuous, then it clearly extends. 
Conversely, when $\varphi$ extends, as $\fL_G$ has property $\bA_1$ by \cite[Theorem 3.1]{JK2005}, we see that $\varphi$ is a vector functional. 
We note that the assumption of no sinks in \cite[Theorem 3.1]{JK2005} is unnecessary as the Wold decomposition of Theorem \ref{thm:Wold-decomp} can be used instead in its proof. 
In particular, we see that the set of absolutely continuous functionals is identified with the predual of $\fL_G$.
Thus it is complete, and forms a norm-closed subspace of the dual of $\cT_+(G)$. 
Also, from \cite[Corollary 3.2]{JK2005}, we see that the weak and weak-$*$ topologies coincide on $\fL_G$. 

The following explains why absolute continuity for free semigroupoid algebras is a natural analogue of the operator theory  fact that any absolutely continuous isometry has an $H^{\infty}$ functional calculus.

\begin{proposition}
Let $G$ be a directed graph and $\fS$ a free semigroupoid algebra generated by a TCK family $\rS$. 
Let $\pi_{\rS} : \cT(G) \rightarrow B(\cH)$ be the representation associated to $\rS$. 
Then the following are equivalent
\begin{enumerate} [label=\normalfont{(\arabic*)}]
\item
$\fS$ is absolutely continuous.
\item
The representation $\pi_{\rS} : \cT_+(G) \rightarrow B(\cH)$ extends to a completely contractive weak-$*$ continuous representation of $\fL_G$.
\end{enumerate}
\end{proposition}

\begin{proof}
Suppose (2) holds and that $\phi$ is given by $\phi(A) = \la \pi_{\rS}(A)x,y \ra$ for $x,y\in \cH$ and $A\in \cT_+(G)$. 
By hypothesis, $\phi$ extends to a weak-$*$ continuous functional on $\fL_G$. 
Hence, $\phi$ is absolutely continuous.

For the converse, suppose $T\in \fL_G$. 
By Remark \ref{R:Cesaro_in_T_+(G)} we can find a net $X_\alpha$ in $\cT_+(G)$ in the $\|T\|$-ball such that $\pi_L(X_\alpha)$ converges to $T$ in the weak operator topology on $\fL_G$. 
In order to define the extension of $\pi_{\rS}$ to $\fL_G$, it will suffice to show that $\pi_{\rS}(X_\alpha)$ converges \wot~to some element $Y \in B(\cH)$. 
Since $\pi_{\rS}(X_\alpha)$ is a bounded net, it has a \wot~cluster point $Y$. 
If $x,y\in\cH$, let $\xi,\eta\in \cH_G$ be chosen so that
\[ \phi(A) := \ip{ \pi_\rS(A) x,y } = \ip{ \pi_L(A) \xi, \eta } .\]
Then $Y$ satisfies
\[
 \ip{ Yx,y} = \lim \ip{ \pi_\rS(X_\alpha) x,y} = \lim \ip{ \pi_L(X_\alpha) \xi, \eta} = \ip{T\xi,\eta}.
\]
Therefore $Y$ is uniquely determined among cluster points of $\pi_{\rS}(X_{\alpha})$, and thus $\wotlim_\alpha \pi_\rS(X_\alpha) = Y$.
So we may define the extension $\pi_{\rS} : \fL_G \rightarrow B(\cH)$ by 
\[ \pi_{\rS}(T) = \wotlim_{\alpha}\pi_{\rS}(X_\alpha) .\]
A similar argument shows that this is independent of the choice of the net $X_\alpha$, and that this extension is \wot-continuous.
Note that since $\pi_{\rS}(X_\alpha)$ is in the $\|T\|$-ball of $B(\cH)$, we get that $\pi_{\rS}$ is contractive on $\fL_G$. 
A similar argument for matrices shows that $\pi_{\rS}$ is completely contractive.
\end{proof}

Following the approach for free semigroup algebras in \cite{DLP}, we also define intertwining maps.

\begin{definition}
Let $\rS=(S_v,S_e)$ be a TCK family for $G$ on $\cH$. 
An \emph{intertwiner} is a map $X \in B(\cH_G,\cH)$ satisfying $S_\mu X = XL_\mu$ for all $\mu\in \bF_G^+$.
We let $\cX(\rS)$ denote the set of all intertwining maps.
Let 
\[ \cV_{ac}(\rS) = \{ x = X\eta : X\in \cX(\rS) \} \]
be the union of the ranges of intertwining maps.
\end{definition}

The following is a straightforward extension of \cite[Theorem 1.6]{DKP}.

\begin{lemma} \label{L:add L}
Suppose that $\rS = (S_v,S_e)$ is a TCK family for a graph $G$ on $\cH$, and let $x \in \cH$ be an absolutely continuous unit vector.
Let $\rT = \rS \oplus \rL = (S_v\oplus L_v, S_e\oplus L_e)$ generate the free semigroupoid algebra $\fT$ on $\cH \oplus \cH_G$.
For any $\ep > 0$, there exists a vector $\zeta \in \cH_{G}$ with $\| \zeta \| < \ep$ such that the restriction of $\fT$ to the cyclic subspace 
$\cN := \fT[x \oplus \zeta]$ is unitarily equivalent to $\fL_{G,V'}$ for some subset $V' \subseteq V$.  
\end{lemma}

\begin{proof}
Let $\pi_\rS$ be the representation of $\cT(G)$ associated to $\rS$.
Consider the vector state on $\cT^+(G)$ given by $\psi(A) = \ip{ \pi_\rS(A)x, x }$ for $A \in \cT^+(G)$.   
Since $x$ is absolutely continuous, there are vectors $\zeta$ and $\eta$ in $\cH_{G}$ such that
\[ \psi(A) = \ip{ \pi_L(A) \zeta,\eta } .\]
As $\psi(A) = \ip{  \phi(S)\, t\zeta, t^{-1}\eta }$ for any non-zero scalar $t$, we may assume that $\| \zeta \| < \ep$.  
We then  compute for all $A \in \cT^+(G)$,
\[
  \ip{  (\pi_\rS(A) \oplus \pi_L(A)) (x \oplus \zeta), x \oplus -\eta } =
  \ip{ \pi_\rS(A) x, x } - \ip{ \pi_L(A) \zeta,\eta } =  0 .
\]
So $x \oplus -\eta$ is orthogonal to the cyclic subspace $\cN := \fT[x \oplus \zeta]$. Let $\rN = (N_v,N_e)$ be the restriction of $\rT$ to $\cN$, and let $\fN$ be its free semigroupoid algebra.

By the Wold decomposition (Theorem \ref{thm:Wold-decomp} and Proposition \ref{P:fully-coisometric space}) for $\rN$, we obtain an \textit{orthogonal} decomposition\footnote{which is not necessarily compatible with the decomposition $\cH \oplus \cH_G$ and is hence denoted $+$ to emphasize this.} $\cN = \cL + \cM$, where $\cL$ is the left regular part determined by the wandering subspaces $\cW_v=(N_v-\sum_{e\in r^{-1}(v)} N_eN_e^*)\cN$, so that
\[
 \cL := \sum_{v\in V} \fT[\cW_v]
\]
and $\cM$ is the fully coisometric CK part, given by Proposition~\ref{P:fully-coisometric space} as
\[
 \cM = \bigcap_{k\ge1} \fT_0^k [x \oplus \zeta] .
\]
Now, since $\fT_0^k (\cH \oplus \cH_G) \subseteq \cH \oplus Q_k \cH_G$
where $Q_k$ is the projection onto the span of $ \{ \xi_\mu : |\mu| \ge k \}$.
The spaces on the right side have intersection $\cH \oplus 0$.
Thus, $\cM$ is contained in $\cH \oplus 0$.

Since $x \oplus \zeta$ belongs to $\cN = \cM + \cL$, we may write $x\oplus \zeta = y + z$, 
where $y = y_0 \oplus 0 \in \cM$ and $z = (x- y_0) \oplus \zeta$ lies in $\cL$.
Moreover, $y$ and $z$ are orthogonal, whence  $0 = \ip{  y_0 , x - y_0 }$.
We saw above that  $x \oplus -\eta$ is orthogonal to $\cN = \cM + \cL$,
and, in particular, is  orthogonal to $y_0 \oplus 0$.
Therefore, we also have $\ip{ x, y_0} =0$.  
Hence, $\ip{y_0,y_0} = 0$,  i.e., $y_0 = 0$.   

Therefore,  $x \oplus \zeta$ is contained in $\cL$.  
But $\cL$ is invariant for $\fT$, and $x\oplus\zeta$ generates $\cN$; whence $\cN = \cL$ and $\cM = 0$.   
Thus, since $\cN$ is a cyclic subspace for $\fT$ generated by a single vector, the dimension of
$\cW_v$ is at most one for each $v\in V$. Define $V'$ to be those $v\in V$ for which $\cW_v$ is of dimension 1.  
Therefore, $\cN$ is unitarily equivalent to $\bigoplus_{v\in V'} \fL_{G, v}$; whence $\fN = \fT|_{\cN} \cong \fL_{G,V'}$.
\end{proof}

With this in hand, we can prove an analogue of \cite[Theorem~2.7]{DLP} relating intertwiners and absolute continuity.

\begin{theorem} \label{T:V_ac}
Let $\rS = (S_v,S_e)$ be a TCK family on $\cH$ for a graph $G$ corresponding to a representation $\pi_\rS$ of $\cT(G)$. Then
\begin{enumerate} [label=\normalfont{(\arabic*)}]
\item $\psi(A) = \ip{\pi_\rS(A) x,y}$ is absolutely continuous for all $x,y\in \cV_{ac}$.
\item $\cV_{ac}$ is a closed subspace consisting of all absolutely continuous vectors.
\item $\cV_{ac}$ is invariant for $\fS$.
\item $\cV_{ac} \oplus 0$ lies in the closed span of the wandering vectors for $\rS \oplus \rL$ in $\cH \oplus \cH_G$.
\end{enumerate}
\end{theorem}

\begin{proof}
Since $x,y$ are in $\cV_{ac}$, there are maps $X,Y\in\cX(\rS)$ and $\zeta, \eta \in\cH_G$ such that $x = X\zeta$ and $y=Y\eta$.
Therefore,
\[
 \psi(A) = \ip{\pi_\rS(A) X\zeta, Y\eta} = \ip{X\pi_L(A) \zeta, Y\eta} = \ip{\pi_L(A) \zeta, X^*Y\eta} .
\]
Hence $\psi$ is absolutely continuous, and (1) holds.
In particular, every vector in $\cV_{ac}$ is absolutely continuous.

If $x$ is absolutely continuous, then by Lemma~\ref{L:add L} there is a vector $\zeta\in\cH_G$ so that $\cN=\fT[x\oplus\zeta] \cong \cH_{G,V'}$ for $V' \subseteq V$ in a way that yields a unitary equivalence between $\fS|_\cN$ and $\fL_{G,V'}$.
This means that there is a partial isometry $U\in B(\cH_G, \cH)$ with domain $\cH_{G,V'}$ and range $\cN$ so that 
\[ ( S_\mu \oplus L_\mu ) U = U L_\mu \qforal \mu\in \bF_G^+ .\]
Let $X = P_\cH U$ and observe that this is an intertwiner in $\cX(\rS)$.
There is a vector $\xi = U^*(x \oplus \zeta) \in \cH_G$ such that $U\xi = x \oplus \zeta$.
Therefore, $X\xi = x$ belongs to $\cV_{ac}$.
Combining this with (1) shows that $\cV_{ac}$ consists of all absolutely continuous vectors.
Also note that $x \oplus\zeta$ lies in the span of the wandering vectors.

It now follows that $\cV_{ac}$ is a closed subspace. Indeed, if $x,y\in\cV_{ac}$, then clearly so are scalar multiples. Moreover,
\[  \ip{\pi_\rS(A) x\!+\!y, x\!+\!y} =  \ip{\pi_\rS(A) x,x} +  \ip{\pi_\rS(A) x,y} +  \ip{\pi_\rS(A) y,x} +  \ip{\pi_\rS(A) y,y} \]
is a sum of weak-$*$ functionals on $\fL_{G}$, and hence is a weak-$*$ continuous functional.
So $x+y$ is absolutely continuous. 
Furthermore, $\cV_{ac}$ is closed because the predual of $\fL_G$ is complete.

If $x\in\cV_{ac}$ and $A\in\cT^+(G)$, select a $\xi\in\cH_G$ and $X\in\cX(\rS)$ so that $x=X\xi$.
Then
\[ \pi_\rS(A) x =  \pi_\rS(A) X\xi = X \pi_L(A) \xi \]
is in the range of $X$. 
Thus, $\cV_{ac}$ is invariant for $\pi_\rS(\cT^+(G))$, and hence is also invariant for its \wot\ closure $\fS$.

Finally, returning to the setup in the second paragraph, we see that $x \oplus\zeta$ lies in a subspace of left regular type, and hence lies in the closed span of wandering vectors. 
By Lemma~\ref{L:add L}, this can be accomplished with $\|\zeta\|<\ep$, and thus $x \oplus 0$ also lies in the closed span of wandering vectors. 
\end{proof}

\begin{corollary} \label{C:largest-ac}
If $\rS$ is a TCK family for $G$, then $\cV_{ac}$ is the unique largest invariant subspace $\cK$ such that $\fS|_{\cK}$ is absolutely continuous.
\end{corollary}

\begin{proof}
By Theorem \ref{T:V_ac} $\cV_{ac}$ is invariant and $\fS|_{\cV_{ac}}$ is absolutely continuous. Also, by definition of absolute continuity and Theorem \ref{T:V_ac} again, we see that if $\cK$ is any invariant subspace for which $\fS|_{\cK}$ is absolutely continuous, then $\cK \subseteq \cV_{ac}$.
\end{proof}

Our next goal is to understand the precise relationship between absolutely continuous and analytic free semigroupoid algebras.

\begin{corollary} \label{P:analytic-implies-ac}
Let $G$ be a graph and $\fS$ a free semigroupoid algebra generated by a TCK family $\rS$. 
If $\fS$ is analytic, then it is absolutely continuous.
\end{corollary}

\begin{proof}
Each wandering vector $x$ is absolutely continuous because $\fS[x]\cong \cH_{G,\supp x}$, and this unitary equivalence is an intertwining operator with range $\fS[x]$. 
Thus, the span of wandering vectors lies in $\cV_{ac}$.
By item (6) of Theorem \ref{thm:structure}, it follows that $\cH$ is the closed linear span of wandering vectors. 
Therefore, $\cV_{ac}=\cH$, and $\fS$ is absolutely continuous.
\end{proof}

\begin{example} \label{E:cycle_abs_cont}
Let $G$ be a graph which contains a cycle $C_n$ that has no entries; i.e.\ the cycle has edges $e_1,\dots,e_n$ on vertices $v_1,\dots,v_n$
and if $f\in E$ such that $r(f) = v_{i\!\! \pmod n +1}$ for some $1 \le i \le n$, then $f= e_i$.
Let $A$ be a measurable subset of the circle $\bT$ such that $0 < m(A) < 1$.
Define a representation $\pi$ of $\cT^+(C_n)$ on $\cV_1=\bC^n \otimes L^2(A)$ by setting $\pi(e_i) = E_{i,i+1}\otimes I$ for $1 \le i < n$
and $\pi(e_n) = E_{n,1}\otimes M_{z,A}$, where $M_{z,A}$ is multiplication by $z$ on $L^2(A)$.
Now enlarge the space and extend this in any manner to a TCK family $\rS_1$ of $G$ by defining the images of the other generators, such that $\cV_1$ is cyclic for $\fS_1$. There are various ways to do this, and we leave the details to the reader.
This extension is now unique up to unitary equivalence because the Wold Decomposition (Theorem \ref{thm:Wold-decomp}) shows that $\cW =\big(  \bigoplus_{s(e)\in C_n} S_e\cV_1 \big)\ominus\cV_1$ is a wandering subspace generating $\cV_1^\perp$. Let $\rS_2$ be a second representation as above, using $A^c = \bT\setminus A$ instead of $A$ to define $\cV_2$, which generates the free semigroupoid algebra $\fS_2$.

We now apply Theorem~\ref{T:cycle_structure} about representations of a cycle.
The representation $\pi_{\rS_1} \oplus\pi_{\rS_2}$ compressed to  $\bC^n \otimes (L^2(A) \oplus L^2(A^c))$ is spanned by wandering vectors, so we see that $\pi_{\rS_1} \oplus\pi_{\rS_2}$ generates an analytic free semigroupoid algebras.
In particular, both $\pi_{\rS_1}$ and $\pi_{\rS_2}$ are absolutely continuous on $\bC^n \otimes L^2(A)$ and $\bC^n \otimes L^2(A^c)$ respectively.
However, $\fS_1$ and $\fS_2$ have proper non-zero structure projections because there are no wandering vectors for $\cV_1$ and $\cV_2$. Thus, $\fS_1$ and $\fS_2$ are absolutely continuous free semigroupoid algebras that are not analytic.
\end{example}

It turns out that this example is the only possible pathology. Note that if $\mu=e_n\dots e_1$ is a cycle in a graph $G$ and $S_\mu$ is a unitary map on $S_{s(\mu)}\cH$, then any edge $f$ distinct from $e_n,...,e_1$ with $r(f)\in\{v_1,\dots,v_n\}$ must have $S_f=0$. Conversely, for a CK family with this property, we must have that $S_\mu$ unitary on $S_{s(\mu)}\cH$.
If the representation is absolutely continuous, then the spectral measure of $S_\mu$ must be absolutely continuous with respect to Lebesgue measure on $\bT$. 
As the example above shows, in order to be analytic, the spectral measure must be equivalent to Lebesgue measure.
Equivalently, we could assume that the spectral measure dominates Lebesgue measure on $\bT$ and is absolutely continuous.

By analogy to the case of free semigroup algebras, we adopt the following definition from \cite{DLP}.

\begin{definition}
A free semigroupoid algebra $\fS$ is \emph{regular} if the absolutely continuous part is analytic.
\end{definition}

We note that a free semigroup algebra $\fS$ can fail to be regular only if the free semigroup has a single generator, 
and the isometry generating it is unitary and the spectral measure has a proper absolutely continuous part which is not equivalent to $m$. This follows from Wermer's Theorem as in Example~\ref{E:cycle_abs_cont}. When this notion was introduced in \cite{DLP}, it was not clear what could happen when the number of generators is two or more. Kennedy \cite{Kennedy2013} showed that free semigroup algebras on at least two generators are always regular. 
The following characterizes regular free semigroupoid algebras.

\begin{theorem}\label{T:abs_cont}
Suppose that $\fS$ is a free semigroupoid algebra of a graph $G$ generated by a TCK family $\rS$ on $\cH$. 
Then $\fS$ is regular if and only if it satisfies the measure theoretic condition
\begin{itemize} 
\item[\normalfont{(M)}] Whenever $\mu$ is a cycle such that $S_\mu$ is unitary on $S_{s(\mu)}\cH$, the spectral measure of $S_\mu$ is either singular or dominates Lebesgue on $\bT$.
\end{itemize}
\end{theorem}

\begin{proof}
First assume that (M) fails to hold. Then there is a cycle $\mu$ such that $S_\mu$ is unitary, 
say with decomposition $S_\mu\cong U_a \oplus U_s$ on $S_{s(\mu)}\cH \cong \cN_a \oplus \cN_s$ into its absolutely continuous and singular parts, such that $U_a \ne 0$ is supported on a subset $A\subseteq \bT$ with $0 < m(A) < 1$.
Since $S_\mu$ is unitary, the cycle $\mu$ has no non-zero entries; i.e., in the graph $G[\supp \rS]$ the cycle $\mu$ has no entries.
By the analysis in Theorem~\ref{T:cycle_structure} and Example~\ref{E:cycle_abs_cont}, we see that the subspace $\cN_a$ is absolutely continuous but contains no wandering vectors, and thus is not analytic.
Hence, $\fS$ is not regular.

Now suppose that (M) holds. By Proposition \ref{P:analytic-implies-ac}, the analytic part is contained in the absolutely continuous part. 
For the converse, by restricting to the absolutely continuous part, we may assume that $\fS$ is absolutely continuous, so that $\cH = \cV_{ac}$. By Theorem~\ref{thm:structure}(6), we need to show that $\cH$ is spanned by wandering vectors. By Proposition \ref{P:compress_to_vertex} and Proposition~\ref{P:extend wandering}, it is enough to show that each $S_v\cH$ is spanned by wandering vectors for the free semigroup algebra $\fS_v = S_v \fS |_{S_v\cH}$.  The free semigroup is generated by the irreducible cycles at $v$. Since $\fS$ is absolutely continuous, each $\fS_v$ is absolutely continuous.

There are then three cases. If there are no cycles through $v$, then $\fS_v \cong \bC$ and every vector in $S_v\cH$ is wandering. If there are at least $2$ irreducible cycle operators, then by \cite[Theorem 4.16]{Kennedy2013} we get that $\fS_v$ is analytic type and $\cH$ is spanned by wandering vectors.
Finally, if $v$ lies on a single irreducible cycle $\mu$, Theorem~\ref{T:cycle_structure} shows that $\fS_v$ is generated by $S_\mu$. If $S_\mu$ is not unitary, then its Wold decomposition contains copies of the unilateral shift $U_+$. As $\fS_v$ is absolutely continuous, this falls into case (2) of Theorem \ref{T:cycle_structure}.
If $S_\mu$ is unitary, as $\fS_v$ is absolutely continuous we get by hypothesis that the spectral measure of $S_\mu$ dominates $m$, which is also case (2) of Theorem \ref{T:cycle_structure}. Hence, in either case, the wandering vectors span the absolutely continuous part (i.e.\ all of $S_v \cH$).
\end{proof}


While it may not always be easy to verify if $\fS$ is regular, the existence of sufficiently many wandering vectors will always suffice.
For this reason, an important special case of the following corollary is that $\rL \oplus \rT$ is always regular for any TCK family $\rT$.

\begin{corollary}
Let $\rS$ and $\rT$ be TCK families for $G$ on $\cH_{\rS}$ and $\cH_{\rT}$ respectively, and assume that $\rS$ is regular. 
Then $\rS \oplus \rT$ is regular if and only if whenever $\mu$ is a cycle and $T_{\mu}$ is unitary on $T_{s(\mu)}\cH_{\rT}$, we either have that
\begin{enumerate}[label=\normalfont{(\arabic*)}]
\item
spectral measure of $T_{\mu}$ is singular or dominates Lebesgue on $\bT$; or
\item
$P_\rS^\perp \cH_{\rS,s(\mu)} \ne \{0\}$, where $P_\rS$ is the structure projection for $\rS$. 
\end{enumerate}
In particular, $\rL \oplus \rT$ is always regular.
\end{corollary}

\begin{proof}
Since $\rS$ is regular, it satisfies (M). So if (1) holds, then (M) is satisfied for $\rS \oplus \rT$.
On the other hand, if (2) holds, then $\rS$ has wandering vectors in $\cH_{\rS,s(\mu)}$ by the Structure Theorem~\ref{thm:structure} (6).
It follows that $S_\mu$ acts as a shift on the cyclic subspace generated by the wandering vector, and hence $S_\mu$ is either a proper isometry
or contains a summand unitarily equivalent to the bilateral shift, in which case its spectral measure dominates Lebesgue measure.
Either way, (M) holds for $\rS \oplus \rT$. Thus, $\rS \oplus \rT$ is regular.

Now $\rL$ has analytic type, and is supported on every vertex, so (2) holds generally. Thus, $\rL \oplus \rT$ is always regular.
\end{proof}

\section{Lebesgue-von Neumann-Wold decomposition} \label{S:LvNW}

We apply the Structure Theorem to obtain a decomposition theorem for regular Toeplitz-Cuntz-Krieger families. We show that the classification of regular representations of the Toeplitz-Cuntz-Krieger algebra $\cT(G)$ reduces to the classification of \emph{fully coisometric} absolutely continuous, singular and dilation type representations. 
This generalizes Kennedy's Leb\-esgue decomposition for isometric $d$-tuples \cite[Theorem 6.5]{Kennedy2013}, which is the free semigroup algebra case.

We already have left regular and analytic TCK families and now absolutely continuous ones.
We define a few other types.

\begin{definition}
Let $\rS = (S_v,S_e)$ be an TCK family for a graph $G$. We will say that $\rS$ is
\begin{enumerate} [label=\normalfont{(\arabic*)}]
\item
\emph{singular} if it has no absolutely continuous restriction to any invariant subspace.
\item 
\emph{von Neumann type} if $\fS$ is a von Neumann algebra.
\item
\emph{dilation type} if it has no restriction to a reducing subspace that is either analytic or a von Neumann algebra.
\end{enumerate}
\end{definition}

Note that if a free semigroupoid algebra $\fS$ is singular, then the structure projection $P=I$ since the range of $P^\perp$ is analytic and hence absolutely continuous. 
Therefore, $\fS$ is a von Neumann algebra.
But we know from Theorem~\ref{T:cycle_structure} that there are absolutely continuous free semigroupoid algebras which are von Neumann algebras. 
However, these examples are not regular.

The following is standard, but we include it for the reader's convenience.

\begin{lemma} \label{L:cyclic reducing}
Suppose that $\fS$ is a free semigroupoid algebra and that $\cV$ is a coinvariant subspace. 
Then $\fS[\cV]$ and $\fS^*[\cV^\perp]$ are reducing for $\fS$. 
\end{lemma}

\begin{proof}
The proofs of the two statements are similar, and we prove only the first.
Observe that $\fS[\cV]$ is invariant for $\fS$ by construction, and it is spanned by vectors of the form $S_\mu \xi$ for $\xi \in\cV$ and $\mu\in \bF_G^+$. 
It follows that for $e\in E$, 
\[
 S_e^* S_\mu \xi = 
\begin{cases}
 S_{\mu'} \xi &\qif \mu = e \mu'\\
 S_e^* \xi &\qif \mu = r(e). \\
 0 &\qif \text{else} 
 
\end{cases}
\]
The co-invariance of $\cV$ shows that $S_e^*\xi \in \cV$; and therefore $S_e^*\fS[\cV] \subseteq \fS[\cV]$.
Thus, $\fS[\cV]$ is reducing.
\end{proof}

\begin{lemma} \label{L:dilation type}
Let $G$ be a graph, and let $\rS=(S_v,S_e)$ be a TCK family for $G$. 
Let $P$ be the structure projection for $\fS$, and set $\cV=P\cH$. 
Then $\rS$ is of dilation type if and only if $\cV$ is cyclic for $\fS$ and $\cV^\perp$ is cyclic for $\fS^*$. 

If $\rS$ is regular, then $\rS$ is of dilation type if and only if it does not have an absolutely continuous summand or a singular summand.

Let $\rA := P\rS|_{P\cH}$. Then $\rA$ is a contractive $G$ family and $\rS$ is the unique minimal dilation of $\rA$ to a TCK family. Moreover, the restriction $\rS|_{P^\perp\cH}$ is unitarily equivalent to a left regular family.
\end{lemma}

\begin{proof}
Assume that $\rS$ is dilation type, and let $\cV = P\cH$. 
Since $\cV$ is coinvariant, the subspaces $\fS[\cV]$ and $\fS^*[\cV^\perp]$ reduce $\rS$ by Lemma~\ref{L:cyclic reducing}.
Now $\fS[\cV]^\perp$ is a reducing subspace contained in $\cV^\perp$, and hence it is an analytic, and thus an absolutely continuous reducing subspace. 
Thus, $\fS[\cV]^{\perp}$ must be $\{0\}$, so that $\cV$ is cyclic.
Similarly, $\cN = \fS^*[\cV^\perp]^\perp$ is a reducing subspace for $\fS$ contained in $\cV$, and thus $\fS|_\cN$ is a von Neumann algebra.
Again this means that $\fS^*[\cV^{\perp}]^{\perp}$ is $\{0\}$, so that $\cV^\perp$ is cyclic for $\fS^*$.

Conversely, if $\cV$ is cyclic, then there is no reducing subspace in $\cV^\perp$; so there is no analytic summand.
Similarly if $\cV^\perp$ is cyclic for $\fS^*$, then there is no reducing subspace of $\cV$, which would be of von Neumann type.
Thus, $\fS$ is dilation type.

We now verify that if $\fS$ is regular, then it has no absolutely continuous or singular summand.
However, in the regular case, the absolutely continuous part is spanned by wandering vectors and lives in the analytic part.
Thus, any absolutely continuous summand is an analytic summand, which cannot occur.

Since $P\cH$ is coinvariant for $\fS$, the compression of $\pi_{\rS}$ to $P\cH$ is a completely contractive homomorphism when we restrict to $\cT_+(G)$. 
This representation is determined by the image of the generators, namely $\rA=(A_v,A_e)$. 
By a result of Muhly and Solel \cite[Theorem 3.3]{MS1998}, $\rA$ has a unique minimal dilation to a TCK family. As $\rS$ is also a minimal dilation of $\rA$ by the previous paragraph, it is the unique minimal dilation up to a unitary equivalence fixing $\cV$. 

Finally, let $\cW = \big( \cV + \bigoplus_{e\in E} S_e \cV \big) \ominus \cV$. It is readily verified that this is a non-empty wandering subspace for $\rS|_{\cV^\perp}$ and that 
\[ \cH = \fS[\cV] = \cV \oplus \fS[\cW] .\]
Therefore $\cW$ is a cyclic wandering subspace for $\fS|_{\cV^\perp}$, which then must be left regular type by Wold Decomposition (Theorem \ref{thm:Wold-decomp}).
\end{proof}

\begin{remark} \label{R:abs cnts diln}
Regularity is required for the second statement in this lemma.
Consider the graph $G$ on two vertices $v,w$ with edges $e,f,g$ such that $s(e)=r(e)=s(f)=v$ and $r(f)=s(g)=r(g)=w$.
Consider the representation $\pi_{\rS_1}$ coming from the cycle at $v$, and the corresponding CK family $\rS_1 = (S_v,S_e)$ of $G$ constructed for this graph as in Example~\ref{E:cycle_abs_cont}. This representation is absolutely continuous. However, $\cV= S_v \cH$ is cyclic and $\cV^\perp=S_w \cH$ is cyclic for $\fS_1^*$. Hence, $\rS_1$ is absolutely continuous and of dilation type.
\end{remark}

We are now able to provide the analogue of Kennedy's Lebesgue decomposition for isometric tuples \cite[Theorem 6.5]{Kennedy2013}.

\begin{theorem} \label{T:LvNW-decomp}
Let $G$ be a graph and let $\rS$ be a TCK family on $\cH$ generating a free semigroupoid algebra $\fS$. 
Then up to unitary equivalence we may decompose
\[
\rS \cong \rS_l \oplus \rS_{an}\oplus \rS_{vN} \oplus \rS_d
\]
where 
\begin{enumerate}[label=\normalfont{(\arabic*)}]
\item
$\rS_l$ is unitarily equivalent to a left-regular TCK family.
\item
$\rS_{an}$ is an analytic fully coisometric CK family. 
\item
$\rS_{vN}$ is a von Neumann type fully coisometric CK family.
\item
$\rS_d$ is a dilation type fully coisometric CK family.
\end{enumerate}

If $\fS$ is regular, then the absolutely continuous summand $\rS_a$ coincides with the analytic part $\rS_{an}$, and the largest singular summand $\rS_s$ coincides with the von-Neumann summand $\rS_{vN}$. Thus, our decomposition can be written instead as
\[ 
\rS \cong \rS_l \oplus \rS_a\oplus \rS_s \oplus \rS_d .
\]
\end{theorem}

\begin{proof}
By the Wold-decomposition, we may write $\rS =\rS_l \oplus \rU$ with respect to a decomposition $\cH = \cH_l \oplus \cH'$, where $\rS_l$ is a left-regular TCK family, and $\rU$ is a fully coisometric CK family on $\cH'$.
Let $P$ be the structure projection for $\rS$; and denote $\cV:= P \cH$. 
Note that since $\rS_l$ is analytic, $\cV$ is orthogonal to $\cH_l$.

Define reducing subspaces 
\[ \cM = \fS[\cV] \qand \cN = \fS^*[\cV^\perp] .\]
Since $\cH_l$ is a reducing subspace contained in $\cV^\perp$, we see that $\cM$ is orthogonal to $\cH_l$.
Define
\[ \cH_{an} = (\cN \cap \cM^\perp) \ominus \cH_l, \quad \cH_{vN} = \cM \cap \cN^\perp \qand \cH_d = \cM \cap \cN. \]
Set $\rS_{an}= \rS|_{\cH_{an}}$, $\rS_{vN}= \rS|_{\cH_{vN}}$ and $\rS_d = \rS|_{\cH_d}$.

Because $\cM \supseteq \cV$ and $\cN \supseteq \cV^\perp$, it follows that the projections $P_\cM$ and $P_\cN$ commute with each other and with $P$. 
We compute that 
\[
 \cM^\perp \cap \cN^\perp \subseteq \cV^\perp \cap \cV = \{0\} .
\]
Therefore,
\begin{align*}
 \cH &= (\cN \cap \cM^\perp) \oplus (\cM \cap \cN^\perp) \oplus (\cM \cap \cN)\\
  &= (\cH_l  \oplus \cH_{an}) \oplus \cH_{vN} \oplus \cH_d . 
\end{align*}
Note that $\cH_{an}$ is contained in $\cV^\perp$, and hence $\rS_{an}$ has analytic type.
As the left regular summand has been removed from $\cH_{an}$, what remains is a fully coisometric CK family.
Next observe that $\cH_{vN}$ is contained in $\cV$, and therefore $\rS_{vN}$ is von Neumann type.
To verify that $\rS_d$ has dilation type, we apply Lemma~\ref{L:dilation type}.
Let $\cV_1 = P\cH_d$ and $\cW_1 = P^\perp \cH_d$.
Then $\cH_d = \cV_1 \oplus \cW_1$ because $P_\cM$ and $P_\cN$ commute with $P$.
Because $\cV\ominus \cV_1 = \cH_{vN}$ is reducing and likewise $\cV^\perp \ominus \cW_1 = \cH_{an}$ is reducing, we obtain
\[
 \fS[\cV_1] \cap \cV^\perp = \fS[\cV] \cap \cV^\perp = \cW_1 \qand
 \fS^*[\cW_1] \cap \cV = \fS^*[\cV^\perp] \cap \cV = \cV_1 .
\]
Therefore $\cH_d$ is dilation type.

Now suppose that $\rS$ is regular.
By definition, the absolutely continuous subspace $\cV_{ac}$ coincides with $P^\perp\cH$.
Thus, the largest absolutely continuous summand coincides with the largest analytic summand, namely $\rS_l \oplus\rS_{an}$.
So the largest fully coisometric CK absolutely continuous summand $\rS_a$ is equal to $\rS_{an}$.
Also, the von Neumann part is contained in $\cV$ and thus is singular; so $\rS_s = \rS_{vN}$. Hence, we obtain the decomposition as in the second part of the statement.
\end{proof}

The following proposition discusses the permanence of the various parts of the decomposition.

\begin{proposition} \label{P:intrinsic}
Let $G$ be a graph, and let $\rS, \rT$ be TCK families generating free semigroupoid algebras $\fS, \fT$ on $\cH_{\rS}, \cH_{\rT}$ respectively. 
The summands of a free semigroupoid algebra of left regular, absolutely continuous and singular type are intrinsic in the sense that
\[ \rS_l \oplus \rT_l = (\rS\oplus \rT)_l,\quad \rS_a \oplus \rT_a= (\rS\oplus \rT)_a\qand \rS_s \oplus \rT_s = (\rS\oplus \rT)_s .\]
Furthermore, if $\rS$ is regular, then the summands of analytic, von Neumann and dilation type are unchanged by adding a direct summand to $\rS$ in the sense that
\[ \rS_{an} = (\rS \oplus \rT)_{an}|_{\cH_\rS}, \quad \rS_{vN} = (\rS \oplus \rT)_{vN}|_{\cH_\rS} \qand \rS_d = (\rS \oplus \rT)_d|_{\cH_\rS} .\]
Finally, if both $\rS$ and $\rT$ are regular, then $\rS \oplus \rT$ is regular and 
\[
\rS_{an} \oplus \rT_{an}= (\rS\oplus \rT)_{an}\qand \rS_{vN} \oplus \rT_{vN} = (\rS\oplus \rT)_{vN}.
\]
\end{proposition}

\begin{proof} 
The left regular part is determined by the Wold decomposition, and depends on the complement of the ranges of edges with range $v$. A direct sum just produces the direct sum of the left-regular spaces, which is the left-regular part.
The absolutely continuous part is supported on the subspace $\cV_{ac}$ of absolutely continuous vectors, which by Corollary~\ref{C:largest-ac} is preserved under direct sums. Thus, it is unaffected by other summands. Therefore, the complementary singular part is also intrinsic.

Next, when $\rS$ is regular, we see that $(\rS \oplus \rT)_a = \rS_a \oplus \rT_a = \rS_{an} \oplus \rT_a$, and clearly the analytic part of $\rS \oplus \rT$ includes $\rS_{an}$ as a summand and is contained in the absolutely continuous part of $\rS \oplus \rT$.  
Hence, $(\rS \oplus \rT)_{an}|_{\cH_{\rS}} = \rS_{an}$. 
It similarly follows that the same holds for von Neumann and dilation types.

Finally, when both $\rS$ and $\rT$ are regular, it is clear from condition (M) of Theorem \ref{T:abs_cont} that so is $\rS \oplus \rT$, and the rest follows from the first part.
\end{proof}

\section{Applications} \label{S:applications}

In this section we provide three applications of the theory we have developed. We show that every free semigroupoid algebra is reflexive, that every regular free semigroupoid algebra satisfies a Kaplansky-type density theorem, and that isomorphisms of nonself-adjoint free semigroupoid algebras can recover any transitive, row-finite graph from their isomorphism class.

First, we apply our structure theorem to obtain reflexivity of free semigroupoid algebras. Recall that a subalgebra $\fA \subseteq B(\cH)$ is \emph{reflexive} if $\fA = \Alg\Lat \fA$ where
\[
\Alg\Lat \fA = \{ \ T\in B(\cH) \ | \ P^{\perp}TP = 0 \ \text{for every} \ P\in \Lat \fA \ \}
\]
and $\Lat \fA$ is the lattice of closed invariant subspaces for $\fA \subseteq B(\cH)$. 
Kribs and Power \cite[Theorem 7.1]{KP2004} show that the left regular free semigroupoid algebra $\fL_G \subseteq B(\cH_G)$ is reflexive for every directed graph $G$. 
A careful look at their proof shows that they actually show that $\fL_{G,v} \subseteq B(\cH_{G,v})$ is reflexive for any vertex $v$.
By taking direct sums, we see that every left regular type free semigroupoid algebra is reflexive. 
We now leverage this result to obtain reflexivity of free semigroupoid algebras.

\begin{theorem} \label{T:reflexivity}
Let $\fS$ be a free semigroupoid algebra for a graph $G$ generated by a TCK family $\rS=(S_v,S_e)$ on $\cH$. 
Then $\fS \subseteq B(\cH)$ is reflexive.
\end{theorem}

\begin{proof}
By the double commutant theorem, every von Neumann algebra is reflexive. 
Thus,
\[ \Alg\Lat \fS \subseteq  \Alg\Lat \fM = \fM .\]
Let $P$ be the structure projection of $\fS$ from Theorem \ref{thm:structure}. 
Since $P^\perp\cH$ is invariant for $\fS$, it is also invariant for any $T \in \Alg\Lat \fS$.
Thus, we may write $T = TP + P^\perp T P^\perp$, where $T \in \fM$. 
The Structure Theorem also shows that $TP \in \fM P = \fS P \subseteq \fS$.
Hence, the problem reduces to showing that the analytic part, $\fS|_{P^\perp\cH}$, is reflexive.
So we may suppose that $P=0$ and $\fS$ is analytic.
The key to our proof is the fact that $\cH$ is spanned by wandering vectors, another consequence of the Structure Theorem.

Observe that since $S_v\in \fS$ and $T = \wotsum_{v\in V} TS_v$, it suffices to show that each $TS_v$ belongs to $\fS$ when $T \in \Alg\Lat \fS$. Denote $\cH_w = P_w \cH$ for each $w\in V$. Next, fix a vertex $v\in V$ such that $S_v \ne 0$, and note that $\cH_v$ is spanned by wandering vectors, since if $x$ is wandering, so is $S_v x$.
Let $G[v]$ be the directed subgraph generated by $v$ on the vertices $V_v$.
Then there is no harm is restricting our attention to $\fS|_\cK$ where $\cK = \bigoplus_{w\in V_v} \cH_w$. Indeed, as $TS_v(\cK^{\perp}) = \{0\}$, if $TS_v|_{\cK} \in \fS|_{\cK}$ then we must have that $TS_v \in \fS$. Thus, as $\fS$ is analytic we assume without loss of generality that $\cH = \cK$ and that there is a canonical isomorphism $\phi:\fL_{G,v}\to \fS$.

Let $x_0$ be a fixed wandering vector in $\cH_v$, and let $x_1$ be any other wandering vector in $\cH_v$.
Let $\cM_i = \fS[x_i]$, and let $V_i: \cH_{G,v} \to \cM_i$ be the canonical isometry given by $V_i\xi_\mu = S_\mu x_i$.
Then $\phi_i(A) = V_i A V_i^*$ is a unitary equivalence of $\fL_{G,v}$ onto $\fS|_{\cM_i}$, and moreover $\phi_i(A) = \phi(A)|_{\cM_i}$.
Now $\fL_{G,v}$ is reflexive by \cite[Theorem 7.1]{KP2004} (and the observation preceding this proof).
Thus, $TS_v$ leaves $\cM_i$ invariant, and $TS_v|_{\cM_i} \in \Alg\Lat \fS|_{\cM_i} = \fS|_{\cM_i}$.
Therefore, there are operators $A_i\in\fL_{G,v}$ so that $TS_v|_{\cM_i} = \phi_i(A_i)$.
By replacing $TS_v$ by $TS_v - \phi(A_0)$, we may also suppose that $A_0=0$.
If we can show that $A_1=0$ for any wandering vector $x_1$, then it follows that $TS_v=0$, and hence it lies in $\fS$.
Now $A_1 = A_1 L_v$ has a Fourier series determined by the coefficients of $\eta = A_1\xi_v$; and if $\eta=0$, then we would obtain that $A_1=0$. Since $V_1$ is an isometry, this is further reduced to showing that $y: =V_1\eta =0$.

Let $V_t = V_0 + tV_1$ for $0 < t \le 0.5$. 
Observe that $V_i A = \phi(A) V_i$ for $i=0,1$ and $A \in \fL_{G,v}$, and therefore $V_t A = \phi(A) V_t$ holds as well; i.e. $V_t$ are intertwining maps.
Moreover, the operator $V_t$ is bounded below by $1-t \ge .5$, and therefore has closed range $\cM_t = V_t \cH_{G,v}$.
The fact that $V_t$ intertwines $\fS$ and $\fL_{G,v}$ shows that $\cM_t$ is invariant for $\fS$, and hence for $TS_v$.
Let $x_t = V_t \xi_v = x_0 + tx_1$, and note that
\[ y = V_1\eta = V_1 A_1 \xi_v = \phi_1(A_1)V_1 \xi_v = TS_v x_1 .\]
Then, $y \in \cM_1$ and since $T|_{\cM_0} = 0$, we have that $\cM_t$ contains
\[
 TS_v x_t = Tx_0 + tTx_1= ty .
\]
We next explain why $y$ lies in $\cM_0$. Indeed,
\[
t y = TS_v x_t = TS_v (V_0 + t V_1) \xi_v = V_t (A_0 + t A_1)\xi_v,
\]
so that with $\zeta_t = \frac{1}{t} (A_0 + t A_1)\xi_v \in \cH_{G,v}$ we get that $V_t\zeta_t = y$. As $V_t$ is bounded below by $1-t$, we have $\|\zeta_t\| \le (1-t)^{-1} \|y \| \le 2 \|y \|$.
Thus, for $0 < s < t < 1$, 
\[ \| \zeta_s - \zeta_t \| \leq 2 \| V_t (\zeta_s - \zeta_t) \| \|y\| = 2 \| (t-s) V_1 \zeta_s \| \|y \| = 2 (t-s) \|\zeta_s\| \|y \| \le 4(t-s) \|y\|^2. \]
So $\zeta_t$ is Cauchy as $t\to 0$ with limit $\zeta_0$, and 
\[ y = \lim_{t\to 0} V_t \zeta_t = V_0 \zeta_0 \in \cM_0 .\]
Therefore $y \in \cM_0 \cap \cM_1$.

We wish to show that $y =0$, and there are three cases depending on how many irreducible cycles there are at $v$. 
If $v$ does not lie on any cycles, then every vector in $\cH_v$ is wandering.
We may select an orthonormal basis $\{ x_k: 0 \le k < \Dim \cH_v \}$ for $\cH_v$.
Then $\fS[x_k]$ are pairwise orthogonal. 
Since $y \in \fS[x_0]\cap \fS[x_k]=\{0\}$ for each $x_k$, $k\ge1$, we obtain that $TS_v|_{\fS[x_k]} = 0$.
Since $\cK = \bigoplus_{k\ge0} \fS[x_i]$, we deduce that $TS_v$ belongs to $\fS$.

If $v$ lies on two or more irreducible cycles, say $\mu_1$ and $\mu_2$ in $\bFGv$, then $x_{0i} = S_{\mu_i}x_0$ are
also wandering vectors such that $\cM_{0i} := \fS[x_{0i}]$ are pairwise orthogonal subspaces of $\cM_0$.
Thus, we have that $TS_v|_{\cM_{0i}} = 0$. 
The same analysis as above can be repeated for both of these invariant subspaces.
We deduce that $y\in\cM_{01}\cap\cM_{02}=\{0\}$. 
Again we see that $TS_v$ belongs to $\fS$.

The final case occurs when there is exactly one irreducible cycle $\mu$ at $v$. Here we use the fact that $\fS$ is analytic to see that $S_\mu$ is unitarily equivalent to $U_+^{(\alpha)} \oplus U$
where $U$ is a unitary with spectral measure absolutely continuous to Lebesgue measure $m$, and
if $\alpha = 0$, then the spectral measure is equivalent to $m$ (Theorem~\ref{T:cycle_structure}).
In the latter case, $U$ has a reducing subspace on which it is unitarily equivalent to the bilateral shift,
and hence has an invariant subspace on which it is unitarily equivalent to the unilateral shift $U_+$.
So in either case, we have a subspace $\cM_0$ of $\cH_v$ with an orthonormal basis $y_k$ for $k\ge0$ such that $S_\mu y_k = y_{k+1}$.
Observe that each $y_k$ is a wandering vector for $\fS_v$, as this is the algebra generated by $S_\mu|_{\cH_v}$.
By Lemma~\ref{P:extend wandering}, the $y_k$ are also wandering vectors for $\fS$.
Now we take $x_0=y_0$. 
As in the previous paragraph, we observe that we could use $y_k$ in place of $y_0$.
By intersecting over the information from all $k\ge1$, it follows that $y$ lies in $\bigcap_{k\ge0} \fS[y_k] = \{0\}$.
Again we find that $TP_x = 0 \in \fS$.
Thus, $\fS$ is reflexive.
\end{proof}

Our next goal is to establish an analogue of the Kaplansky density theorem akin to \cite[Theorem 5.4]{DLP} for regular free semigroupoid algebras, or alternatively those free semigroupoid algebras whose TCK families satisfy condition (M).

\begin{theorem} \label{T:Kaplansky}
Let $G$ be a graph, and let $\rS$ be a regular TCK family generating the free semigroupoid algebra $\fS$.
Let  $\pi_{\rS}$ be the corresponding $*$-representation of $\cT(G)$.
Then the unit ball of $M_n(\pi_{\rS}(\cT_+(G)))$ is weak-$*$ dense in the unit ball of $M_n(\fS)$.
\end{theorem}

\begin{proof}
Fix a separable Hilbert space $\cH$, and let $\pi_u$ denote the direct sum of all $*$-representations of $\cT(G)$ on $\cH$.
Then the universal von Neumann algebra $\fM_u = \pi_u(\cT(G))''$ is canonically $*$-isomorphic and weak-$*$ homeomorphic to $\cT(G)^{**}$.
It is a standard result that every representation $\pi$ of $\cT(G)$ extends uniquely to a weak-$*$ continuous $*$-representation $\bar\pi$
of $\fM_u$, and that $\bar\pi(\fM_u) = \pi(\cT(G))''$. 
Moreover, there is a central projection $Q_\pi$ so that $\ker\bar\pi = \fM_u Q_\pi^\perp$.
The representation $\pi_u$ is unitarily equivalent to its infinite ampliation.
Hence, the \wot\ and weak-$*$ topologies coincide on $\fM_u$.

Since $\cT_+(G)$ is a closed subalgebra of $\cT(G)$, basic functional analysis shows that $\cT_+(G)^{**}$ may be naturally identified
with the weak-$*$ closure $\fS_u$ of $\pi_u(\cT_+(G))$ in $\fM_u$. 
This may be considered to be the universal free semigroupoid algebra for $G$.
We apply the Structure Theorem~\ref{thm:structure} to $\fS_u$ to obtain the structure projection $P_u$ so that 
\[  \fS_u = \fM_u P_u + P_u^\perp \fS_u P_u^\perp \qand P_u^\perp \fS_u P_u^\perp = \fS_u P_u^\perp \cong \fL_G .\]
By Goldstine's Theorem, the unit ball of $\cT_+(G)$ is weak-$*$ dense in the unit ball of $\cT_+(G)^{**}$;
whence the unit ball of $\pi_u(\cT_+(G))$ is weak-$*$ dense in the unit ball of $\fS_u$.

Let $\tilde\pi_{\rS}$ be the restriction of $\bar\pi_{\rS}$ to $\fS_u$.
Because $\tilde\pi_{\rS}$ is weak-$*$ continuous, it will map $\fS_u$ into $\fS$.
Let $P$ be the structure projection for $\fS$.
We wish to show that $P =  \bar\pi_{\rS}(P_uQ_{\pi_{\rS}}) = \bar\pi_{\rS}(P_u)$. 
From the Structure Theorem \ref{thm:structure}(6), we know that both $P_u^{\perp} \cH_u$ and $P^{\perp} \cH_{\pi_{\rS}}$ are spanned by wandering vectors, so that $\bar\pi_{\rS}(P_u^{\perp}) \geq P^{\perp}$. 
On the other hand, by Proposition~\ref{P:analytic-implies-ac} we know that the analytic part $P_u^\perp\cH_u$ of $\fS_u$ is absolutely continuous. 
Hence, the subspace $P_u^{\perp}\cH_{\pi_{\rS}}$ is absolutely continuous for $\fS$ and contains the analytic part $P^{\perp}\cH_{\pi_{\rS}}$. 
Since $\fS$ is regular, these spaces coincide, so that $\bar\pi_{\rS}(P_u^{\perp}) = P^{\perp}$.

If $\fS$ is a von Neumann algebra, then the structure projection is $P=I$. 
Since $\bar\pi_{\rS}(P_u^\perp Q_{\pi_{\rS}}) = P^\perp = 0$, we see that $Q_{\pi_{\rS}} \le P_u$. 
Thus, $Q_{\pi_{\rS}} \in \fM_u P_u$, and hence belongs to $\fS_u$.
The map taking $X\in\fS_u$ to $XQ_{\pi_{\rS}} \in \fM_u Q_{\pi_{\rS}} = \fS$ is clearly a normal complete quotient map.

Otherwise, $P^\perp \ne 0$. In this case the restriction of  $\tilde\pi_{\rS}$ to the corner $\fS_u P_u^\perp$ factors as
\[  \fS_u P_u^\perp \cong \fL_G \to \fL_{G[V_w]} \cong \fS P^\perp ,  \]
where the map in the middle is the complete quotient map $\rho_{V,V_w}$ of Theorem~\ref{T:left_quotient}, where $V_w$ are the vertices $v$ such that $S_v \notin \fS_0$, as in the Structure Theorem \ref{thm:structure}. 
In particular, $\bar\pi_{\rS}|_{\fS_u P_u^{\perp}}$ is surjective onto $\fS^{\perp}$. Furthermore, the restriction of $\tilde\pi_{\rS}$ to the left ideal $\fM_u P_u$ is given by multiplication by $Q_{\rS}$ and maps onto $\fM P$. 
Combining these two together we obtain that $\tilde\pi_{\rS}$ maps $\fS_u$ onto $\fS$.

For $X \in \fS$, we write $X = XP + XP^\perp$. 
As $P =  \bar\pi_{\rS}(P_uQ_{\rS})$, we get that $\fM P$ is canonically $*$-isomorphic to $\fM_u P_u Q_{\rS}$.
Thus, there is an element $Y_1 = Y_1P_u Q_{\rS} \in \fM_u$ such that $\bar\pi_{\rS}(Y_1) = XP$.
The corner $XP^\perp$ is identified with an element $A \in \fL_{G[V_w]}$.
Let $B = j_{V_w,V}(A)$ be the image of $A$ in $\fL_G$ via the completely isometric section $j_{V_w,V}$ as in Theorem~\ref{T:left_quotient}.
Let $Y_2$ be the element of $\fS_u P_u^\perp$ which is identified with $B$.
Then $Y = Y_1+Y_2\in\fS_u$ and $X = \tilde\pi_{\rS}(Y) = YQ_{\rS}$.
Moreover, by construction, $Y-YQ_{\rS} = Y_2 Q_{\rS}^\perp$.
Therefore, since $Q_{\rS}$ is central,
\begin{align*}
  \|Y\| & = \max\{ \|YQ_{\rS}\|, \|YQ_{\rS}^\perp\| \} \\
	&= \max\{ \|YQ_s \| , \|Y_2 Q_s^{\perp} \| \} \\
  &= \max\{ \|X\|, \|B\| \} \\
  &=  \max\{ \|X\|, \|A\| \} = \|X\| .
\end{align*}
It follows that $\tilde\pi_{\rS}$ is a quotient map of $\fS_u$ onto $\fS$. The identical argument works for matrices over $\fS_u$, and thus $\tilde\pi_{\rS}$ is a complete quotient map. 

Finally, let $X$ belong to the unit ball of $\fS$. Pick a $Y$ in the unit ball of $\fS_u$ such that $\bar\pi_{\rS}(Y) = X$.
Using Goldstine's Theorem, select a net $A_\alpha$ in the unit ball of $\cT_+(G)$ so that $\pi_u(A_\alpha)$ converges weak-$*$ to $Y$.
Then $\bar\pi_{\rS} \pi_u(A_\alpha) = \pi_{\rS}(A_\alpha)$ is a net in the unit ball of $\pi_{\rS}(\cT_+(G))$ which converges weak-$*$ to $X$. The same argument then works for any matrix ampliation.
\end{proof}

\begin{example}
This theorem is false for some free semigroupoid algebras of a cycle. 
For simplicity, consider the graph $C_1$ on a single vertex $v$ with one edge $e$ from $v$ to itself.
Then $\cT_+(C_1)$ is naturally isomorphic to $A(\bD)$.
Fix a subset $A \subset \bT$ with $0< m(A) < 1$, and consider the representation $\pi_A$ on $L^2(A)$ sending $S_e$ to $M_{z,A}$ as in Example~\ref{E:cycle_abs_cont}.
Then the free semigroupoid algebra is $L^\infty(A)$.
However, the unit ball of $\pi_A(A(\bD))$ cannot be weak-$*$ dense in the unit ball of $L^\infty(A)$.
Indeed, if this were true, then for each $f\in L^\infty(A)$ of norm 1, there would be a sequence $h_n$ in the unit ball of $A(\bD)$ such that $\pi_A(h_n)$ converges weak-$*$ to $M_f$. 
In particular, $h_n$ converges to $f$ $m$ a.e.\ on $A$.
Considering this as a sequence in the unit ball of $H^\infty$, there would be a subsequence which also converges weak-$*$ to some $h$ in the unit ball of $H^\infty$.
This would mean that in $\cB(L^2(\bT))$ that $M_{h_n}$ converges weak-$*$ to $M_h$, and in particular $h_n$ converges to $h \ae (m)$ on $\bT$.
By restriction to $A$, we can conclude that $f=h|_A$.
But not every $f\in L^\infty(A)$ is the restriction of an $H^\infty$ function, let alone one of the same norm. As an example we can take $B\subseteq A$ with $0 < m(B) < m(A)$, and define $f= \upchi_B$. If there was $h\in H^\infty$ such that $h|_A = \upchi_B$, we see that $h$ is a non-zero function in $H^{\infty}$ which vanishes on a subset of $\bT$ with non-zero measure. This contradicts the F. and M. Riesz Theorem (see \cite[Theorem 6.13]{Dou98}). 
\end{example}

Finally, we conclude this section by obtaining some isomorphism theorems for nonself-adjoint free semigroupoid algebras of transitive row-finite graphs. We recall the definition of the Jacobson radical of a Banach algebra.

\begin{definition}
Let $\cA$ be a Banach algebra. The Jacobson Radical of $\cA$, denoted $\rad(\cA)$, is the intersection of kernels of all algebraically irreducible representations of $\cA$. We say that $\cA$ is semi-simple if $\rad(\cA) = \{0\}$.
\end{definition}

It is well-known that all elements of $\rad(\cA)$ are quasi-nilpotent, that any nil ideals of $\cA$ are contained in $\rad(\cA)$, and that $\rad(\cA)$ is an algebraic isomorphism invariant. We refer to Bonsall and Duncan \cite{BD73} for these results and more. 

\begin{proposition} \label{P:semisimple}
Let $\fS$ be a free semigroupoid algebra for a transitive graph $G$ generated by a TCK family $\rS= (S_v,S_e)$ on $\cH$.
Let $\mathfrak{M} =W^*(\rS)$ be the von Neumann algebra generated by $\rS$, and let $P$ be the structure projection given in Theorem~$\ref{thm:structure}$. 
Then the Jacobson radical of $\fS$ is $P^{\perp} \fS P = P^{\perp} \fM P$.
\end{proposition}

\begin{proof}
Clearly $P^{\perp}\fS P$ is a nil ideal, and is hence contained in the Jacobson radical. On the other hand, by \cite[Theorem 5.1]{KP2004} $\fL_G$ is semisimple because $G$ is transitive.
This equals $P^{\perp} \fM P$ by the Structure Theorem~\ref{thm:structure}(1).
Thus, the analytic type part $\fS P^\perp$ is semisimple. Therefore, the quotient 
\[ \fS/ P^{\perp}\fS P  \cong P \mathfrak{M} P \oplus \fS P^{\perp} \]
is a direct sum of semisimple operator algebras, and hence is semisimple. 
Thus, $P^{\perp} \fS P$ is the Jacobson radical of $\fS$.
\end{proof}

\begin{lemma} \label{lem:finite-idemp}
Let $G$ be a countable directed graph. Then every idempotent $0 \neq P \in \fL_G$ must be finite. 
That is, if $Q\in \fL_G$ is an idempotent such that $PQ=QP=Q$ and $P$ and $Q$ are algebraically equivalent, 
then $P=Q$.
\end{lemma}

\begin{proof}
Let $\Phi_0 : \fL_G \rightarrow \ell^{\infty}(V)$ be the canonical conditional expectation, which is also a homomorphism. 
Then for every non-zero idempotent $Q \in \fL_G$, we have that $\Phi_0(Q) \neq 0$. 
Indeed, if not, then $Q=Q^k \in \fL_{G,0}^k$ for all $k\geq 0$ and hence $Q$ must be $0$.

Next, suppose that $A,B\in \fL_G$  satisfy $AB=P$ and $BA=Q$. 
Since $\Phi_0$ is  a homomorphism of $\fL_G$ into a commutative algebra, we see that 
\[ \Phi_0(P) = \Phi_0(A) \Phi_0(B) = \Phi_0(B) \Phi_0(A) =  \Phi_0(Q) .\]
Hence $\Phi_0(P-Q) = 0$. 
Our assumptions guarantee that $P-Q$ is an idempotent, and therefore $P=Q$ by the first paragraph.
\end{proof}

Next we will show that when the graphs are transitive, the structure projection $P$ appearing in Theorem \ref{thm:structure} is intrinsic to the free semigroupoid algebra. 
This will allow us to recover the underlying graph from an algebraic isomorphism class.

\begin{theorem} \label{thm:internal-structure-proj}
Let $\fS_1$ and $\fS_2$ be \emph{nonself-adjoint} free semigroupoid algebras for transitive graphs $G_1$ and $G_2$.
Suppose that $\phi: \fS_1 \rightarrow \fS_2$ is an algebraic isomorphism. 
Let $P_1$ and $P_2$ be the structure projections of $\fS_1$ and $\fS_2$ as in Theorem~$\ref{thm:structure}$. 
Then there is an invertible element $V \in \fS_2$ such that $\hat{\phi}:= \Ad_V \circ \phi$ satisfies $\hat{\phi}(P_1) = P_2$. 
Hence, there is an algebraic isomorphism between $\mathfrak{L}_{G_1}$ and $\mathfrak{L}_{G_2}$.
\end{theorem}

\begin{proof}
Since both $\fS_1$ and $\fS_2$ are nonself-adjoint, we see that $P^{\perp}_1$ and $P^{\perp}_2$ are both non-zero. 
Let $\fN_1 =P_1\fS_1 P_1$ and $\fN_2= P_2 \fS_2 P_2$ be the von Neumann algebra corners of the free semigroupoid algebras, as in Theorem \ref{thm:structure}. 
The isomorphism must carry the Jacobson radical of $\fS_1$ onto the radical of $\fS_2$.
Therefore, by Proposition \ref{P:semisimple}, there is an induced isomorphism between the quotients by the Jacobson radicals. That is, there is an isomorphism $\widetilde{\phi}$ from $\fN_1 \oplus \fS_1 P_1^{\perp}$ onto $\fN_2 \oplus \fS_2 P_2^{\perp}$. However, as $G_1$ and $G_2$ are transitive, we get that $\fS_i P_i^{\perp}$ is completely isometrically and weak-* homeomorphically isomorphic to $\fL_{G_i}$ for $i=1,2$.
In particular, $\widetilde{\phi}$ preserves the centers, so we get that $\widetilde\phi$ carries $Z(\fN_1) \oplus \bC P_1^\perp$ onto $Z(\fN_2) \oplus \bC P_2^\perp$.

Note that $P_i^\perp$ is a minimal projection in the respective centers.
So $\widetilde\phi(P_1^\perp)$ is a minimal projection in $Z(\fN_2) \oplus \bC P_2^\perp$.
If this is not $P_2^\perp$, then it is a minimal projection $Q\in Z(\fN_2)$.
However, this would imply an isomorphism between $\fL_{G_1}$ and a von Neumann algebra $Q \fN_2 Q$.
Since $G_1$ is transitive, $\fL_{G_1}$ is infinite dimensional, and the ideal $\fL_{G_1,0}$ properly contains its square $\fL_{G_1,0}^2$.
However, every ideal of any C*-algebra is equal to its own square; so $\fL_{G_1}$ is not isomorphic to a von Neumann algebra.
Hence, $\widetilde\phi(P_1^\perp) = P_2^\perp$.

This means that $\phi(P_1^{\perp}) = P_2^{\perp} + X$ where $X = P_2^{\perp}XP_2$ lies in the Jacobson radical of $\fS_2$ and satisfies $X^2 = 0$. 
Hence, taking $V = I + X$, we see that $V^{-1} = I - X$ and $V\phi(P_1^{\perp})V^{-1} = P_2^{\perp}$.
The algebraic isomorphism $\hat{\phi} : \fS_1 \rightarrow \fS_2$ given by $\hat{\phi}(A) = V \phi(A)V^{-1}$ maps $P_1^{\perp}$ to $P_2^{\perp}$, and thus $\hat{\phi}(P_1) = P_2$. 

Thus, this yields an algebraic isomorphism between $\fL_{G_1} \cong \fS_1 P_1^{\perp}$ and $\fL_{G_2} \cong \fS_2 P_2^{\perp}$.
\end{proof}

\begin{corollary} \label{C:isomorphisms}
Let $G_1$ and $G_2$ be transitive \emph{row-finite} graphs. The following are equivalent
\begin{enumerate} [label=\normalfont{(\arabic*)}]
\item
$G_1$ is isomorphic to $G_2$.
\item 
$\mathfrak{L}_{G_1}$ and $\mathfrak{L}_{G_2}$ are completely isometrically isomorphic and weak-$*$ homeomorphic.
\item
There exist two algebraically isomorphic \emph{nonself-adjoint} free semigroupoid algebras $\fS_1$ and $\fS_2$ for $G_1$ and $G_2$ respectively.
\end{enumerate}
\end{corollary}

\begin{proof}
Clearly if $G_1$ and $G_2$ are isomorphic, then $\mathfrak{L}_{G_1}$ and $\mathfrak{L}_{G_2}$ are completely isometrically isomorphic and weak-$*$ homeomorphic. So (1) implies (2).
$\mathfrak{L}_{G_1}$ and $\mathfrak{L}_{G_2}$ are examples of nonself-adjoint free semigroupoid algebras, so (2) implies (3).

Finally, by Theorem \ref{thm:internal-structure-proj}, an algebraic isomorphism between $\fS_1$ and $\fS_2$ yields an algebraic isomorphism between $\mathfrak{L}_{G_1}$ and $\mathfrak{L}_{G_2}$. 
By \cite[Corollary 3.17]{KK2004} $G_1$ and $G_2$ are isomorphic, so that (3) implies (1).
\end{proof}

\begin{corollary}
Let $\fS$ be a free semigroupoid algebra of a transitive graph $G$ generated by a TCK family $S= (S_v,S_e)$ on $\cH$. 
If $\fS$ is algebraically isomorphic to $\fL_G$, then $\fS$ is analytic type.
\end{corollary}

\begin{proof}
By Theorem \ref{thm:internal-structure-proj} there is an isomorphism $\hat{\phi} : \fS \rightarrow \fL_G$ that maps the structure projection of $\fS$ to the structure projection of $\fL_G$, namely to $0$. Hence, the structure projection of $\fS$ is $0$, so that $\fS$ is weak-$*$ homeomorphic completely isometrically isomorphic to $\fL_G$ via the canonical surjection as in Theorem \ref{thm:structure}, and $\fS$ is analytic type.
\end{proof}

\section{Self-adjoint free semigroupoid algebras} \label{S:self-adjoint}

In spite of the nonself-adjoint analysis of our algebras, a free semigroupoid algebra can occasionally turn out to be a von Neumann algebra. We have seen this phenomenon occur for cycle algebras, and Corollary \ref{C:fS=fS_0} assures us that this can only occur when the supported graph for the TCK family is a disjoint union of transitive components. An example of Read \cite{Read} (c.f. \cite{D2006}) shows that free semigroup algebras can also be self-adjoint, and this is our starting point. Our goal in this section is to provide a large class of graphs on which this is also possible.

\begin{theorem}[Read] \label{T:Read}
There exist isometries $Z_1,Z_2$ with complementary pairwise orthogonal ranges so that the free semigroup algebra generated by this family is $B(\cH)$.
\end{theorem} 

For any word $\gamma=\gamma_1\cdots\gamma_n\in\bF_2^+$ of $\{1,2\}$, we denote $Z_\gamma=Z_{\gamma_1}Z_{\gamma_2}\cdots Z_{\gamma_n}$. 
It is further shown in \cite{D2006} that every rank one operator is a \wot-limit of a bounded sequence of operators $T_n\in\spn\{Z_\gamma: \gamma\in\bF_2^+, |\gamma|=2^n\}$. 
We first show for any $d\geq 2$, we can find a family of Cuntz-isometries $W_1,\cdots,W_d$, so that the free semigroup algebra generated by this family is $B(\cH)$. 
This provides a correct argument for \cite[Corollary 1.8]{D2006}. 

\begin{lemma}\label{L:ReadIso} 
For every $d\ge 3$, there exist Cuntz isometries $W_1,\cdots,W_d$ in $B(\cH)$ so that the \wot-closed algebra that they generate is $B(\cH)$. 
\end{lemma} 

\begin{proof} 
Let $Z_1,Z_2$ be a pair of Cuntz-isometries as in Read's example.
Define isometries $W_i$ by
\[
W_1 = Z_1 ,\ \ 
W_2 = Z_{21} ,\ \ 
 \dots  ,\ \ 
W_{d-1} = Z_{2\cdots 21} ,\AND 
W_d = Z_{2\cdots 22} .
\]

It is not hard to see that $\sum_{i=1}^d W_i W_i^* = I$. 
Let $\cA$ be the unital algebra generated by $\{W_1,\dots,W_d\}$. 
For every word $\gamma\in\bF_2^+$ of $\{1,2\}$, we claim that $Z_{\gamma 1}=Z_\gamma Z_1$ belongs to $\cA$. 
This can be shown by induction on $|\gamma|$. 
It is trivial when $\gamma$ is the empty word. 
If $\gamma$ starts with the letter $1$ and $\gamma=1\cdot \gamma'$, then $Z_{\gamma 1}=W_1 Z_{\gamma' 1}\in\cA$ by the induction. 
Otherwise, $\gamma$ starts with letter $2$. 
If $\gamma=2\cdots 21 \gamma'$, then we can factor $Z_{\gamma 1}=W_d^k W_i Z_{\gamma'1}\in\cA$ for some choice of $k\geq 0$ and $1\leq i\leq d-1$. 
If $\gamma=22\cdots 2$, then $Z_{\gamma 1}=W_d^k W_i\in\cA$ for some choice of $k\geq 0$ and $1\leq i\leq d-1$. 

Thus, for every $T_n\in\spn\{Z_\gamma: |\gamma|=2^n\}$, $T_n Z_1\in\cA$. 
Since every rank one operator $xy^*$ is a \wot-limit of a sequence of such $T_n$,  
\[ xy^* Z_1=x(Z_1^* y)^* \in \wotclos{\cA} . \]
However, $Z_1^*$ is surjective, and thus every rank one operator lies in $\wotclos{\cA}$. 
Since rank one operators are \wot-dense in $B(\cH)$, we have $\wotclos{\cA} = B(\cH)$. 
\end{proof} 

For free semigroupoid algebras of graphs that have more than one vertex, we would like to find examples that are self-adjoint. In the case of free semigroup algebras, which is just a free semigroupoid algebra where the graph has a single vertex and at least two loops, $B(\cH)$ is the only self-adjoint example known. We will construct a large class of directed graphs for which there exist free semigroupoid algebra that are self-adjoint.

\begin{definition} A directed graph $G$ is called \textit{in-degree regular} if the number of edges $e\in E$ with $r(e)=v$ is the same for every vertex $v\in V$. 
Similarly, we define \textit{out-degree regular} graphs.
A \emph{transitive} graph $G$ is called \textit{aperiodic} if for any two vertices $v,w\in V$ there exists a positive integer $K_0$ so that for any $K\geq K_0$ there exists a directed path $\mu$ with $|\mu|=K$ and $s(\mu)=v$ and $r(\mu)=w$.
\end{definition} 

Note that a transitive graph is $p$-periodic for $p\ge2$ if the vertices $V$ can partitioned into disjoint sets $V_1,\dots,V_p$ so that $s(e)\in V_i$ implies that $r(e)\in V_{i\!\!\pmod p +1}$.

Our main goal of this section is to show that the free semigroupoid algebra associated with an aperiodic, in-degree regular, transitive, finite directed graph can be self-adjoint. The graph in Figure~\ref{fig.graph} is an example of such graph on three vertices with in-degree $2$. 

\begin{figure}[hbt] 
\begin{center}
\begin{tikzpicture}

\draw[decoration={ markings,
  mark=at position 0.5 with {\arrow[line width=1.2pt]{>}}},
  postaction={decorate}] plot [smooth]
	coordinates {(0,1) (-0.3,1.4) (0,1.6) (0.3,1.4) (0,1)};
\draw[decoration={ markings,
  mark=at position 0.5 with {\arrow[line width=1.2pt]{>}}},
  postaction={decorate}] 
	(0,1) arc (90:180:1);
\draw[decoration={ markings,
  mark=at position 0.5 with {\arrow[line width=1.2pt]{>}}},
  postaction={decorate}]  
	(0,1) -- (-1,0);
\draw[decoration={ markings,
  mark=at position 0.5 with {\arrow[line width=1.2pt]{>}}},
  postaction={decorate}] 
	(-1,0) arc (-135:-45:1.42);
\draw[decoration={ markings,
  mark=at position 0.5 with {\arrow[line width=1.2pt]{>}}},
  postaction={decorate}]  
	(0,1) -- (1,0);
\draw[decoration={ markings,
  mark=at position 0.5 with {\arrow[line width=1.2pt]{>}}},
  postaction={decorate}] 
	(1,0) arc (0:90:1);
	
\node at (-1,0){$\bullet$};
\node at (1,0) {$\bullet$};
\node at (0,1) {$\bullet$};
\end{tikzpicture}
\end{center}
\label{fig.graph}
\caption{An aperiodic, in-degree regular, transitive,  directed graph on 3 vertices.}
\end{figure}
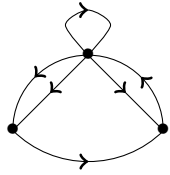

\begin{definition} 
Let $G=(V,E)$ be a directed graph. An \textit{edge colouring} of $G$ using $d$ colours is a function $c:E \to\{1,2,\cdots,d\}$. 
A colouring $c$ is called a \textit{strong edge colouring} if for every vertex $v\in V$, distinct edges with range $v$ have different colours.
\end{definition} 

Note that if $G$ has a strong edge colouring with $d$ colours, there can be at most $d$ edges with range $v$. 
Moreover, if $G$ is in-degree regular with in-degree $d$, this implies that for every colour $1\leq j\leq d$, there exists a unique edge $e$ with $r(e)=v$ and $c(e)=j$. 

An edge colouring $c$ of the directed graph induces a colouring of paths $c: \bF_G^+ \to \Fd$, where a directed path $\mu=e_k\dots e_1$ is given the colour $c(\mu)=c(e_k)\dots c(e_1)$ in the free semigroup $\bF_d^+$. It is then standard to show that for any vertex $v\in V$ and any word $\gamma$ of colours, there exists a unique path $\mu$ with $r(\mu)=v$ and $c(\mu)=\gamma$.

\begin{definition} \label{D:SyncWord}
Let $G$ be a finite graph with an edge colouring $c : E \rightarrow \{1,2,\dots,d\}$. A word of colours $\gamma =c_k \dots c_1 \in \bF_d^+$ is called a \textit{synchronizing word} for a vertex $v\in V$ if for any vertex $w\in V$ 
(including $v$ itself), there exists a path $\mu$ with $s(\mu)=v$ and $r(\mu)=w$ and $c(\mu)=\gamma$. 
\end{definition} 

In the graph theory literature, it is standard (and more natural) to reverse the direction of edges and consider out-degree regular graphs.
Then in a strong edge colouring, each vertex will have $d$ edges with source $v$ with distinct colours.
A word of colours then yields a unique path of that colour with source $v$. A synchronizing word is then a word in colours with the property that every path with that word of colours terminates at the same vertex, independent of the initial vertex.
The nomenclature \textit{synchronizing word} can be interpreted as saying that there are universal directions to get to a vertex $v$ by following the path of colour $\gamma$ independent of the starting point $w$.

When $G$ is a transitive, in-degree regular graph with a strong edge colouring $c$, then it is standard to show that there exists a vertex with a synchronizing word if and only if every vertex has a synchronizing word. Thus, for transitive in-degree regular graphs we can speak of a colouring being synchronizing without specifying the vertex. It is also easy to see that the existence of a synchronized colouring implies aperiodicity. The converse, known as the road colouring problem, was a highly coveted result that was eventually established by Trahtman \cite{Trahtman}.

\begin{theorem}[Trahtman] \label{T:Trahtman}
Let $G$ be an aperiodic in-degree $d$ regular transitive finite graph. 
Then $G$ admits a strong edge colouring with $d$ colours that has a synchronizing word.
\end{theorem}

A strong edge colouring enables us to define a CK family. 
Let $G$ be in-degree regular with a strong edge colouring $c$.
Let $W_1,\cdots,W_d$ be Cuntz isometries on $B(\cH)$, i.e. $W_i^*W_i=I$ for $1 \le i \le d$ and
\[ \sum_{j=1}^d W_j W_j^* = I. \]
Define $\cK=\bigoplus_{v \in V} \cH_v$, where each $\cH_v\cong\cH$ via a fixed unitary $J_v$. We will henceforth think of $J_v$ as a partial isometry from $\cK$ onto $\cH$ with cokernel $\cH_v$.
For each vertex $v\in V$, define $S_v = P_{\cH_v}$ to be the projection onto $\cH_v$. 
For each edge $e\in E$, with $s(e)=w$ and $r(e)=v$, define $S_e = J_v^* W_{c(e)}  J_w$. 
The family $\rS_c = (S_v,S_e)$ is the CK family associated with the strong edge colouring $c$ and the isometries $W_1,\dots,W_d$. 

Indeed, it is not hard to verify $\rS_c$ is a CK family. For each edge $e$, 
\[ S_e^* S_e = J_w^*J_w = S_w = S_{s(e)} \neq 0 .\]
and, because the $d$ distinct edges with $r(e)=v$ have distinct colours,
\[
 \sum_{r(e)=v} S_e S_e^* = \sum_{r(e)=v} J_v^* W_{c(e)} W_{c(e)} ^* J_v 
 = J_v^* \sum_{i=1}^d  W_i W_i^* J_v = J_v^*J_v = S_v .
\]
Therefore, $\rS_c$ is a CK family.

Let us fix an aperiodic, in-degree regular, transitive, finite directed graph $G$ of in-degree $d$. 
Use Trahtman's Theorem~\ref{T:Trahtman} to select a strong edge colouring which is synchronizing.
Let us also use Lemma \ref{L:ReadIso} to choose the isometries $W_1,\cdots,W_d \in B(\cH)$ so that the free semigroup algebra that they generate is $B(\cH)$. 
For each word $\gamma = c_m\dots c_1$ of colours, denote $W_\gamma = W_{c_m}\dots W_{c_1}$. 
Let $\fS_c$ be the free semigroupoid algebra that $\rS_c$ generates.

We now prove a lemma towards the main result.
 
\begin{lemma}\label{L:SyncDiag} 
Let $c$ be a strong edge colouring of an in-degree regular directed graph $G$. Suppose there exists a synchronizing word $\gamma$ for a vertex $v\in V$. Then $\fS_{c,v}:= S_v\fS_cS_v=B(\cH_v)$.
\end{lemma}

\begin{proof} 
For each word $\gamma'$ of colours, we have a unique word $\mu'$ with $r(\mu')=v$ and $c(\mu')=\gamma'$. 
Let $w=s(\mu')$. Since $\gamma$ is synchronizing, there exists a path $\mu$ with $s(\mu)=v$, $r(\mu)=w$ and $c(\mu)=\gamma$. 
Since $r(\mu)=w=s(\mu')$, the path $\lambda=\mu'\mu$ satisfies $s(\lambda)=r(\lambda)=v$. 
Therefore, $S_\lambda$ belongs to $\fS_{c,v} = S_v\fS_c S_v$. 
Notice that $S_\lambda = J_v^* W_{\gamma'} W_\gamma J_v$.
Therefore, for any word $\gamma'$ of colours, $J_v^* W_{\gamma'} W_\gamma J_v \in B(\cH_v)$ belongs to $\fS_c$. 

Since $\wotclos{\Alg}\{W_{\gamma'}: \gamma' \in \Fd \} = B(\cH)$, we see that $\fS_c$ contains $J_v^* A'W_\gamma J_v$ for any $A' \in B(\cH_v)$.
In particular, taking $A'=AW_\gamma^*$ for  $A \in B(\cH)$ shows that $\fS_{c,v} = B(\cH_v)$.
\end{proof}

We can now complete the construction.

\begin{theorem} \label{T:Read ex}
If $G$ is an aperiodic, in-degree regular, transitive, finite directed graph, then there exists a strong edge colouring $c$ so that the free semigroupoid algebra $\fS_c$ is $B(\cK)$. 
\end{theorem} 

\begin{proof} 
By Theorem~\ref{T:Trahtman} we have that $G$ has a strong edge colouring with a synchronizing word.
It follows that all vertices have synchronizing words. 
Hence by Lemma~\ref{L:SyncDiag}, we get that $S_v\fS_c S_v = B(\cH_v)$ for every vertex $v\in V$.

Let $v,w\in V$ be distinct vertices. 
There exists a path $\mu$ with $s(\mu)=w$ and $r(\mu)=v$.
Thus, $S_\mu$ maps $\cH_w$ isometrically into $\cH_v$. 
Since $B(\cH_v)$ is contained in $\fS_c$, we see that $A S_\mu$ lies in $\fS_c$ for all $A\in B(\cH_v)$. 
In particular, setting $A= B S_\mu^*$ for $B\in B(\cH_w,\cH_v)$ shows that $B$ lies in $\fS_c$.
Therefore, $B(\cH_w,\cH_v)\subseteq \fS_c$. 
We conclude that $\fS_c = B(\cK)$.
\end{proof}

\begin{example}[O'Brien \cite{Obrien}] 
Let $G$ be an aperiodic, in-degree regular, transitive, finite graph that contains a loop (i.e.\ an edge $e$ with $r(e)=s(e)=v$). 
We can explicitly construct a synchronizing word for this graph. 

We  first arbitrarily colour the loop $e$ at the vertex $v$ with colour $c_1$. 
Since $G$ is transitive, we can find a directed spanning tree starting at $v$. 
We colour all edges in this spanning tree by $1$ as well. 
Notice that for each vertex $w\neq v$, there is precisely one edge in the spanning tree that points into $w$. 
Therefore, we can colour the remaining edges to obtain a strong edge colouring. 
Suppose the depth of the spanning tree is $k$. 
Then the word $\gamma=1^k$ is a synchronizing word. 

Indeed, for any $w\in V$, we can find a path $\mu$ along the spanning tree with $s(\mu)=v$ and $r(\mu)=w$. 
Let the length of $\mu$ be $m\leq k$, and set $\mu'=\mu e^{k-m}$ be the path that begins with the loop $k-m$ times followed by $\mu$. 
Then the length of $\mu'$ is $k$ with $s(\mu')=v$ and $r(\mu')=w$. 
All edges in $\mu'$ are coloured by $1$. 
Therefore, $c(\mu')=\gamma$. 
Hence, $\gamma$ is a synchronizing word. 
\end{example}

When the graph is periodic, it is impossible to have a synchronizing word in the sense of the Definition \ref{D:SyncWord} due to periodicity. However, there are some generalizations of Trahtman's result to the periodic graphs \cite{BealPerrin} when one relaxes the definition of a synchronizing word. 

\begin{question} If $G$ is a transitive, periodic, in-degree $d$-regular finite graph with $d\geq 3$, is there a free semigroupoid algebra on $G$ which is von Neumann type? in particular $B(\cH)$? \footnote{This question had recently been resolved positively. In fact, in \cite{DL+} it is shown that every non-cycle finite transitive graph has $B(\cH)$ as a free semigroupoid algebra.}
\end{question} 

A technical question which we are unable to answer could resolve this: is there a family of Cuntz isometries $W_1,\cdots,W_d$ such that the \wot-closed algebra generated by $\{W_\mu: |\mu|=m, \mu\in\bF_d^+\}$ is $B(\cH)$ for every $m\ge1$? 
We know from \cite[Theorem 1.7]{D2006} that the answer is affirmative when $d=2$. 
As a result, we are able to extend Theorem \ref{T:Read ex} to the periodic case with in-degree $2$. 
If the answer to this question is affirmative for all $d\ge2$, one would be able to extend Theorem \ref{T:Read ex} for every in-degree regular, transitive, finite directed graph.

\begin{example}
The assumption of in-degree regular is essential in our construction as it allows us to convert our problem to edge colouring. 
However, this assumption is certainly not necessary for the free semigroupoid algebra to be self-adjoint. 
Consider a graph $G$ on $n$ vertices $v_1,\cdots,v_n$. 
Let $e_i$ be edges with $s(e_i)=v_i$ and $r(e_i)=v_{i+1}$ for all $1\leq i\leq n-1$. 
Let $f_1,f_2$ be two edges with $s(f_j)=v_n$ and $r(f_j)=v_1$. 
This graph is a periodic, transitive directed graph that is not in-degree regular. 
Define $S(e_i)=I$ and $S(f_j)=Z_j$ where $Z_1,Z_2$ are Cuntz isometries as in Read's example. 
It is not hard to check that $S$ is a CK family for $G$, and that the free semigroupoid algebra generated by $S$ is $B(\cK)$. 
\end{example}

\section{Atomic representations} \label{S:atomic}

In this section, we classify a set of interesting examples of free semigroupoid algebras
which generalize the atomic representations of the Cuntz-Toeplitz algebra studied in \cite[Section 3]{DP1999}.
These correspond to the case where $G$ consists of a single vertex with several edges. We shall prove a similar result for atomic free semigroupoid algebras. 

\begin{definition} \label{D:atomic}
Let $G$ be a directed graph. 
We call a TCK family $\rS=(S_v,S_e)$ \textit{atomic} if there exist pairwise disjoint sets $\{\Lambda_v\}_{v\in V}$, an orthonormal basis $\{\xi_{v,i}: v\in V,\  i\in \Lambda_v\}$ of $\cH$ and injections $\pi_e : \Lambda_{s(e)} \rightarrow \Lambda_{r(e)}$ such that   
\begin{enumerate} [label=\normalfont{(\arabic*)}]
\item For each vertex $v\in V$, $S_v$ is the projection onto $\overline{\spn}\{\xi_{v,i} :  i\in\Lambda_v \}$.
\item For each $e\in E$ and $i\in \Lambda_{s(e)}$, there are scalars $\lambda_{e,i} \in \bT$ so that $S_e\xi_{s(e),i}=\lambda_{e,i}\xi_{r(e),\pi_e(i)}$, and $S_e$ maps all other basis vectors to $0$. 
\item The ranges of $\{ \pi_{e} : e \in E, \ s(e) =v \}$ are pairwise disjoint for $v\in V$.  
\end{enumerate}
\end{definition}
    
It is not hard to see that $\rS = (S_v,S_e)$ defines a TCK family for the directed graph $G$. 
Indeed, each $S_v$ is a projection and their ranges are pairwise orthogonal.
If $e_1,\cdots,e_k$ have the same range $w$, then $S_{e_i}$ have pairwise orthogonal ranges dominated by the range of $S_w$. One is encouraged to think of atomic TCK families as a ``combinatorially" defined TCK family.

We shall prove a similar result to the one in \cite[Section 3]{DP1999} showing that the atomic representations of the free semigroupoid $\mathbb{F}^+(G)$ have three types: left regular type, inductive type, and cycle type.

First of all, let us define a labeled directed graph $H$ associated with the atomic representation. 
The vertices of $H$ are precisely the set of basis vectors $\{\xi_{v,j}\}$. 
An edge points from a vertex $\xi_{v,i}$ to $\xi_{u,j}$ if there exists some $e\in E$ with $s(e)=v$ and $i\in \Lambda_v$ so that $r(e) = u$ and $\pi_e(i)=j$, and the edge is labelled $e$. Equivalently, there's an edge when $S_e\xi_{v,i}=\lambda \xi_{u,j}$ for some $\lambda \in \bT$.
For each vertex $\xi_{v,i}$ in this graph, we let the number of edges pointing into this vertex be the in-degree, denoted by $\degin(\xi_{v,i})$. Similarly, let the number of edges pointing out from this vertex be the out-degree, denoted by $\degout(\xi_{v,j})$.

We begin with a few simple observations.

{\samepage
\begin{lemma}\label{L:atomic_easy} 
Let $G$ be a directed graph and let $\rS=(S_v,S_e)$ be an atomic TCK family for $G$ as in Definition $\ref{D:atomic}$. 
Suppose that $H$ is the labeled directed graph associated with $\rS$.
\begin{enumerate}[label=\normalfont{(\arabic*)}]
\item Each connected component of $H$ corresponds to a reducing subspace.
\item We have $\degin(\xi_{v,i})\leq 1$ and $\degout(\xi_{v,i}) = \degout(v)$ for each $i\in \Lambda_v$.
\item $S_v - \sum_{e \in r^{-1}(v)}S_e S_e^*$ is the projection onto $\spn\{\xi_{v,i} : \degin(\xi_{v,i}) = 0 \}$.
\end{enumerate}
\end{lemma}
}

\begin{proof}
Identify each undirected connected component $C$ of $H$ with the span $\cH_C$ of the basis vectors which are vertices of $C$. 
This subspace is evidently invariant under each $S_e$ and $S_e^*$ for $e\in E$,
and each basis vector is an eigenvector for every $S_v$.
So $\cH_C$ is reducing for $\rS$.

As the range of each $S_e$ corresponds to the range of $\pi_e$, and these are pairwise disjoint,
the in-degree of a vertex $x_{v,i}$ is either 1 or 0 depending on whether it is in the range of some $\pi_e$. 
On the other hand, each $e\in s^{-1}(v)$ determines a partial isometry with initial space $\spn\{\xi_{v,i} : i \in \Lambda_v \}$. Therefore, $\degout(\xi_{v,i}) = \degout(v)$. Finally, the \sot-sum $\sum_{e \in r^{-1}(v)}S_e S_e^*$ has range consisting of 
\[ \spn\{\xi_{v,i} : \degin(\xi_{v,i}) = 1 \} .\]
The complement in $\ran S_v$ is evidently $\spn\{\xi_{v,i} : \degin(\xi_{v,i}) = 0 \}$.
\end{proof}

The first part of Lemma \ref{L:atomic_easy} tells us that we may restrict our attention to the case in which $H$ is connected as an undirected graph, as the general case is a direct sum of such representations.

The next proposition uses the fact that the in-degree is at most $1$ to recursively select a unique predecessor of each vertex in the graph having in-degree $1$. 
Begin at some basis vector $\xi_{v_0,i_0}$ in $H$.
If $\degin(\xi_{v_0,i_0})=0$, the procedure stops.
If $\degin(\xi_{v_0,i_0})=1$, there is a unique vertex $\xi_{v_{-1}, i_{-1}}$ mapped to $\xi_{v_0,i_0}$ in $H$ by a unique edge $e_{-1}$.
Recursively define a sequence $\xi_{v_{-n},i_{-n}}$ and edges $e_{-n}$ for $n\ge0$ by repeating this procedure.
This backward path either terminates at a vector with in-degree 0, is an infinite path with distinct vertices in $H$, or eventually repeats a vertex and hence becomes periodic. 
This establishes the following.

\begin{proposition} \label{P:atomic types}
Let $G$ be a directed graph, and let $\rS=(S_v,S_e)$ be an atomic TCK family for $G$. 
Suppose that $H$ is the labeled directed graph associated with $\rS$.
Let $\xi_{v_0,i_0}$ be a standard basis vector. 
Construct the sequence $\xi_{v_{-n},i_{-n}}$ and edges $e_{-n}$ as above. Then there are three possibilities:
\begin{enumerate}[label=\normalfont{(\arabic*)}]
\item There is an $n\ge0$ so that $\degin(\xi_{v_{-n},i_{-n}}) = 0$.

\item The sequence $\xi_{v_{-n},i_{-n}}$ consists of distinct vertices in $H$.

\item There is a smallest integer $N$ and a smallest positive integer $k \ge1$ so that $\xi_{v_{-n-k}, i_{-n-k}} = \xi_{v_{-n}, i_{-n}}$ and $e_{-n-k-1}=e_{-n-1}$ for all $n\ge N$.
\end{enumerate} 
\end{proposition}

\begin{proof} Since each $\degin(\xi_{v_{-k},i_{-k}})\in\{0,1\}$, the construction of $\xi_{v_{-n},i_{-n}}$ either stops at a vertex with in-degree $0$ (which falls into case (1)) or produce an infinite sequence of vertices $\xi_{v_{-n},i_{-n}}$ in $H$. These vertices are either all distinct (which falls into case (2)) or we can find a smallest $N$ so that $\xi_{v_{-N},i_{-N}}$ appears at least twice in the sequence. Let $k\geq 1$ be the smallest integer so that $\xi_{v_{-N},i_{-N}}=\xi_{v_{-N-k},i_{-N-k}}$. Since they are the same vertex, $\xi_{v_{-N-1},i_{-N-1}}=\xi_{v_{-N-k-1},i_{-N-k-1}}$ is the unique vertex mapped to $\xi_{v_{-N},i_{-N}}=\xi_{v_{-N-k},i_{-N-k}}$ by the edge $e_{-N-1}=e_{-N-k-1}$. By induction, $\xi_{v_{-n-k}, i_{-n-k}} = \xi_{v_{-n}, i_{-n}}$ and $e_{-n-k-1}=e_{-n-1}$ for all $n\ge N$ (which falls into case (3)).
\end{proof}

\begin{lemma} \label{L:atomic wandering}
Let $G$ be a directed graph, and let $\rS=(S_v,S_e)$ be an atomic TCK family for $G$. 
Suppose that $H$ is the labeled directed graph associated with $\rS$.
Given a basis vector $\xi = \xi_{v,i}$, if there is no non-trivial path $\mu$ in $H$ so that $S_\mu\xi=\lambda\xi$ for a scalar $\lambda \in \bT$, then $\xi$ is a wandering vector. 
Moreover, we have the unitary equivalence $\rS|_{\fS[ \xi_{v,i}]} \cong \rL_{G,v}$.
\end{lemma} 

\begin{proof} 
Each $S_\mu \xi$, for $\mu\in \bF_G^+$ with $s(\mu)=v$, is a multiple of another basis vector by a unimodular scalar.
Thus, two such vectors are either orthogonal or scalar multiples of each other. 
Suppose that there are two distinct paths $\mu, \mu'$ so that $S_\mu\xi = \lambda S_{\mu'}\xi$ for some unimodular scalar $\lambda$. 
Write $\mu = e_n \dots e_1$ and $\mu'=f_m \dots f_1$. With no loss of generality, we may suppose that $n\ge m$.
Because the in-degree of $\xi$ is at most one, we see that $e_n = f_m$. 
Therefore,  we can inductively deduce that the last $m$ edges in $\mu$  and $\mu'$ coincide. 
It follows that $\mu=\mu'\nu$.  
Applying $S_{\mu'}^*$ to the identity above, we obtain that $S_\nu \xi= \lambda\xi$. 
Since this cannot happen non-trivially by hypothesis, $\nu$ must be the trivial path $\nu=v$, and thus $\mu = \mu'$. 
Hence, $\xi$ is a wandering vector. 

It follows that $\fS[ \xi_{v,i}] = \ol{\spn}\{ S_\mu \xi_{v,i} : s(\mu)=v \}$. 
It is evident from the action of $\rS$ on this basis that the natural identification $U: \cH_{G,v} \to \fS[ \xi_{v,i}] $ 
given by $U \xi_\mu = S_\mu \xi_{v,i}$ for $s(\mu)=v$ yields the desired unitary equivalence between
$\rL_{G,v}$ and $\rS|_{\fS[ \xi_{v,i}] }$.
\end{proof}

\begin{lemma} \label{L:one cycle}
Let $G$ be a directed graph, and let $\rS=(S_v,S_e)$ be an atomic TCK family for $G$. 
Suppose that $H$ is the labeled directed graph associated with $\rS$.
If $H$ is a connected graph, then it contains at most one cycle.
\end{lemma} 

\begin{proof} 
Suppose that $H$ has a cycle $C$ and let $H_0=H[C]$ be the subgraph induced on the smallest directed set containing $C$. This is a directed subgraph of $H$ in which every vertex has in-degree 1. As the in-degree in $H$ is at most 1 at each vertex, it follows that $H_0$ is a connected component of $H$.
So $H_0=H$.

Consider a vertex of the form $\xi_{v,i} = S_\mu \xi$ which is not in the cycle.
Since the in-degree of each vertex is at most 1 and $H=H[C]$, we can follow a path backward uniquely until it enters the cycle $C$.
In particular, $\xi_{v,i}$ cannot be in another cycle (from which there is no escape along a backward path).
Therefore, the connected component containing $C$ has a unique cycle.
\end{proof}

The Wold decomposition theorem allows us to first identify representations of \textit{left regular type}.

\begin{proposition}[Left Regular Type] \label{P:lrtype} 
Let $G$ be a directed graph, and let $\fS$ be the free semigroupoid algebra generated by an atomic TCK family $\rS=(S_v,S_e)$ for $G$. Suppose that $H$ is the labeled directed graph associated with $\rS$.
If $\degin(\xi_{v,i}) = 0$, then $\xi_{v,i}$ is a wandering vector, and the connected component $C_{v,i}$ of $H$ containing $\xi_{v,i}$ determines the reducing subspace $\cH_{v,i} = \fS[\xi_{v,i}]$. Moreover, the restriction of $\rS$ to $\cH_{v,i}$ is unitarily equivalent to $\rL_{G,v}$.

Additionally, if $H$ is connected as an undirected graph and there is a vertex $\xi_{v,i}$ with $\degin(\xi_{v,i}) = 0$, then $\fS$ is analytic.
\end{proposition}

Next we consider the case of an \textit{infinite tail} (case (2) of Proposition \ref{P:atomic types}).
We synthesize the information in such a representation as follows.
 
\begin{definition} \label{D:tail type}
Let $G$ be a directed graph and $\tau = e_{-1}e_{-2}e_{-3} \dots$ be an infinite backward path in $G$. 
Let $s(e_{i-1}) = r (e_i) = v_i$ with $i \le 0$.
For each $v_i$ with $i\le0$, let $\cH_i = \cH_{G,v_i}$, and let $\rL^{(i)}=\rL_{G,v_i}$ act on $\cH_i$.
Write $L^{(i)}_{\mu}$, for $\mu\in \bF_G^+$ a path operator of this family.
There is a natural isometry $J_i:\cH_i \to \cH_{i-1}$ given by $J_i\xi_\mu = \xi_{\mu e_i}$ which satisfies
$L^{(i-1)}_{\mu} J_i = J_i L^{(i)}_{\mu}$.
Let $\cH_\tau$ be the direct limit Hilbert space $\cH_\tau  = \dirlim (\cH_i, J_i)$, which we can also consider as the completion of the union of an increasing sequence of subspaces. Let $K_i : \cH_i \to \cH_\tau$ be the natural injections.
Then we define $S_\mu$ on $\cH_\tau$ such that $S_\mu K_i = K_i L^{(i)}_{\mu}$.
This defines a TCK family $\rS_\tau$.
\end{definition}

Since $\bigoplus_{e\in E} S_e \cH_i = \cH_i \ominus \bC \xi_{v_i}$, it follows that
$\bigoplus_{e\in E} S_e \cH_\tau  = \cH_\tau$.
We deduce that $\rS_\tau$ is a non-degenerate fully coisometric CK family. In this case, we show that the free semigroupoid algebra is still analytic.
  
\begin{proposition} \label{P:S_lambda}
Let $G$ be a directed graph, and let $\tau = e_{-1}e_{-2}e_{-3} \dots$ be an infinite backward path in $G$. 
Let $G' = G[\{v_{-n}: n\ge0\}]$ where $v_i=s(e_{i-1})$ for $i\leq 0$. 
Suppose $\fS_{\tau}$ is the free semigroupoid algebra of $\rS_{\tau}$ as in Definition $\ref{D:tail type}$. 
Then $\fS_\tau$ is completely isometrically isomorphic and weak-$*$ homeomorphic to $\fL_{G'}$.
\end{proposition}

\begin{proof} 
We adopt the same notation as in Definition \ref{D:tail type}. 
For each $i \le 0$, the restriction of $\rS_\tau$ to the invariant subspace $\cH_i$ is unitarily equivalent to $\fL_{G,v_i}$.
Thus, $\fS_\tau$ is a \wot-inductive and \wot-projective limit of these algebras.
That $\fS_\tau$ is completely isometrically isomorphic and weak-$*$ homeomorphic to $\fL_{G'}$ follows from Corollary~\ref{C:lr_inductive}.
\end{proof}

The following result is now straightforward.

\begin{proposition}[Inductive Type] \label{P:inductive_type} 
Let $G$ be a directed graph and $\tau = e_{-1}e_{-2}\dots$ , be an infinite backward path in $G$. Let $\rS=(S_v,S_e)$ be an atomic TCK family for $G$ that generates a free semigroupoid algebra $\fS$. 
Suppose that the labeled directed graph $H$ associated with $\rS$ is connected as an undirected graph and that the sequence $(\xi_{v_{-n},i_{-n}})$ in Proposition $\ref{P:atomic types}$ consists of distinct vertices in $H$ with connecting labeled edges $e_i$. 
Then $\rS$ is unitarily equivalent to $\rS_\tau$, and $\fS$ is completely isometrically isomorphic and weak-$*$ homeomorphic to $\fL_{G'}$, where $G' = G[\{v_{-n}: n\ge0\}]$.
\end{proposition}

\begin{corollary}
Suppose $G$ is a directed graph and $\rS$ is an atomic TCK family for $G$. 
If $\rS$ is a direct sum of inductive or regular types, then it is analytic.
\end{corollary}

Finally, we consider the cycle type (case (3) in Proposition \ref{P:atomic types}). 
Assume that $H$ is connected and contains a cycle $C$. 
Let us denote the vertices of $C$ as $\xi_1,\dots,\xi_k$ and $e_1,\dots,e_k$ the labels of edges between them, and $\lambda_1,\dots,\lambda_k$ the scalars in $\bT$ such that
\[ S_{e_j} \xi_i = \lambda_j \xi_{j+1} \FOR 1 \le j < k \qand S_{e_k}\xi_k = \lambda_k \xi_1 .\]
Let $v_j$ be the vertices of $G$ so that $\xi_j = S_{v_j}\xi_j$. 
Then for every $e \ne e_j$ in $E$ with $s(e)=v_j$, the vector $S_e\xi_j$ is a wandering vector by Lemma~\ref{L:atomic wandering}. 
We can visualize this setup as the cycle $C$ from which various edges map each node $\xi_j$ to basis vectors which do not lie in $C$.
By Lemma~\ref{L:atomic wandering}, these vectors are wandering vectors which generate a copy of $\cH_{G,v}$ if they are in the range of $S_v$. 
Think of the cycle as having many copies of $\cH_{G,v}$, for $v\in V$, attached.
Note also that the edges $e_j$ fit together, end to end, and form a cycle in $G$. We formalize this as a definition.

\begin{definition} \label{D:cycle type}
Suppose $G$ is a directed graph, let $w = e_k\dots e_1$ be a cycle in $G$, and let $\lambda\in \bT$. Form a Hilbert space 
\[
 \cH_w = \bC^k \oplus \bigoplus_{j=1}^k \bigoplus_{\substack{f \ne e_j\\s(f)=s(e_j)}}  \cH_{G, r(f)} = \bC^k \oplus \cK_w .
\]
Choose an orthonormal basis $\xi_1,\dots,\xi_k$ for $\bC^k$, and for each $1\le j\le k$ and each $f\in E$ such that
$s(f)=s(e_j)$ and $f \ne e_j$, denote the standard basis for a copy of $\cH_{G, r(f)}$ by $\xi_{j,f,\mu}$ for $s(\mu)=r(f)$.
Define a fully coisometric CK family $\rS_{w,\lambda}$ on $\cH_w$ as follows:
\[ S_v \xi_j = \delta_{v, s(e_j)} \xi_j \qand S_v \xi_{j,f,\mu} = \delta_{v,r(\mu)} \xi_{j,f,\mu} \qfor v\in V.\]
And for $e\in E$, define
\[
 S_e \xi_j = \begin{cases}
 \xi_{j+1} &\IF e=e_j,\ 1 \le j < k\\
 \lambda \xi_1 &\IF e=e_k,\ j = k\\
 \xi_{j,e,r(e)} &\IF e \ne e_j,\ s(e)=s(e_j)\\
 0 &\IF s(e) \ne s(e_j)
\end{cases}
\]
and
\[
S_e \xi_{j,f,\mu} = \begin{cases}
 \xi_{j,f,e\mu} &\IF s(e)=r(\mu)\\
 0 &\IF s(e) \ne s(\mu).
\end{cases} 
\]
\end{definition}

\begin{proposition}[Cycle Type] \label{P:cycle_type} 
Let $G$ be a directed graph, and let $\rS=(S_v,S_e)$ an atomic TCK family for $G$. 
Suppose that the labeled directed graph $H$ associated with $\rS$ is connected as an undirected graph and contains a cycle $C$ with basis $\xi_1,\dots,\xi_k$, edges $e_1,\dots,e_k$, 
and  scalars $\lambda_1,\dots,\lambda_k$ in $\bT$ such that
\[ S_{e_j} \xi_i = \lambda_j \xi_{j+1} \FOR 1 \le j < k \qand S_{e_k}\xi_k = \lambda_k \xi_1 .\]
Define $w = e_k \dots e_1$ and $\lambda =  \prod_{j=1}^k \lambda_k$.
Then $\rS$ is unitarily equivalent to $\rS_{w,\lambda}$.
\end{proposition}

\begin{proof} 
This proposition is established by rescaling the basis. First replace the basis vectors $\xi_j$ by $\prod_{i=1}^{j-1}\lambda_i\  \xi_j$, so that now
$S_{e_j} \xi_j = \xi_{j+1}$ for $1 \le j < k$ and $S_{e_k} \xi_k = \lambda \xi_1$.
Then for each $f\in E$ with $s(f)=s(e_j)$ and $f \ne e_j$ and $\mu\in \bF_G^+$ with $s(\mu)=r(f)$,
we rescale so that each $S_{\mu f}\xi_j$ is a basis vector rather than a scalar multiple of one.
After this change, it is clear that we have a representation unitarily equivalent to $\rS_{w,\lambda}$.
\end{proof}

We note that the decomposition of any atomic representation into a direct sum of atomic representations
on undirected connected components of $H$ is canonical. However, these subrepresentations are not always irreducible.
So we now consider the decomposition of these representations into irreducible representations.
This will provide an analog of \cite[Proposition 3.10]{DP1999} for the case of free semigroup algebras. The proof itself follows the lines of the free semigroup case in \cite{DP1999} quite closely, and so we omit the details.

Combining Proposition~\ref{P:atomic types} with the detailed information about each case yields the following.

\begin{theorem}\label{T:atomic}
Let $G$ be a directed graph and let $\rS=(S_v,S_e)$ be an atomic TCK family for $G$.
Then $\rS$ decomposes as a direct sum of representations of left regular type, inductive type and cycle type.
\end{theorem}

Next we need to identify which of these atomic representations are irreducible, and when two are unitarily equivalent.

\begin{definition}
Let $G$ be a directed graph. A cycle $w=e_k\dots e_1 \in \bF_G^+$ is called \textit{primitive} if it is not a power of a strictly smaller word. 
Alternatively, $w$ is not primitive when there is a path $u$ and a positive integer $p\ge2$ so that $w=u^p$.
\end{definition}

\begin{definition}
Two infinite paths $\tau= e_{-1}e_{-2}e_{-3} \dots$  and $\tau' = e'_{-1}e'_{-2}e'_{-3} \dots$ are \textit{shift tail equivalent} if there are integers $N \leq 0$ and $k \in \bN$ so that $e'_{m+k} = e_m$ for all $m \le N$. 
An infinite path is \textit{eventually periodic} if it is shift tail equivalent to a periodic path $\tau$, in the sense that $e_{m-k}=e_m$ for all $m\le0$. 
If $k$ is chosen to be minimal, so that the cycle $u=e_{-1}\dots e_{-k}$ is primitive in $\tau$, then $\tau$ has period $k$ and we write $\tau= u^\infty$.
\end{definition}

\begin{theorem} \label{T:atomic_irred}\ 
Let $G$ be a directed graph. Consider the atomic representations of  left regular type, inductive type and cycle type.
\begin{enumerate}[label=\normalfont{(\arabic*)}]
\item The left regular TCK family $\rL_{G,v}$ is irreducible for every $v\in G$. 
Two irreducible representations of left regular type are unitarily equivalent if and only if they arise from the same vertex.

\item Let $w=e_k\dots e_1$ be a cycle in $G$, and let $\lambda\in\bT$.
Then $\rS_{w,\lambda}$ is irreducible if and only if $w$ is primitive.
Two irreducible cycle representations $\rS_{w,\lambda}$ and $\rS_{u,\mu}$ are unitarily equivalent if and only if $u$ is a cyclic permutation of $w$ and $\mu=\lambda$.
If $w=u^p$ for some primitive cycle $u$, then letting $\theta_j$ denote the $p$th roots of $\lambda$, we have a decomposition into irreducible families
\[ \rS_{w,\lambda} \cong \sum_{j=1}^p \oplus \rS_{u,\theta_j} .\]

\item  Let $\tau = e_{-1}e_{-2}e_{-3} \dots$ be an infinite backward path. 
Two inductive type atomic representations are unitarily equivalent if and only if the two infinite paths are shift tail equivalent. 
Moreover, $\tau$ is not eventually periodic if and only if $\rS_\tau$ is irreducible, and when it is eventually periodic, say $\tau$ is shift tail equivalent to $u^\infty$ for a primitive cycle $u=e_{-1}\dots e_{-k}$, then
\[ \rS_\tau  \cong \int_\bT^\oplus \rS_{u,\lambda} \,d\lambda . \]
\end{enumerate}
\end{theorem}

\begin{remark}
To decide if two atomic representations are unitarily equivalent, first decompose the graph $H$ into connected components, which induces a decomposition of the representations into a direct sum of left regular, infinite tail and cycle types. 
Then split each cycle representation into irreducible ones. 
Next, count the multiplicities of each representation to decide unitary equivalence. 
It is not necessary to decompose periodic infinite tails into a direct integral as with two cycle representations, an infinite tail representation up to shift-tail equivalence either appears, or does not appear in the direct sum decomposition. 
In fact, the direct integral over a proper subset $A \subset \bT$ does not yield an atomic representation.
\end{remark}

Finally, we explain the structure of $\fS_{w,\lambda}$ for the cycle type representations.

\begin{proposition} \label{P:structure cycle}
Let $G$ be a directed graph, $w = e_k\dots e_1$ a primitive cycle in $G$ and $\lambda \in \bT$ and $s(e_j) = v_j$. 
Suppose that $\rS_{w,\lambda}$ is given as in Definition $\ref{D:cycle type}$. 
Then with respect to the decomposition $\cH_w = \bC^k \oplus \cK_w$, we obtain
\[ \fS_{w,\lambda} \cong  
\begin{bmatrix}
 \cM_k & 0 \\
 \cB(\bC^k, \cK_w) & \fL 
\end{bmatrix}
\]
Where $\fL$ is the free semigroupoid algebra generated by $\bigoplus_{v\in V} \rL_{G,v}^{(\alpha_v)}$ such that $\alpha_v = \sum_{j=1}^{k} \big| \{f \in E : s(f) = v_j, r(f) = v \AND f \ne e_j \} \big|$. Thus, the $2,2$ corner is isomorphic to $\fL_{G'}$ with $G'=G[\{ v \in V : \alpha_v \ne 0 \}]$. 
In fact, the projection of $\cH_w$ onto $\bC^k$ is the structure projection for $\fS_{w,\lambda}$, 
and $\fS_{w,\lambda}$ is of dilation type.
\end{proposition}

\begin{proof}
Notice that when $w$ is primitive, $\ol{\lambda}S_w \xi_1 = \xi_1$.
If we begin with any $\xi_j$ for $2 \le j \le k$, the fact that $w$ does not coincide with the cyclic permutation of $w$ which maps $\xi_j$ to $\lambda\xi_j$ means that $\ol{\lambda}S_w\xi_j$ is either $0$ or a multiple of a basis vector which is not in the cycle.
Hence, by Lemma~\ref{L:atomic wandering}, every standard basis vector except for $\xi_1$ is mapped by $\ol{\lambda}S_w$ to a wandering vector.
Therefore, higher powers of this word map each basis vector to a sequence of zeroes and basis vectors which converges weakly to 0.
Consequently,
\[ \wotlim_{n \to\infty} ( \ol{\lambda}S_w )^k = \xi_1\xi_1^*  \]
belongs to $\fS$. Thus so does $\xi_j\xi_1^* = S_{e_{j-1}\dots e_1}\xi_1\xi_1^*$ for $1 \le j \le k$.
Similarly, $\xi_j \xi_i^*$ belongs to $\fS_{w,\lambda}$ for $1 \le i,j \le k$, and these operators span $\cM_k = \cB(\bC^k)$.
Each $\xi_i$ is a cyclic vector for $\rS_{w,\lambda}$, and from this we can deduce that $\xi_{j,f,\mu}\xi_i^*$ belongs to $\fS_{w,\lambda}$;
whence $\fS_{w,\lambda}$ contains $\cB(\bC^k, \cK_w)$.
Also, $\cK_w$ is invariant, so the $1,2$ entry of the operator matrix is $0$.

So now we may consider the restriction of $\rS_{w,\lambda}$ to $\cK_w$.
Each $\xi_{j,f,r(f)}$ is a wandering vector, and it generates a subspace canonically identified with $\cH_{G,r(f)}$.
From Wold decomposition (Theorem \ref{thm:Wold-decomp}), the restriction of $\rS_{w,\lambda}$ to $\cK_w$ is a direct sum of copies of $\rL_{G,v}$ where the multiplicities are determined by the dimension of the corresponding wandering space.
These multiplicities are precisely the cardinalities of the sets of $\xi_{j,f,r(f)}$ for which $r(f)=v_j$, $f\neq e_j$ and $s(f)=v$ for some $j$. 
That is, the multiplicity at $v$ is exactly $\alpha_v$.

Finally, it follows from Theorem~\ref{T:left_quotient} that the free semigroupoid algebra $\fL$ on $\cK_w$ is completely isometrically isomorphic and weak-$*$ homeomorphic to $\fL_{G'}$ where $G'= G[\{ v \in V : \alpha_v \ne 0 \}]$. The rest follows from the Structure Theorem \ref{thm:structure}.
\end{proof}

\bibliographystyle{amsalpha}

\end{document}